\documentclass[11pt]{amsart}

\usepackage{amssymb}
\usepackage[all,cmtip]{xy}
\usepackage{mathrsfs}
\usepackage[foot]{amsaddr}

\newtheorem{theorem}{Theorem}[section]
\newtheorem{lem}[theorem]{Lemma}
\newtheorem{cor}[theorem]{Corollary}
\newtheorem{prop}[theorem]{Proposition}

\theoremstyle{definition}
\newtheorem{defi}[theorem]{Definition}
\newtheorem{example}[theorem]{Example}

\theoremstyle{remark}
\newtheorem{rem}[theorem]{Remark}

\newcommand{\BC}{\mathbb{C}}            
\newcommand{\BZ}{\mathbb{Z}}             
\newcommand{\BN}{\mathbb{N}}            
\newcommand{\dstirling}[2]{\genfrac{[}{]}{0pt}{0}{#1}{#2}}   

\newcommand{\Gaff}{\hat{\mathfrak{g}}}        
\newcommand{\Glie}{\mathfrak{g}}             
\newcommand{\Hlie}{\mathfrak{h}}          

\newcommand{\qaf}{U_q(L\mathfrak{g})}    
\newcommand{\Borel}{U_q(\mathfrak{b})}   
\newcommand{\lBorel}{U_q(\mathfrak{c})}  
\newcommand{\CU}{\mathcal{U}}            
\newcommand{\BCU}{\mathcal{U}^{\mathrm{ab}}}   
\newcommand{\BY}{Y^{\mathrm{ab}}}    
\newcommand{\CR}{\mathcal{R}}    
\newcommand{\barR}{\overline{\mathcal{R}}}    
\newcommand{\stimes}{\ \widetilde{\otimes}\ }       

\newcommand{\BQ}{\mathbf{Q}}                
\newcommand{\BP}{\mathbf{P}}            
\newcommand{\wt}{\mathrm{wt}}         
\newcommand{\CL}{\mathcal{L}}       

\newcommand{\Be}{\mathbf{e}}         
\newcommand{\Bf}{\mathbf{f}}  
\newcommand{\Bn}{\mathbf{n}}
\newcommand{\Bm}{\mathbf{m}}
\newcommand{\Bp}{\mathbf{p}}

\allowdisplaybreaks                
\title{Theta series for quantum loop algebras and Yangians}
\author{Huafeng Zhang}
\address{Univ. Lille, CNRS, UMR 8524 - Laboratoire Paul Painlev\'{e}, F-59000 Lille, France}
\email{huafeng.zhang@univ-lille.fr}
\begin{document}
\begin{abstract}
We introduce and study a family of power series, which we call {\it Theta series}, whose coefficients are in the tensor square of a quantum loop algebra. They arise from a coproduct factorization of the T-series of Frenkel--Hernandez, which are leading terms of transfer matrices of certain infinite-dimensional irreducible modules over the upper Borel subalgebra in the category $\mathcal{O}$ of Hernandez--Jimbo. We prove that each weight component of a Theta series is polynomial. As applications, we establish a decomposition formula and a polynomiality result for R-matrices between an irreducible module and a finite-dimensional irreducible module in category $\mathcal{O}$.

We extend T-series and Theta series to Yangians by solving difference equations determined by the truncation series of Gerasimov--Kharchev--Lebedev--Oblezin. We prove polynomiality of Theta series by interpreting them as associators for triple tensor product modules over shifted Yangians.

\bigskip 


\end{abstract}

\maketitle
\tableofcontents

\section{Introduction}
Fix $\Glie$ to be a complex finite-dimensional simple Lie algebra and $q$ to be a nonzero complex number which is not a root of unity. Consider the quantum loop algebra $\qaf$ and the Yangian $Y(\Glie)$. These are Hopf algebra deformations of the universal enveloping algebras of the loop algebra $\Glie[z,z^{-1}]$ and of the current algebra $\Glie[z]$ respectively. In this paper, we establish a polynomiality property for the coproduct of these Hopf algebras and deduce from it various polynomiality results at the level of representations.

Our main results are parallel for these two algebras, but their proofs are of different nature and require auxiliary algebras.  In the loop case, $\qaf$ being a quantum group in the sense of Drinfeld--Jimbo, the key tools are the upper Borel subalgebra $\Borel$ and universal R-matrix $\CR(z)$. In the current case, we make central use of the so-called shifted Yangians and their shifted coproduct.

\medskip

Shifted Yangians are a family of algebras $Y_{\mu}(\Glie)$ parameterized by integral coweights $\mu$ of the underlying simple Lie algebra $\Glie$. They appeared first for type A in the work of Brundan--Kleshchev \cite{BK} on
the representation theory of finite $W$-algebras.
For general types they were introduced by Kamnitzer--Webster--Weekes--Yacobi \cite{KWWY} in their study of affine Grassmannian slices 
and by Braverman--Finkelberg--Nakajima \cite{BFN} in their study of Coulomb branches of 3d $\mathcal{N} = 4$ supersymmetric quiver gauge 
theories. Shifted quantum affine algebras $\CU_{\mu}(\Gaff)$ were defined by Finkelberg--Tsymbaliuk \cite{FT} in the study of K-theoretic Coulomb branches and again depend on a coweight $\mu$.

The defining relations of the shifted Yangian $Y_{\mu}(\Glie)$ and shifted quantum affine algebra $\CU_{\mu}(\Gaff)$ are the same as those of the ordinary Yangian and  quantum loop algebra in Drinfeld's new realization \cite{Dr}, except that the initial conditions for the Drinfeld--Cartan generators are modified according to the coweight $\mu$. In particular, $Y_0(\Glie) = Y(\Glie)$ and $\CU_0(\Gaff)$ is a central extension of $\qaf$. 

Finkelberg--Kamnitzer--Pham--Rybnikov--Weekes \cite{coproduct} extended the coproduct of the ordinary Yangian to a family of algebra homomorphisms, called shifted coproduct
$$\Delta_{\mu,\nu}: Y_{\mu+\nu}(\Glie) \longrightarrow Y_{\mu}(\Glie) \otimes Y_{\nu}(\Glie) \quad \textrm{for $\mu$ and $\nu$ coweights.} $$
Such a shifted coproduct is known to exist in type A for shifted quantum affine algebras by Finkelberg--Tsymbaliuk \cite{FT} and in other types it remains conjectural.

In \cite{BK,KTWWY} a category $\mathcal{O}^{sh}$ was defined as the sum over all coweights of category $\mathcal{O}_{\mu}$ of $Y_{\mu}(\Glie)$-modules. Its irreducible modules are classified by their highest weights, which are tuples of ratios of monic polynomials in $z$, one for each Dynkin node of $\Glie$. The shifted coproduct equips category $\mathcal{O}^{sh}$ with a tensor product bifunctor.  

For representations of shifted quantum affine algebras category $\hat{\mathcal{O}}^{sh}$ is defined in the same way. By embedding the upper Borel subalgebra $\Borel$ in antidominantly shifted quantum affine algebras, Hernandez \cite{H} obtained close relations between category $\hat{\mathcal{O}}^{sh}$ and the category $\mathcal{O}$ of $\Borel$-modules introduced by Hernandez--Jimbo \cite{HJ}. Irreducible modules in both categories have the same highest weight parametrization: tuples of ratios of polynomials in $z$ with nonzero constant term.

Of particular importance in these categories are {\it negative prefundamental modules}, whose highest weights are tuples of inverses of polynomials. They are limits of finite-dimensional modules \cite{HJ,Z2} and admit geometric realizations \cite{VV}.

\medskip

\noindent {\bf Motivation: associators for shifted Yangians.} 
Let $M$ be a $Y_{\mu}(\Glie)$-module, $K$ be a $Y_{\zeta}(\Glie)$-module and $N$ be a $Y_{\nu}(\Glie)$-module, not necessarily in category $\mathcal{O}^{sh}$. The shifted coproduct being non co-associative, it equips the same triple tensor product with two different $Y_{\mu+\zeta+\nu}(\Glie)$-module structures: $(M \otimes K) \otimes N$ and $M \otimes (K \otimes N)$. We are interested in module isomorphisms between them, called associators. 

As our main example, assume $K$ one-dimensional and $N$ negative prefundamental. 
In \cite{HZ} for any highest weight irreducible module $M$ over a shifted Yangian we constructed nonzero module morphisms, called R-matrices 
$$ \check{R}_{M,K}: M \otimes K \longrightarrow K \otimes M,\quad \check{R}_{M,N}: M \otimes N \longrightarrow N \otimes M. $$ 
We aim to produce a new R-matrix between $M$ and $N\otimes K$. The reason is that $N \otimes K$ can often be defined over the ordinary Yangian but this is seldom the case for $N$ and $K$. One way is to complete the following chain of module morphisms:
\begin{gather*}  
 \xymatrixcolsep{5pc} \xymatrix{
M \otimes (N \otimes K) \ar@{-->}[r]^?  & (M \otimes N) \otimes K \ar[r]^{\check{R}_{M,N} \otimes \mathrm{Id}_K} &  (N\otimes M) \otimes K  \ar@{-->}[d]^? \\
 (N \otimes K) \otimes M   &     N \otimes (K \otimes M)  \ar@{-->}[l]_?        & N \otimes (M \otimes K) \ar[l]_{\mathrm{Id}_N \otimes \check{R}_{M,K}}
 } 
\end{gather*}
The three morphisms represented by dashed arrows are examples of associators.

\medskip

\noindent {\bf Motivation: R-matrices for asymptotic representations.} 
Let us take $M$ to be a finite-dimensional module over the quantum loop algebra. On the space $M[[z]]$ there is a mutually commuting family of operators, called transfer matrices, defined by properly evaluating the universal R-matrix $\CR(z)$ of the quantum loop algebra:
$$ (\mathrm{trace}_N \otimes \mathrm{Id})(\CR_{N, M}(z)) \quad \textrm{for finite-dimensional $\qaf$-modules $N$}.  $$
This is a quantum integrable model generalizing the XXZ-model. It was conjectured by Frenkel--Reshetikhin \cite{FR1} that the eigenvalues of the above transfer matrices are described in terms of certain polynomials, generalizing Baxter polynomials of the XXZ-model. Frenkel--Hernandez \cite{FH} solved this conjecture by realizing Baxter polynomials as eigenvalues of transfer matrices associated to {\it positive prefundamental representations} of $\Borel$ in category $\mathcal{O}$, namely, irreducible modules whose highest weights are tuples of polynomials. We note that these are honest representations of the upper Borel subalgebra and they do not extend to the full quantum loop algebra.

The above transfer matrix construction applied to the universal R-matrix of the Yangian \cite{Dr1} defines a quantum integrable model generalizing the XXX-model. However, there is no analog of Borel subalgebras for the Yangian. It is desirable to have a characterization of Baxter polynomials using representations of the quantum loop algebra and of the Yangian directly. Such was the motivation of introducing the so-called asymptotic representations for $\qaf$ in \cite{Z1} and for $Y(\Glie)$ in \cite{Z2}.

In some particular cases, mostly for $\Glie$ of type A, the asymptotic representations were shown \cite{FZ,Z3} to play the same role as the positive prefundamental representations did in \cite{FH}. We expect this to be true in general. One important step is to establish polynomiality results for the universal R-matrix evaluated at an asymptotic representation and a finite-dimensional representation.

\medskip

To present the main results of the paper, we need the notion of deformed module $M_z$ associated to a module $M$ over a shifted Yangian or over a shifted quantum affine algebra, with $z$ the formal spectral parameter coming from the $\BC$-action on the shifted Yangian or the $\BC^{\times}$-action on the shifted quantum affine algebra. Its underlying space is polynomial, $M[z]$ in the current case and $M[z,z^{-1}]$ in the loop case. 

\medskip

\noindent {\bf Trivial modules.} Call a module over a shifted Yangian or a shifted quantum affine algebra {\it trivial} if the Drinfeld generators $x^{\pm}$ act as zero. The key examples include one-dimensional modules and their deformed modules. As our first observation:
\begin{itemize}
\item[(i)] Theorem \ref{thm: trivial associativity Yangian}. Let $M, K$ and $N$ be modules over three shifted Yangians. If either $M$ or $N$ is trivial, then the identity map is a module isomorphism from $(M \otimes K)\otimes N$ to $M \otimes (K\otimes N)$.
\end{itemize}
While the shifted coproduct is unavailable for shifted quantum affine algebras, there is the Drinfeld formal coproduct \cite{FT}. Given a module $M$ over $\CU_{\mu}(\Gaff)$ and another module $N$ over $\CU_{\nu}(\Gaff)$, assuming one of them is trivial, via the formal coproduct we can equip $M \otimes N$ with a module structure over $\CU_{\mu+\nu}(\Gaff)$.

It turns out that the asymptotic representations in \cite{Z1, Z2} are tensor products of negative prefundamental modules with one-dimensional modules. The proofs in both cases are similar, based on explicit realizations of tensor products with trivial modules.

\medskip

\noindent {\bf T-series.} Given a one-dimensional module $K$ over a shifted Yangian and a coweight $\mu$, by solving formally the additive difference equation for one-dimensional R-matrices \cite{HZ}, we obtain an invertible series $T_K^{\mu}(z)$ in $z$ with coefficients in the Drinfeld--Cartan subalgebra of $Y_{\mu}(\Glie)$. It satisfies an intertwining property; recall that shifted Yangians are graded by the root lattice of the simple Lie algebra $\Glie$.
\begin{itemize}
\item[(ii)] Theorem \ref{thm: Yangian T-series}. For $M$ a graded $Y_{\mu}(\Glie)$-module, the action of $T_K^{\mu}(z)$ on $M$ induces a module morphism from $K_z \otimes M$ to a suitable completion of $M \otimes K_z$.
\end{itemize}
If $M$ is an irreducible module in category $\mathcal{O}^{sh}$, then after normalization $T_K^{\mu}(z)$ recovers the R-matrix in our previous work \cite{HZ}. For $\mu = 0$, the T-series can be seen as a universal version of the abelianized transfer operators of Gautam--Wendlandt \cite{GW}. 

For shifted quantum affine algebras, the formal solution of the analogous difference equation was given by Hernandez \cite{H}. Upon an identification of Drinfeld--Cartan subalgebras for shifted quantum affine algebras and the quantum loop algebra, it is the T-series of Frenkel--Hernandez \cite{FH} defined as a limit of the transfer matrix of a positive prefundamental representation over the upper Borel subalgebra. 

\medskip

\noindent {\bf Theta series.} Given a one-dimensional module $K$ over a shifted Yangian and two coweights $\mu$ and $\nu$, as the central construction of the paper, we factorize the shifted coproduct of T-series in a particular form (Definition \ref{def: Theta series Yangian}):
\begin{equation*} 
\Delta_{\mu,\nu}(T_K^{\mu+\nu}(z)) = (1 \otimes T_K^{\nu}(z)) \times \Theta_K^{\mu,\nu}(z) \times (T_K^{\mu}(z) \otimes 1).
\end{equation*}
The middle term at the right-hand side is what we call {\it Theta series}; it is a formal series with coefficients in $Y_{\mu}(\Glie) \otimes Y_{\nu}(\Glie)$. Let $\BQ_+$ be positive root cone of $\Glie$ which is defined as the $\BN$-span of the simple roots. As a main result of the paper:
\begin{itemize}
\item[(iii)] Theorem \ref{thm: Yangian poly Theta}. The Theta series $\Theta_K^{\mu,\nu}(z)$ is an infinite sum, over $\beta \in \BQ_+$, of {\it polynomials} in $z$ with coefficients in $Y_{\mu}(\Glie)_{-\beta} \otimes Y_{\nu}(\Glie)_{\beta}$.
\end{itemize}
It is a consequence of the intertwining property below applied to regular representations:
\begin{itemize}
\item[(iv)] Theorem \ref{thm:Yangian associator}. Let $M$ and $N$ be graded modules over the shifted Yangians $Y_{\mu}(\Glie)$ and $Y_{\nu}(\Glie)$ respectively. The action of $\Theta_K^{\mu,\nu}(z)$ on $M \otimes N$ induces a module morphism from $(M \otimes K_z) \otimes N$ to a suitable completion of $M \otimes (K_z \otimes N)$.
\end{itemize}
In turn (iii) implies that if both modules are in $\mathcal{O}^{sh}$ then the Theta series restricts to a module isomorphism of ordinary triple tensor products, namely, an associator.

For $K$ a one-dimensional module over a shifted quantum affine algebra, we perform the same factorization on the Drinfeld--Jimbo coproduct $\Delta(T_K(z))$ of the ordinary T-series. The resulting Theta series $\Theta_K(z)$ is a power series in $z$ with coefficients in the tensor square of the quantum loop algebra. Parallel to the Yangian situation (iii):
\begin{itemize}
\item[(v)] Theorem \ref{thm: polynomiality Theta}. The Theta series $\Theta_K(z)$ is an infinite sum, over $\beta \in \BQ_+$, of polynomials in $z$ with coefficients in $\qaf_{-\beta} \otimes \qaf_{\beta}$.
\end{itemize}
The proof of (iii) does not work for (v), because the triple tensor products of (iv) make no sense for shifted quantum affine algebras. Instead, we derive (v) from a stronger polynomiality property of the positive prefundamental representations of the upper Borel subalgebra \cite{FH,FJMM} and the quasi-triangularity of the quantum loop algebra.

\medskip

\noindent {\bf Decomposition of R-matrices.} Given two irreducible modules $M$ and $N$ over shifted Yangians such that $M$ is finite-dimensional and $N$ is negative prefundamental, by \cite{HZ} we have a natural R-matrix $\check{R}_{M,N}(z)$ from $M_z \otimes N$ to $N\otimes M_z$. We are able to obtain a new R-matrix after tensoring $N$ with the deformed module $K_w$ of a one-dimensional module. This tensor product module includes many examples of representations of the ordinary Yangian, in contrast $N$ is rarely such a representation.
\begin{itemize}
\item[(vi)] Theorem \ref{thm: T Theta R Yangian}. There is a module morphism from $M_z \otimes (N \otimes K_w)$ to $(N\otimes K_w) \otimes M_z$ which decomposes into the known R-matrix $\check{R}_{M,N}(z)$, a T-series acting on $M$ and a Theta series acting on $M\otimes N$.
\end{itemize}
The proof of (vi) combines the module morphisms of (i), (ii) and (iv).

Let $N$ and $K$ be modules over shifted quantum affine algebras such that $K$ is one-dimensional and both $N$ and $N\otimes K_w$ can be restricted to the upper Borel subalgebra. We have a similar decomposition formula for evaluations of the universal R-matrix:
\begin{itemize}
\item[(vii)]  Theorem \ref{thm: decomposition R-matrices}. For $M$ be a finite-dimensional $\qaf$-module, the evaluation of the universal R-matrix $\CR_{N\otimes K_w,M}(z)$ decomposes into a simpler evaluation $\CR_{N,M}(z)$, a T-series acting on $M$ and a Theta series acting on $N\otimes M$.
\end{itemize}
In constract to (vi), to prove (vii) we make crucial use of Damiani's factorization formula of the universal R-matrix of the quantum loop algebra \cite{Damiani}.

\medskip

\noindent {\bf Polynomiality of R-matrices.} 
Concerning the evaluation of the universal R-matrix for irreducible modules in category $\mathcal{O}$, we have a general polynomiality result:
\begin{itemize}
\item[(viii)] Theorem \ref{thm: poly R Borel}. Let $N$ be an irreducible $\Borel$-module in category $\mathcal{O}$ and $M$ a finite-dimensional irreducible $\qaf$-module. Divided by explicit constant series both $\CR_{N, M}(z)$ and its inverse stabilize the polynomial space $N\otimes M[z]$.
\end{itemize}
Previously the polynomiality of $\CR_{N,M}(z)$ was known for $N$ positive prefundamental, that is, whose highest weight is a tuple of polynomials \cite{FH,FJMM}. 

\medskip

\noindent {\bf Perspectives.} In the Yangian case, (vi) gives an intertwining operator for tensor products of an asymptotic representation and a finite-dimensional irreducible representation. Inspired by its counterpart (vii),  it is a first question to identify such an intertwining operator with an evaluation of the universal R-matrix. When this is the case, we should be able to adapt the construction of Baxter operators \cite{FZ,Z3} to Yangian integrable models and obtain generalized Baxter polynomials for these models. 

Let $V$ and $W$ be irreducible modules over the Yangian of rational highest weights such that $V$ is finite-dimensional. Upon identification of intertwining operators with evaluations of the universal R-matrix, the decomposition formula of (vi) suggests that after renormalizations the denominator of the R-matrix $\CR_{V,W}(z)$, as a vector-valued rational function of the spectral parameter $z$, divides the denominator of a T-series (depending on $W$) acting on the finite-dimensional module $V$. This has been observed by Gautam--Wendlandt \cite[\S 7.9]{GW}. Indeed, they conjectured furthermore that for special choices of $V$ and $W$ the two denominators should coincide. 

In the quantum loop algebra case, we expect our polynomiality result (viii) to have applications in the conjecture of Hernandez \cite{H} on
highest weight classification of irreducible representations of truncated shifted quantum affine algebras, and in the conjecture of Frenkel--Hernandez \cite{FH0} on polynomiality of transfer matrices associated to more general representations in category $\mathcal{O}$. 
The two different R-matrices in (viii) might be related to each other by a duality functor of category $\mathcal{O}$ constructed recently by Pinet \cite{Pinet}. Hernandez \cite{H1} developed an approach to R-matrices of category $\mathcal{O}$ using algebraic versions of Maulik--Okounkov
stable maps. It is interesting to interact our Theta series with the uni-triangularity property of algebraic stable maps.

 The similarities of definitions and properties of Theta series should be best explained in the framework of Gautam--Toledano Laredo relating the Yangian and the quantum loop algebra at the level of algebras \cite{GTL0} and at the level of representations \cite{GTL1}. One main obstruction at the moment is the non-compatibility of the functor of \cite{GTL2} with the Drinfeld-Jimbo coproduct. 

At last, we ask if there is a compact formula for Theta series. For $\Glie = sl_2$, the Yangian Theta series is the exponential of the constant $x_{1,0}^- \otimes x_{1,0}^+$, and the quantum loop Theta series is a $q$-exponential of the polynomial $x_{1,0}^- \otimes x_{1,-1}^+z$. In higher ranks, we expect both Theta series to be ordered products of exponentials and $q$-exponentials, as in the case of the universal R-matrix of the finite type quantum group $U_q(\Glie)$.

\medskip

The main body of the paper is divided into two independent parts, Sections \ref{sec: Yangian}--\ref{sec: R Yangian} on shifted Yangians and Sections \ref{sec: quantum}--\ref{sec: type A} on the quantum loop algebra. In the two appendices we include several technical results that we find interesting on their own right.

Section \ref{sec: Yangian} collects basic properties of shifted Yangians, highest weight representations and polynomial R-matrices constructed in a previous work \cite{HZ}. 

In Section \ref{sec: trivial} we prove associativity for triple tensor product modules over shifted Yangians in which the first or the third tensor factor is a trivial module. 

In Section \ref{sec: T Yangian} we first construct S-series by solving a universal difference equation; it is an invertible power series with coefficients in a shifted Yangian. Then by a dressing procedure we obtain T-operators acting on graded representations of shifted Yangians, which are recognized as R-matrices for one-dimensional modules. 

In Section \ref{sec: Theta Yangian} by factorizing the shifted coproduct of S-series we define Theta series. We prove that Theta series induce associators for completed triple tensor products and deduce from it the polynomiality of weight components of Theta series. We also show that Theta series are multiplicative with respect to one-dimensional modules.

In Section \ref{sec: R Yangian} we establish polynomiality of T-operators for a large class of representations of shifted Yangians and prove a decomposition formula for R-matrices of tensor products with one-dimensional modules. As an application, we compute a particular diagonal entry of R-matrices.

Sections \ref{sec: quantum}--\ref{sec: Borel} review well-known facts on shifted quantum affine algebras, the upper Borel subalgebra and the universal R-matrix, including in particular the root vectors and the monodromy matrix construction. 

In Section \ref{sec: Theta} we define Theta series for the quantum loop algebra by factorizing the coproduct of T-series of Frenkel--Hernandez. We express Theta series in terms of entries of monodromy matrices associated to positive prefundamental representations in category $\mathcal{O}$ and establish polynomiality for the monodromy matrices.

In Section \ref{sec: R dec} we prove a decomposition formula for R-matrices of tensor products with one-dimensional modules over shifted quantum affine algebras.

In Section \ref{sec: R poly} we prove polynomiality results for R-matrices of irreducible modules over the upper Borel subalgebra in category $\mathcal{O}$.

In Section \ref{sec: type A} for $\Glie$ of type A we extend the associator construction of shifted Yangians to shifted quantum affine algebras.

In Appendix \ref{app: asym}  we identify asymptotic representations with tensor products of irreducible modules with one-dimensional modules.

In Appendix \ref{app: coefficients} we prove that the coefficients of Theta series for shifted Yangians lie in the subalgebra generated by $x^- \otimes x^+$.


\medskip

\noindent {\bf Acknowledgments:} It is a pleasure to thank Giovanni Felder, David Hernandez and Sasha Tsymbaliuk for valuable discussions and correspondences. The author is also grateful to the anonymous referees for helpful comments and suggestions. The author acknowledges support from the Labex CEMPI (ANR-11-LABX-0007-01).


\section{Generalities on shifted Yangians} \label{sec: Yangian}
In this section we review basic properties of shifted Yangians and their representation theory from \cite{HZ}. The ground field is $\BC$, and  $\BN := \BZ_{\geq 0}$.

Fix $\Glie$ a finite-dimensional simple Lie algebra. Let $\Hlie$ be a Cartan subalgebra of $\Glie$, and $I := \{1,2,\cdots,r\}$ be the set of Dynkin nodes. The dual Cartan subalgebra $\Hlie^*$ admits a basis of {\it simple roots} $(\alpha_i)_{i \in I}$ and a non-degenerate invariant symmetric bilinear form $(,): \Hlie^* \times \Hlie^* \longrightarrow \BC$ normalized in such a way that the $d_i := \frac{(\alpha_i,\alpha_i)}{2}$ for $i \in I$ are coprime positive integers in $\{1,2,3\}$. We have the Cartan matrix $(c_{ij})_{i,j\in I}$ and the symmetrized Cartan matrices $(b_{ij})_{i,j\in I}$ and $(d_{ij})_{i,j\in I}$ whose entries are half integers defined by
$$c_{ij} := \frac{2(\alpha_i,\alpha_j)}{(\alpha_i,\alpha_i)} \in \BZ,\quad b_{ij} := (\alpha_i,\alpha_j) \in \BZ,\quad d_{ij} := \frac{(\alpha_i,\alpha_j)}{2} \in \frac{1}{2} \BZ. $$

Let $(\varpi_i^{\vee})_{i \in I}$ be the basis of $\Hlie$ dual to the basis $(\alpha_i)_{i\in I}$ of $\Hlie^*$ with respect to the natural pairing $\langle, \rangle: \Hlie \times \Hlie^* \longrightarrow \BC$; the $\varpi_i^{\vee}$ are called {\it fundamental coweights}. We shall need the coweight lattice, the root lattice and its subsets 
$$\BP^{\vee} := \bigoplus_{i\in I} \BZ \varpi_i^{\vee} \subset \Hlie, \quad  \BQ := \bigoplus_{i\in I} \BZ \alpha_i \subset \Hlie^*,\quad \BQ_+ := \sum_{i\in I} \BN \alpha_i,\quad \BQ_- := -\BQ_+.  $$
By a coweight we mean an element $\mu$ of the coweight lattice $\BP^{\vee}$, namely, a $\BZ$-linear combination $\sum_{i\in I} n_i \varpi_i^{\vee}$ of the fundamental coweights. Notice that $n_i = \langle \mu, \alpha_i \rangle$. Call $\mu$ {\it dominant} if all the $n_i$ are non-negative integers. Call $\mu$ {\it antidominant} if $-\mu$ is dominant.
\subsection{Shifted Yangians}
For $\mu = \sum_{i\in I} n_i \varpi_i^{\vee}$ a coweight, the {\it shifted Yangian} $Y_{\mu}(\Glie)$ is the associative algebra defined by generators 
$$ x_{i,n}^{\pm},\quad \xi_{i,p} \quad \mathrm{for}\ (i, n, p) \in I \times \BN \times \BZ$$
subject to the following relations \cite{KWWY, BFN}: 
 \begin{gather*}
[\xi_{i,p}, \xi_{j,q}] = 0, \quad [x_{i,m}^+,x_{j,n}^-] = \delta_{ij} \xi_{i,m+n},  \\
[\xi_{i,p+1}, x_{j,n}^{\pm}] - [\xi_{i,p}, x_{j,n+1}^{\pm}] = \pm d_{ij} (\xi_{i,p}x_{j,n}^{\pm} + x_{j,n}^{\pm} \xi_{i,p}), \\
[x_{i,m+1}^{\pm}, x_{j,n}^{\pm}] - [x_{i,m}^{\pm}, x_{j,n+1}^{\pm}] = \pm d_{ij} (x_{i,m}^{\pm}x_{j,n}^{\pm} + x_{j,n}^{\pm} x_{i,m}^{\pm}),  \\
\mathrm{ad}_{x_{i,0}^{\pm}}^{1-c_{ij}} (x_{j,0}^{\pm}) = 0 \quad \mathrm{if}\ i \neq j,  \\
\xi_{i,-n_i - 1} = 1,\quad \xi_{i,p} = 0 \quad \mathrm{for}\ p < -n_i - 1. 
\end{gather*}
Here $\mathrm{ad}_x(y) := xy - yx$. Define the generating series for $i \in I$:
 \begin{equation*} 
x_i^{\pm}(z) :=  \sum_{n\in \BN} x_{i,n}^{\pm} z^{-n-1},\quad \xi_i(z) :=  \sum_{p\in \BZ} \xi_{i,p} z^{-p-1},\quad \overline{\xi}_i(z) := z^{-n_i} \xi_i(z).
\end{equation*}
These are Laurent series in $Y_{\mu}(\Glie)((z^{-1}))$, with leading terms $x_{i,0}^{\pm} z^{-1},\ z^{n_i}$ and 1.

The shifted Yangian $Y_{\mu}(\Glie)$ admits a $\BQ$-grading, called its {\it weight grading}, defined by declaring the weights of the generators $x_{i,n}^+, x_{i,n}^-$ and $\xi_{i,p}$ to be $\alpha_i, -\alpha_i$ and $0$. Alternatively, for $\beta \in \BQ$, an element $x \in Y_{\mu}(\Glie)$ is of weight $\beta$ if and only if $[\xi_{i,-n_i}, x] = (\alpha_i,\beta) x$ for all $i \in I$.
Let $Y_{\mu}(\Glie)_{\beta}$ denote the subspace of elements of weight $\beta$. 

We have a $\BC[w]$-algebra automorphism $\tau_w$ of $Y_{\mu}(\Glie)[w]$ defined by
\begin{gather}  \label{rel: spectral shift formal Yangian}
    \tau_w(X_p) = \sum_{n\in \BN} \binom{p}{n}  X_{p-n} w^n \quad \mathrm{for}\ X \in \{x_i^{\pm}, \xi_i\}\ \mathrm{and}\ p \in \BZ.
\end{gather}
Here it is understood that $x_{i,p}^{\pm} = 0$ for $p < 0$. Evaluating $w$ at complex numbers, we get a one-parameter family of algebra automorphisms $\tau_a$ of $Y_{\mu}(\Glie)$ satisfying $\tau_a \circ \tau_b = \tau_{a+b}$ and $\tau_0 = \mathrm{Id}$ for $a, b \in \BC$. We refer to $\tau_w$ and $\tau_a$ as {\it spectral parameter automorphisms}. In terms of generating series we have $\tau_w(X(z)) = X(z-w)$ for $X \in \{x_i^{\pm}, \xi_i\}$. 

\subsection{GKLO series}
In the shifted Yangian $Y_{\mu}(\Glie)$ we have five subalgebras:
\begin{gather*}
Y_{\mu}^+(\Glie) = \langle x_{i,n}^+\rangle_{(i,n)\in I\times \BN},\quad Y_{\mu}^0(\Glie) = \langle \xi_{i,p}\rangle_{(i,p)\in I \times \BZ},\quad Y_{\mu}^-(\Glie) =  \langle x_{i,n}^-\rangle_{(i,n)\in I\times \BN}, \\
Y_{\mu}^{\geq}(\Glie) =  \langle x_{i,n}^+, \xi_{i,p}\rangle_{(i,n,p)\in I\times \BN \times \BZ},\quad Y_{\mu}^{\leq}(\Glie) = \langle x_{i,n}^-, \xi_{i,p}\rangle_{(i,n,p)\in I\times \BN \times \BZ}.
\end{gather*}
(The subalgebras $Y^+, Y^-, Y^0, Y^{\leq}, Y^{\geq}$ here correspond to $Y^>, Y^<, Y^=, Y^-, Y^+$ in \cite{HZ}. We make this modification to match the notations of shifted quantum affine algebras studied later.)
The weight grading and the algebra homomorphism $\tau_w$ restrict to these five subalgebras.
By definition the subalgebra $Y^0_{\mu}(\Glie)$ is commutative. By the PBW theorem \cite[\S 3]{coproduct}, multiplication induces an isomorphism of vector spaces
$$ Y_{\mu}^-(\Glie) \otimes Y_{\mu}^0(\Glie) \otimes Y_{\mu}^+(\Glie) \longrightarrow Y_{\mu}(\Glie), \quad a \otimes b \otimes c \mapsto abc. $$
Furthermore, the defining relations of $Y_{\mu}(\Glie)$ induce a presentation for these subalgebras; for $Y_{\mu}^{\pm}(\Glie)$ one needs the stronger form of the Serre relations in \cite[Definition 3.1]{coproduct}. As a consequence, for $\nu$ another coweight we have canonical identifications of algebras 
$$ Y_{\mu}^{\pm}(\Glie) \cong Y_{\nu}^{\pm}(\Glie),\quad x_{i,n}^{\pm} \mapsto x_{i,n}^{\pm}. $$

\begin{defi}\cite[Lemma 2.1]{GKLO}  \label{defi: GKLO series}
The Gerasimov--Kharchev--Lebedev--Oblezin series $A_i(z)$ for $i \in I$, also called GKLO series, are power series in $z^{-1}$ of leading term 1 with coefficients in the commutative subalgebra $Y_{\mu}^0(\Glie)$ uniquely determined by:
$$ \overline{\xi}_i(z) =   \frac{1}{A_i(z)A_i(z-d_i)} \prod_{j: c_{ji} < 0} \prod_{t=1}^{-c_{ji}} A_j(z-d_{ij}-t d_j). $$
Introduce the GKLO elements $a_{i,m} \in Y_{\mu}^0(\Glie)$ for $(i, m) \in I \times \BN$ so that 
$$A_i(z) = \exp(\sum_{m\geq 0}d_i a_{i,m}  z^{-m-1}) \in 1 + z^{-1} Y_{\mu}^0(\Glie)[[z^{-1}]]. $$
\end{defi}

As in \cite[\S 2.6]{GTL0}, for $i \in I$ let $\sigma_i^+$ denote the algebra endomorphism of $Y_{\mu}^+(\Glie)$ sending $x_{j,n}^+$ to $x_{j,n+\delta_{ij}}^+$. The algebra endomorphism $\sigma_i^-$ of $Y_{\mu}^-(\Glie)$ is defined similarly. Then in the shifted Yangian $Y_{\mu}(\Glie)$ we have the following commutation relations \cite{GKLO}:
\begin{equation} \label{comm: A x}
\begin{split}
A_i(z) x_{j,n}^- A_i(z)^{-1} = \frac{z-\sigma_i^- + d_i \delta_{ij}}{z-\sigma_i^-}  (x_{j,n}^-),    \\
A_i(z)^{-1} x_{j,n}^+ A_i(z) = \frac{z-\sigma_i^+ + d_i \delta_{ij}}{z-\sigma_i^+}  (x_{j,n}^+). 
\end{split}
\end{equation}
At the right-hand side, we take the Taylor expansion of the rational function $\frac{z-\sigma_i^{\pm}+d_i\delta_{ij}}{z-\sigma_i^{\pm}}$ at $z = \infty$ to get a power series in $z^{-1}$ whose coefficients are polynomials in $\sigma_i^{\pm}$ and act on $Y_{\mu}^{\pm}(\Glie)$ as linear operators. In terms of GKLO elements Eq.\eqref{comm: A x} becomes
\begin{equation}  \label{comm: a x}
\begin{split} 
[a_{i,m}, x_{j,n}^-] &= \delta_{ij}  \frac{(\sigma_i^-)^{m+1} - (\sigma_i^--d_i)^{m+1}}{(m+1)d_i} (x_{j,n}^-),  \\
[x_{j,n}^+, a_{i,m}] &=\delta_{ij}  \frac{(\sigma_i^+)^{m+1} - (\sigma_i^+-d_i)^{m+1}}{(m+1)d_i} (x_{j,n}^+). 
\end{split}
\end{equation}
In particular, $[a_{i,0}, x_{j,n}^{\pm}] = \mp \delta_{ij} x_{j,n}^{\pm}$.

\subsection{Shifted homomorphisms and coproduct}
The zero-shifted Yangian $Y_0(\Glie)$ is the ordinary Yangian $Y(\Glie)$ with deformation parameter $\hbar = 1$. It is a Hopf algebra whose coproduct has been determined in \cite{GNW}. The following theorem, due to Finkelberg--Kamnitzer--Pham--Rybnikov--Weekes, extends the ordinary coproduct to arbitrary shifted Yangians in a compatible way.

\begin{theorem}\cite[Corollary 3.16, Theorem 4.12, Proposition 4.14]{coproduct}  \label{thm: coproduct Yangian}
\begin{itemize}
\item[(i)] For $\epsilon, \eta$ antidominant coweights, the following assignments define an injective algebra morphism $\iota_{\epsilon,\eta}^{\mu}: Y_{\mu}(\Glie) \longrightarrow Y_{\mu+\epsilon+\eta}(\Glie)$, called {\it shift homomorphism}:
\begin{gather*} 
x_{i,n}^+ \mapsto x_{i,n-\langle\epsilon,\alpha_i\rangle}^+, \quad x_{i,n}^- \mapsto x_{i,n-\langle\eta,\alpha_i\rangle}^-,\quad \xi_{i,p} \mapsto \xi_{i,p-\langle\epsilon+\eta,\alpha_i\rangle}.
\end{gather*}
\item[(ii)] There exists a unique family of algebra homomorphisms $$\Delta_{\mu,\nu}: Y_{\mu+\nu}(\Glie) \longrightarrow Y_{\mu}(\Glie) \otimes Y_{\nu}(\Glie)$$ 
for all coweights $\mu, \nu$ such that $\Delta_{0,0}$ is the coproduct of the ordinary Yangian, 
\begin{gather*} 
\Delta_{\mu,\nu}(x_{i,n}^+) = x_{i,n}^+ \otimes 1\ \mathrm{if}\ n < - \langle \mu,\alpha_i\rangle,\quad \Delta_{\mu,\nu}(x_{i,n}^-) = 1 \otimes x_{i,n}^- \ \mathrm{if}\ n < -\langle \nu,\alpha_i\rangle, 
\end{gather*}
and the following diagrams commute for $\epsilon$ and $\eta$ antidominant:
\begin{gather} 
 \xymatrixcolsep{6pc} \xymatrix{
Y_{\mu+\nu}(\Glie) \ar[d]^{\iota_{\epsilon,\eta}^{\mu+\nu}} \ar[r]^{\Delta_{\mu,\nu}} & Y_{\mu}(\Glie) \otimes Y_{\nu}(\Glie) \ar[d]^{\iota_{\epsilon,0}^{\mu} \otimes \iota_{0,\eta}^{\nu}} \\
Y_{\mu+\nu+\epsilon+\eta}(\Glie) \ar[r]^{\Delta_{\mu+\epsilon,\nu+\eta}}          & Y_{\mu+\epsilon}(\Glie) \otimes Y_{\nu+\eta}(\Glie), }   \label{rel: coproduct vs shift} \\
  \xymatrixcolsep{6pc} \xymatrix{
Y_{\mu+\epsilon+\nu}(\Glie) \ar[d]^{\Delta_{\mu,\epsilon+\nu}} \ar[r]^{\Delta_{\mu+\epsilon,\nu}} & Y_{\mu+\epsilon}(\Glie) \otimes Y_{\nu}(\Glie) \ar[d]^{\Delta_{\mu,\epsilon} \otimes \mathrm{Id}} \\
Y_{\mu}(\Glie) \otimes Y_{\epsilon+\nu}(\Glie) \ar[r]^{\mathrm{Id} \otimes \Delta_{\epsilon,\nu}}          & Y_{\mu}(\Glie) \otimes Y_{\epsilon}(\Glie) \otimes Y_{\nu}(\Glie). }  \label{rel: coproduct coass}
\end{gather}
\end{itemize}
\end{theorem}

The shift homomorphism sends the series $\overline{\xi}_i(z), A_i(z) \in Y_{\mu}(\Glie)[[z^{-1}]]$ to the corresponding series $\overline{\xi}_i(z), A_i(z) \in Y_{\mu+\epsilon+\eta}(\Glie)[[z^{-1}]]$. 

Recall $\BQ_+ = \sum_{i\in I} \BN \alpha_i$ and set $\BQ_> := \BQ_+ \setminus \{0\}$.
\begin{lem}\cite[Lemma 2.5]{HZ} \label{lem: Yangian coproduct estimation}
For $i \in I,\ n \in \BN$ and $p \in \BZ$, we have:
\begin{gather*}
\Delta_{\mu,\nu}(x_{i,n}^+) \equiv  x_{i,n}^+ \otimes 1 + \sum_{m\in \BN} \xi_{i,n-m-1} \otimes x_{i,m}^+  \ \mathrm{mod}.\sum_{\beta \in \BQ_>, t \in \BZ} Y_{\mu}^-(\Glie)_{-\beta} \xi_{i,t} \otimes Y_{\nu}^+(\Glie)_{\beta+\alpha_i},  \\
\Delta_{\mu,\nu}(x_{i,n}^-) \equiv 1 \otimes x_{i,n}^- + \sum_{m\in \BN} x_{i,m}^- \otimes \xi_{i,n-m-1}  \ \mathrm{mod}. \sum_{\beta \in \BQ_>, t \in \BZ}  Y_{\mu}^-(\Glie)_{-\beta-\alpha_i} \otimes \xi_{i,t} Y_{\nu}^+(\Glie)_{\beta}, \\
\Delta_{\mu,\nu}(\xi_{i,p}) \equiv \sum_{t\in \BZ} \xi_{i,t} \otimes \xi_{i,p-t-1} \ \mathrm{mod}. \sum_{\beta \in \BQ_>} Y_{\mu}^{\leq}(\Glie)_{-\beta} \otimes Y_{\nu}^{\geq}(\Glie)_{\beta}, \\
\Delta_{\mu,\nu}(\xi_{i,-\langle\mu+\nu,\alpha_i\rangle}) = \xi_{i,-\langle\mu,\alpha_i\rangle} \otimes 1 + 1 \otimes \xi_{i,-\langle\nu,\alpha_i\rangle}.
\end{gather*}
\end{lem}
While the above lemma is slightly stronger than \cite[Lemma 2.5]{HZ}, its proof is essentially given in \cite{Knight}.

\subsection{Representations of shifted Yangians}  \label{ss: rep shifted Yangians}
Fix $\mu = \sum_{i\in I} n_i \varpi_i^{\vee}$ to be a coweight. Let $V$ be a module over $Y_{\mu}(\Glie)$. For $\beta \in \Hlie^*$, define
$$ V_{\beta} := \{v \in V\ |\ \xi_{i,-n_i} v = ( \alpha_i, \beta ) v \quad \mathrm{for}\ i \in I \}. $$
If $V_{\beta}$ is nonzero, then $\beta$ is called a weight of $V$ and $V_{\beta}$ is the weight space of weight $\beta$. Notice that $Y_{\mu}(\Glie)_{\gamma} V_{\beta} \subset V_{\gamma+\beta}$ for $\gamma \in \BQ$.

Call $V$ {\it weight graded} if $V$ is a direct sum of the weight spaces. 
Call $V$ {\it top graded} if it is weight graded and there exists $\beta_0 \in \Hlie^*$ such that
 $$ \textrm{$\dim V_{\beta_0} = 1$ and $V_{\beta} \neq \{0\}$ only if $\beta_0 - \beta \in \BQ_+.$}$$
  In this case $\beta_0$ is uniquely determined by the module structure on $V$. Call $\beta_0$ the top weight and $V_{\beta_0}$  the top weight space. Bottom graded modules are defined in the similar way by replacing the condition $\beta_0-\beta \in  \BQ_+$ with $\beta-\beta_0 \in \BQ_+$.
  
An $I$-tuple $\Bf = (\Bf_i(z))_{i\in I} \in \BC((z^{-1}))^I$ of Laurent series in $z^{-1}$ is called {\it $\ell$-weight of coweight $\mu$} if for $i\in I$ the leading term of $\Bf_i(z)$ is $z^{n_i}$; let $\CL_{\mu}$ denote the set of $\ell$-weights of coweight $\mu$. 
To such an $I$-tuple we attach the {\it Verma module} $M(\Bf)$, the $Y_{\mu}(\Glie)$-module generated by $\omega$ subject to relations \cite[\S 3.3]{KTWWY}
$$ \xi_i(z) \omega = \Bf_i(z) \omega,\quad x_i^+(z) \omega = 0 \quad \mathrm{for}\ i \in I. $$
It has a unique irreducible quotient, denoted by $L(\Bf)$. Call a $Y_{\mu}(\Glie)$-module $V$ {\it highest $\ell$-weight} if it is a nonzero quotient of a Verma module $M(\Bf)$; necessarily, $V$ is top graded.
The $\ell$-weight $\Bf$ is called the highest $\ell$-weight of $V$ and it is uniquely determined by the module structure. Any nonzero vector satisfying the above two equations is called a highest $\ell$-weight vector of $\ell$-weight $\Bf$. Replacing the second equation $x_i^+(z) \omega = 0$ with $x_i^-(z) \omega = 0$, we obtain the {\it lowest $\ell$-weight} Verma module, its irreducible quotient, and lowest $\ell$-weight modules. These are bottom graded modules.

An $\ell$-weight $\Bf \in \CL_{\mu}$ is called {\it rational} if for $i \in I$
the series $\Bf_i(z)$ is the Laurent expansion at $z = \infty$ of a rational function. It is {\it polynomial} if each $\Bf_i(z)$ is a monic polynomial of degree $n_i$, which forces the coweight $\mu$ to be dominant. 

\begin{prop} \cite[Theorem 3.12, Lemma 4.1]{HZ} 
Let $\Bf\in \CL_{\mu}$. Then $\Bf$ is rational if and only if all weight spaces of the irreducible $Y_{\mu}(\Glie)$-module $L(\Bf)$ are finite-dimensional. $\Bf$ is polynomial if and only if $L(\Bf)$ is one-dimensional.
\end{prop}  

\begin{defi}  \label{defi: deformed module Yangian}
Let $V$ be a $Y_{\mu}(\Glie)$-module. The polynomial space $V[w]$  being a module over the algebra $Y_{\mu}(\Glie)[w]$ by scalar extension, its pullback along the algebra homomorphism $\tau_w: Y_{\mu}(\Glie) \longrightarrow Y_{\mu}(\Glie)[w]$ of Eq.\eqref{rel: spectral shift formal Yangian} defines a $Y_{\mu}(\Glie)$-module, denoted by $V_w$ and called {\it deformed module}. We refer to $w$ as {\it spectral parameter}.
\end{defi}
Evaluating $w$ at complex numbers gives a one-parameter family of representations $V_a := \tau_a^*V$, for $a \in \BC$, on the same underlying space as $V_0 = V$.  If $V$ is of highest $\ell$-weight $\Bf$, then $V_a$ is of highest $\ell$-weight $(\Bf_i(z-a))_{i\in I}$. By abuse of language, let $\tau_a(\Bf)$ denote the $\ell$-weight $(\Bf_i(z-a))_{i\in I}$.

\begin{example}  \label{example: deformed one-dim Yangian}
Let $\Bp$ be a polynomial $\ell$-weight of coweight $\zeta$. The deformed module $L(\Bp)_w$ associated to the one-dimensional $Y_{\zeta}(\Glie)$-module $L(\Bp)$ is the pullback of the regular representation of $\BC[w]$ along the algebra homomorphism 
$$\pi_{\Bp,w}: Y_{\zeta}(\Glie) \longrightarrow \BC[w],\qquad x_i^{\pm}(z) \mapsto 0,\quad \xi_i(z) \mapsto \Bp_i(z-w)  \quad \mathrm{for}\ i \in I.  $$
 For $j \in I$ and $a \in \BC$, define the {\it prefundamental $\ell$-weight} $\Psi_{j,a}$ to be the polynomial $\ell$-weight of coweight $\varpi_j^{\vee}$ whose $i$th component is $(z-a)^{\delta_{ij}}$ for $i\in I$. 
\end{example}

Let $\mu$ and $\nu$ be coweights. Recall from Theorem \ref{thm: coproduct Yangian} the algebra homomorphism $\Delta_{\mu,\nu}$ from $Y_{\mu+\nu}(\Glie)$ to $Y_{\mu}(\Glie) \otimes Y_{\nu}(\Glie)$. If $M$ and $N$ are modules over $Y_{\mu}(\Glie)$ and $Y_{\nu}(\Glie)$ respectively, then $M \otimes N$ is naturally a module over $Y_{\mu+\nu}(\Glie)$ and by Lemma \ref{lem: Yangian coproduct estimation}:
$$M_{\alpha} \otimes N_{\beta} \subset (M \otimes N)_{\alpha+\beta}\quad\text{ for }\alpha, \beta \in \Hlie^*.$$ 
So, a tensor product of weight modules is still weight graded. For $\Be \in \CL_{\mu}$ and $\Bf \in \CL_{\nu}$, the component-wise multiplication gives an $\ell$-weight $\Be \Bf \in \CL_{\mu+\nu}$. If  $v_1 \in M$ and $v_2 \in N$ are highest $\ell$-weight vectors of $\ell$-weights $\Be$ and $\Bf$ respectively, then $v_1 \otimes v_2 \in  M  \otimes N$ is a highest $\ell$-weight vector of $\ell$-weight $\Be\Bf$. 

\begin{theorem} \cite[Theorem 5.2]{HZ} \label{thm: poly R Yangian} 
Let $M$ and $N$ be highest $\ell$-weight irreducible modules over $Y_{\mu}(\Glie)$ and $Y_{\nu}(\Glie)$ respectively such that the inverse of the highest $\ell$-weight of $N$ is polynomial. Then we have a unique  $Y_{\mu+\nu}(\Glie)$-module morphism 
$$ \check{R}_{M,N}(w): M_w \otimes N \longrightarrow N \otimes M_w $$
which is $\BC[w]$-linear and preserves tensor products of highest $\ell$-weight vectors.
\end{theorem}
One of the goals of this paper is to produce R-matrices $\check{R}_{M,N}(w)$ where $N$ is a tensor product module $L(\Bn^{-1}) \otimes L(\Bm)$ with $\Bm$ and $\Bn$ polynomial. When $\Bm = \Psi_{i,y}$ and $\Bn = \Psi_{i,0}$, such a tensor product module is identified with an asymptotic representation over the ordinary Yangian in \cite[Definition 16]{Z2}; see Eq.\eqref{asym: Yangian}.

\section{Trivial associativity}  \label{sec: trivial}
Let $M, K$ and $N$ be modules over three shifted Yangians $Y_{\mu}(\Glie), Y_{\zeta}(\Glie)$ and $Y_{\nu}(\Glie)$. On the triple tensor product $M \otimes K \otimes N$ there are two $Y_{\mu+\zeta+\nu}(\Glie)$-module structures:
\begin{align*}
    (M \otimes K) \otimes N \quad \mathrm{via}\ (\Delta_{\mu,\zeta} \otimes \mathrm{Id})\Delta_{\mu+\zeta,\nu}; \\
    M \otimes (K \otimes N) \quad \mathrm{via}\ (\mathrm{Id} \otimes \Delta_{\zeta, \nu})\Delta_{\mu, \zeta+\nu}. 
\end{align*}
In general, one has to distinguish the two module structures because the diagram \eqref{rel: coproduct coass} fails to commute for a general coweight $\zeta$. 
By {\it trivial associativity} we mean that the two structures coincide, namely, the identity map $(M \otimes K) \otimes N \longrightarrow M \otimes (K \otimes N)$ is a module isomorphism. This is the case when $K$ is a module over an antidominantly shifted Yangian, in view of the commutative diagram \eqref{rel: coproduct coass}.

The main result of this section is trivial associativity for $M$ or $N$ belonging to a particular family of modules without restriction on the middle module $K$.  

\begin{defi} \label{defi: trivial modules Yangians}
For $\mu$ a coweight, let $\BY_{\mu}(\Glie)$ denote the quotient of $Y_{\mu}(\Glie)$ by the two-sided ideal generated by the $x_{i,m}^{\pm}$ for $i \in I$ and $m \in \BN$. Let $\pi_{\mu}: Y_{\mu}(\Glie) \longrightarrow \BY_{\mu}(\Glie)$ be the quotient map. Call a $Y_{\mu}(\Glie)$-module {\it trivial} if it factorizes through $\pi_{\mu}$.
\end{defi}

By abuse of language, let $\xi_i(z)$ denote the image of $\xi_i(z)$ by $\pi_{\mu}$. The spectral parameter automorphism of Eq.\eqref{rel: spectral shift formal Yangian} descends to $\BY_{\mu}(\Glie)$, still denoted by $\tau_w$.

Given two coweights $\mu$ and $\nu$, let us define 
$$ F_{\mu,\nu} := (\pi_{\mu} \otimes \mathrm{Id}) \Delta_{\mu,\nu}, \quad G_{\mu,\nu} := (\mathrm{Id} \otimes \pi_{\nu}) \Delta_{\mu,\nu}  $$
as algebra homomorphisms from $Y_{\mu+\nu}(\Glie)$ to $\BY_{\mu}(\Glie) \otimes Y_{\nu}(\Glie)$ and to $Y_{\mu}(\Glie) \otimes \BY_{\nu}(\Glie)$ respectively. To describe their actions on the generating series of $Y_{\mu+\nu}(\Glie)$, following \cite[\S 2.3]{HZ}, for a formal power series $f(z) = \sum_{n\in \BZ} f_n z^{-n-1}$ with coefficients in a vector space, we define its {\it principal part} to be the power series $\sum_{n\in \BN} f_n z^{-n-1}$, denoted by $\langle f(z)\rangle_+$. As a consequence of Lemma \ref{lem: Yangian coproduct estimation} we have (see also \cite[(4.25)]{HZ}):
\begin{equation}  \label{equ: F G Yangian}
\begin{split}
    F_{\mu,\nu}:  x_i^+(z) \mapsto  \langle \xi_i(z) \otimes x_i^+(z)\rangle_+, \quad x_i^-(z) \mapsto 1 \otimes x_i^-(z),\quad \xi_i(z) \mapsto \xi_i(z) \otimes \xi_i(z); \\
    G_{\mu,\nu}:  x_i^+(z) \mapsto x_i^+(z) \otimes 1, \quad x_i^-(z) \mapsto \langle x_i^-(z) \otimes \xi_i(z) \rangle_+,\quad \xi_i(z) \mapsto \xi_i(z) \otimes \xi_i(z).
    \end{split}
\end{equation}

\begin{lem}  
    Let $a \in \BC$ and $\mu, \nu, \epsilon, \eta$ be coweights such that $\epsilon$ and $\eta$ are antidominant. Then we have the following identities of algebra homomorphisms:
    \begin{gather}
        (\mathrm{Id} \otimes \iota^{\nu}_{\epsilon,\eta}) F_{\mu,\nu} = F_{\mu,\epsilon+\eta+\nu} \iota_{\epsilon, \eta}^{\mu+\nu}, \quad (\iota^{\mu}_{\epsilon,\eta} \otimes \mathrm{Id}) G_{\mu,\nu} = G_{\mu+\epsilon+\eta,\nu} \iota_{\epsilon,\eta}^{\mu+\nu}; \label{F shift} \\
        (\tau_a \otimes \tau_a) F_{\mu,\nu} = F_{\mu,\nu} \tau_a, \quad (\tau_a \otimes \tau_a) G_{\mu,\nu} = G_{\mu,\nu} \tau_a.  \label{F G spectral Yangian} 
    \end{gather}
\end{lem}
\begin{proof}
We only prove the first half of Eq.\eqref{F shift}, as the idea works for the others. Write $\epsilon = \sum_{i\in I} b_i \varpi_i^{\vee}$ and $\eta = \sum_{i\in I} c_i \varpi_i^{\vee}$. Then $\iota_{\epsilon,\eta}^{\nu}$  sends the generating series of $Y_{\nu}(\Glie)$ to 
$$ x_i^+(z) \mapsto \langle z^{-b_i} x_i^+(z)\rangle_+, \quad x_i^-(z) \mapsto \langle z^{-c_i}x_i^-(z)\rangle_+, \quad \xi_i(z) \mapsto z^{-b_i-c_i} \xi_i(z). $$
Evaluate both sides of the first half of Eq.\eqref{F shift} at $x_i^+(z)$:
\begin{align*}
(\mathrm{Id} \otimes \iota^{\nu}_{\epsilon,\eta}) F_{\mu,\nu}: x_i^+(z) & \mapsto  \langle \xi_i(z) \otimes x_i^+(z)\rangle_+ \mapsto  \langle\xi_i(z) \otimes \langle z^{-b_i} x_i^+(z)\rangle_+ \rangle_+, \\
F_{\mu,\epsilon+\eta+\nu} \iota_{\epsilon,\eta}^{\mu+\nu}: x_i^+(z) & \mapsto \langle z^{-b_i} x_i^+(z) \rangle_+ \mapsto \langle z^{-b_i} \langle \xi_i(z) \otimes x_i^+(z)\rangle_+ \rangle_+.
\end{align*}
Notice that $\xi_i(z)$ at the first tensor factor is a polynomial with coefficients in $\BY_{\mu}(\Glie)$, and $b_i, c_i$ are negative integers. The two sides coincide because of the equation 
$$\langle g(z) \langle h(z) f(z)\rangle_+\rangle_+ = \langle g(z) h(z) f(z)\rangle_+ $$
where $f(z)$ is a formal power series and $g(z)$ and $h(z)$ are polynomials with coefficients in an algebra. Similarly both sides evaluated at $x_i^-(z)$ and $\xi_i(z)$ coincide.
\end{proof}

It is unclear to us whether Eq.\eqref{F G spectral Yangian} remains true if $F_{\mu,\nu}$ is replaced by $\Delta_{\mu,\nu}$. We arrive at the main result of this section, whose proof is close in spirit to \cite[Proposition 4.14]{coproduct}.
\begin{theorem}  \label{thm: trivial associativity Yangian}
Given $\mu, \zeta$ and $\nu$ coweights, we have as algebra homomorphisms from $Y_{\mu+\zeta+\nu}(\Glie)$ to $Y_{\mu}(\Glie)\otimes Y_{\zeta}(\Glie) \otimes \BY_{\nu}(\Glie)$ and to $\BY_{\mu}(\Glie) \otimes Y_{\zeta}(\Glie) \otimes Y_{\nu}(\Glie)$ respectively, 
\begin{equation*}  
\begin{split}
(\pi_{\mu} \otimes \mathrm{Id} \otimes \mathrm{Id}) (\mathrm{Id} \otimes \Delta_{\zeta,\nu})  \Delta_{\mu,\zeta+\nu} = (\pi_{\mu} \otimes \mathrm{Id} \otimes \mathrm{Id}) (\Delta_{\mu,\zeta} \otimes \mathrm{Id}) \Delta_{\mu+\zeta,\nu},   \\
(\mathrm{Id} \otimes \mathrm{Id} \otimes \pi_{\nu}) (\mathrm{Id} \otimes \Delta_{\zeta,\nu})  \Delta_{\mu,\zeta+\nu} = (\mathrm{Id} \otimes \mathrm{Id} \otimes \pi_{\nu}) (\Delta_{\mu,\zeta} \otimes \mathrm{Id}) \Delta_{\mu+\zeta,\nu}. 
\end{split}
\end{equation*}
\end{theorem}
\begin{proof}
We shall prove the second equation, as the first one is parallel. Choose an antidominant coweight $\eta$ such that $\zeta + \eta$ is antidominant and compose the left-hand side of the second equation by $\mathrm{Id} \otimes \iota_{0,\eta}^{\zeta} \otimes \mathrm{Id}$. We get the following identities of algebra homomorphisms (the underline indicates which expression is being altered next)
\begin{alignat*}{2}
&\qquad (\mathrm{Id} \otimes \iota_{0,\eta}^{\zeta} \otimes \mathrm{Id}) (\mathrm{Id} \otimes \mathrm{Id} \otimes \pi_{\nu}) (\mathrm{Id} \otimes \Delta_{\zeta,\nu})  \Delta_{\mu,\zeta+\nu} && \\
&=  (\mathrm{Id} \otimes (\iota_{0,\eta}^{\zeta} \otimes \pi_{\nu}) \Delta_{\zeta,\nu})\Delta_{\mu,\zeta+\nu} = (\mathrm{Id} \otimes \underline{(\iota_{0,\eta}^{\zeta} \otimes \mathrm{Id})  G_{\zeta,\nu}})\Delta_{\mu,\zeta+\nu}  &\qquad& \mathrm{Eq.}\eqref{F shift} \\
&= (\mathrm{Id} \otimes G_{\zeta+\eta,\nu} \iota_{0,\eta}^{\zeta+\nu}) \Delta_{\mu,\zeta+\nu} 
= (\mathrm{Id} \otimes G_{\zeta+\eta,\nu}) \underline{(\mathrm{Id} \otimes  \iota_{0,\eta}^{\zeta+\nu}) \Delta_{\mu,\zeta+\nu}} &&\mathrm{Eq.}\eqref{rel: coproduct vs shift}  \\
&= (\mathrm{Id} \otimes G_{\zeta+\eta,\nu}) \Delta_{\mu,\zeta+\eta+\nu} \iota_{0,\eta}^{\mu+\zeta+\nu} &&  \\
&=   (\mathrm{Id} \otimes \mathrm{Id} \otimes \pi_{\nu}) \underline{(\mathrm{Id} \otimes \Delta_{\zeta+\eta,\nu})  \Delta_{\mu,\zeta+\eta+\nu}} \iota_{0,\eta}^{\mu+\zeta+\nu} && \mathrm{Eq.}\eqref{rel: coproduct coass}  \\
&= (\mathrm{Id} \otimes \mathrm{Id} \otimes \pi_{\nu}) (\Delta_{\mu,\zeta+\eta} \otimes \mathrm{Id}) \Delta_{\mu+\zeta+\eta,\nu} \iota_{0,\eta}^{\mu+\zeta+\nu} &&    \\
&=(\Delta_{\mu,\zeta+\eta} \otimes \mathrm{Id}) \underline{G_{\mu+\zeta+\eta,\nu} \iota_{0,\eta}^{\mu+\zeta+\nu}} && \mathrm{Eq.}\eqref{F shift} \\
&= (\Delta_{\mu,\zeta+\eta} \otimes \mathrm{Id})  (\iota_{0,\eta}^{\mu+\zeta} \otimes \mathrm{Id}) G_{\mu+\zeta,\nu}  = (\underline{\Delta_{\mu,\zeta+\eta} \iota_{0,\eta}^{\mu+\zeta}} \otimes \mathrm{Id})G_{\mu+\zeta,\nu}  &&\mathrm{Eq.}\eqref{rel: coproduct vs shift} \\
&= ((\mathrm{Id} \otimes \iota_{0,\eta}^{\zeta}) \Delta_{\mu,\zeta} \otimes \mathrm{Id})  G_{\mu+\zeta,\nu}  && \\
&= (\mathrm{Id} \otimes \iota_{0,\eta}^{\zeta} \otimes \mathrm{Id}) (\mathrm{Id} \otimes \mathrm{Id} \otimes \pi_{\nu}) (\Delta_{\mu,\zeta} \otimes \mathrm{Id}) \Delta_{\mu+\zeta,\nu} &&
\end{alignat*}
which is precisely the right-hand side of the second equation composed by $(\mathrm{Id} \otimes \iota_{0,\eta}^{\zeta} \otimes \mathrm{Id})$. Conclude from the injectivity of $\iota_{0,\eta}^{\zeta}$.  
\end{proof}
A reformulation of the theorem is the following trivial associativity: for any triple $(M, K, N)$ of modules over three shifted Yangians, if $M$ or $N$ is a trivial module, then the identity map defines a module isomorphism $(M \otimes K) \otimes N \longrightarrow M \otimes (K \otimes N)$.
\begin{example}  \label{example: tensor prefund Yangian}
Let $\Bp$ be a polynomial $\ell$-weight of coweight $\zeta$. Recall from Example \ref{example: deformed one-dim Yangian} the algebra homomorphism $\pi_{\Bp,w}: Y_{\zeta}(\Glie) \longrightarrow \BC[w]$. Given a coweight $\mu$, we have two algebra homomorphisms $F_{\Bp,w}^{\mu}$ and $G_{\Bp,w}^{\mu}$, both from $Y_{\mu}(\Glie)$ to $Y_{\mu-\zeta}(\Glie)[w]$, defined by:
\begin{equation*}  
F_{\Bp,w}^{\mu} := (\pi_{\Bp,w} \otimes \mathrm{Id})  \Delta_{\zeta,\mu-\zeta}, \quad G_{\Bp,w}^{\mu} := (\mathrm{Id} \otimes \pi_{\Bp,w})  \Delta_{\mu-\zeta,\zeta}.
\end{equation*}
Since $\pi_{\Bp,w}$ defines the trivial module $L(\Bp)_w$, from Eq.\eqref{equ: F G Yangian} we get :
\begin{align*}
F_{\Bp,w}^{\mu}: &\ x_i^+(z) \mapsto \langle \Bp_i(z-w) x_i^+(z)\rangle_+,\quad x_i^-(z) \mapsto x_i^-(z),\quad \xi_i(z) \mapsto \Bp_i(z-w)\xi_i(z), \\
G_{\Bp,w}^{\mu}: &\ x_i^+(z) \mapsto x_i^+(z),\quad x_i^-(z) \mapsto \langle \Bp_i(z-w)x_i^-(z)\rangle_+, \quad \xi_i(z) \mapsto \Bp_i(z-w) \xi_i(z).
\end{align*}
Let $V$ be a $Y_{\mu-\zeta}(\Glie)$-module. View $V[w]$ as a module over  $Y_{\mu-\zeta}(\Glie)[w]$ by scalar extension. Then tensor products with $L(\Bp)_w$ are pullback modules:
$$ L(\Bp)_w \otimes V = (F_{\Bp,w}^{\mu})^* (V[w]), \quad V \otimes L(\Bp)_w = (G_{\Bp,w}^{\mu})^* (V[w]). $$
Recall the deformed module $V_z$ from Definition \ref{defi: deformed module Yangian} with spectral parameter $z$. On the same underlying space $V[z,w]$ we have identifications of modules:
$$L(\Bp)_w \otimes V_z = (L(\Bp)_{w-z}\otimes V)_z, \quad V_z \otimes L(\Bp)_w  = (V \otimes L(\Bp)_{w-z})_z. $$
This follows from Eq.\eqref{F G spectral Yangian} by viewing $a$ as a formal variable.
\end{example}

\section{S-series and T-operators}  \label{sec: T Yangian}
In this section we construct module morphisms $L(\Bp)_w \otimes V \longrightarrow V \otimes L(\Bp)_w$, for $\Bp$ a polynomial $\ell$-weight and $V$ a graded module, by solving in the shifted Yangian an additive difference equation determined by the GKLO series. 

For $G(z) \in 1 + z^{-1} \BC[[z^{-1}]]$ a power series and $x$ an element in an associative algebra $\mathcal{A}$, take the logarithm $\log G(z) \in z^{-1} \BC[[z^{-1}]]$ and define the power series
$$ G(z)^x := \exp(x \log G(z)) \in 1 + z^{-1} \mathcal{A}[[z^{-1}]].  $$
As an example, taking $G(z) = \frac{z}{z+c}$ with $c \in \BC$, we have
$$ \left(\frac{z}{z+c}\right)^x = \exp (x \sum_{k>0} \frac{(-c)^k}{k}z^{-k} ).  $$
Recall from  Definition \ref{defi: GKLO series} the GKLO series $A_i(z)$ and elements $a_{i,m}$.
\begin{prop}  \label{prop: S-series}
Fix a coweight $\mu$ and $i \in I$. There exists a unique power series $S_i(z) \in 1 + z^{-1} Y_{\mu}^0(\Glie)[[z^{-1}]]$ which solves the additive difference equation
\begin{equation}
S_i(z+d_i) = S_i(z) \times A_i(z) \times \left(\frac{z}{z+d_i}\right)^{a_{i,0}}. \label{difference: S A}
\end{equation}
Furthermore, the following relations hold in $Y_{\mu}(\Glie)[[z^{-1}]]$ for $(j,n) \in I \times \BN$:
\begin{equation}   \label{comm: S x}
\begin{split}
S_i(z) x_{j,n}^- S_i(z)^{-1} &= (1-\sigma_i^- z^{-1})^{\delta_{ij}} (x_{j,n}^-),    \\
S_i(z)^{-1} x_{j,n}^+ S_i(z) &= (1- \sigma_i^+ z^{-1})^{\delta_{ij}} (x_{j,n}^+). 
\end{split}
\end{equation}
\end{prop}
\begin{proof}
We shall solve Eq.\eqref{difference: S A} and prove the first half of \eqref{comm: S x}, as the second half is parallel. Write $S_i(z) = \exp(\sum_{m\geq 0} s_mz^{-m-1})$. Then Eq.\eqref{difference: S A} is equivalent to the following equation in $Y_{\mu}^0(\Glie)[[z^{-1}]]$ with variables $s_m$ for $m\in \BN$:
$$ \sum_{m\geq 0} s_m ((z+d_i)^{-m-1} - z^{-m-1}) = \sum_{m\geq 0} (d_ia_{i,m} + \frac{(-d_i)^{m+1}}{m+1} a_{i,0})z^{-m-1}. $$
At both sides the coefficient of $z^{-1}$ is zero. For $m> 0$, the coefficient of $z^{-m-1}$ at the left-hand side is $-m d_i s_{m-1}$ plus a linear combination of $s_k$ for $0\leq k \leq m-2$, and at the right-hand side it is the constant $d_ia_{i,m} + \frac{(-d_i)^{m+1}}{m+1} a_{i,0}$.  We obtain therefore a triangular system of linear equations in the variables $s_m$ which has a unique solution. For $m \in \BN$ fixed, $s_m$ is a $\BC$-linear combination of the $a_{i,k}$ for $0\leq k \leq m+1$, and in view of Eq.\eqref{comm: a x} there exists a polynomial $c_m(t) \in \BC[t]$ such that 
$$ [s_m, x_{j,n}^-] = \delta_{ij} c_m(\sigma_i^-)(x_{j,n}^-)\quad \mathrm{for}\ (j,n) \in I \times \BN. $$
Therefore $S_i(z)$ commutes with the $x_{j,n}^-$ for $j \neq i$ and
\begin{align*}
S_i(z) x_{i,n}^- S_i(z)^{-1} = g(z) (x_{i,n}^-) \quad \mathrm{where} \\
g(z) := \exp(\sum_{m\geq 0} c_m(\sigma_i^-) z^{-m-1}) \in 1 + z^{-1} \BC[\sigma_i^-][[z^{-1}]]. 
\end{align*}
Notice from the commutation relation $[a_{i,0}, x_{i,n}^-] =  x_{i,n}^-$ that
\begin{align*}
\left(\frac{z}{z+d_i}\right)^{a_{i,0}} x_{i,n}^-  = \frac{z}{z+d_i} x_{i,n}^- \left(\frac{z}{z+d_i}\right)^{a_{i,0}}. 
\end{align*}
From Eq.\eqref{comm: A x} we get that the conjugations of the $x_{i,n}^-$ by both sides of Eq.\eqref{difference: S A} give another difference equation in $1 + z^{-1}\BC[\sigma_i^-][[z^{-1}]]$:
\begin{align*}
g(z+d_i) &= g(z) \times \frac{z-\sigma_i^-+d_i}{z-\sigma_i^-} \times \frac{z}{z+d_i}.
\end{align*}
Standard uniqueness arguments as in the case of the difference equation \eqref{difference: S A} indicate that $g(z) = 1 - \sigma_i^- z^{-1}$. This proves the first half of Eq.\eqref{comm: S x}.
\end{proof}
\begin{rem}  
Our difference equation \eqref{difference: S A} at the level of algebras is close to the difference equation \cite[Eq.(5.35)]{HZ} at the level of representations. In the present situation we have to introduce an additional factor at the right-hand side in order to guarantee the existence of a solution. Combining the difference equation \eqref{difference: S A} with Definition \ref{defi: GKLO series} we express the Drinfeld--Cartan series in terms of S-series in the same way as in the last paragraph of the proof of \cite[Proposition 5.8]{HZ}:
\begin{gather} \label{xi S}
\overline{\xi}_i(z) = \prod_{j \in I}  \frac{S_j(z-d_{ij})}{S_j(z+d_{ij})}  \left(\frac{z-d_{ij}}{z+d_{ij}}\right)^{a_{j,0}}.   
\end{gather}
\end{rem}

For $\mu$ a coweight and $\Bp = \Psi_{i_1,a_1} \Psi_{i_2,a_2} \cdots \Psi_{i_n,a_n}$ a polynomial $\ell$-weight, define the S-series $S_{\Bp}(w) \in 1 + w^{-1} Y_{\mu}^0(\Glie)[[w^{-1}]]$ to be the power series
\begin{equation}   \label{def: S series one-dim Yangian}
S_{\Bp}(w) := S_{i_1}(w+a_1) S_{i_2}(w+a_2) \cdots S_{i_n}(w+a_n).
\end{equation} 
As a consequence of Eq.\eqref{comm: S x}, we have in $Y_{\mu}(\Glie)[[z^{-1},w^{-1}]]$:
\begin{equation}   \label{comm: S x general}
\begin{split}
S_{\Bp}(w) x_i^-(z)   &= \frac{1}{\Bp_i(-w)} \langle \Bp_i(z-w) x_i^-(z) \rangle_+S_{\Bp}(w),    \\
x_i^+(z) S_{\Bp}(w) &= \frac{1}{\Bp_i(-w)} S_{\Bp}(w) \langle \Bp_i(z-w) x_i^+(z) \rangle_+.
\end{split}
\end{equation}
Here $\langle \rangle_+$ always means the principal part with respect to $z$ as defined before Eq.\eqref{equ: F G Yangian}. Indeed, by repeatedly applying the first formula of Eq.\eqref{comm: S x} we have
$$ S_{\Bp}(w) x_i^-(z) S_{\Bp}(w)^{-1} = \frac{\Bp_i(\sigma_i^--w)}{\Bp_i(-w)}(x_i^-(z)). $$
The first formula of Eq.\eqref{comm: S x general} follows using that $(\sigma_i^-)^n (x_i^-(z)) = \langle z^n x_i^-(z)\rangle_+$ for $n \in \BN$. The second formula is proved in the same way.

\begin{defi} \label{defi: root graded Yangian} Let $V$ be a module over the shifted Yangian $Y_{\mu}(\Glie)$.

(i) Call $V$ {\it root graded} if it is a $\BQ$-graded module over the $\BQ$-graded algebra $Y_{\mu}(\Glie)$. Namely, it is a direct sum of vector subspaces $V = \oplus_{\beta \in \BQ} V_{(\beta)}$ such that $Y_{\mu}(\Glie)_{\gamma} V_{(\beta)} \subset V_{(\gamma+\beta)}$ for $\beta, \gamma \in \BQ$. It is {\it positive root graded} if $V = \oplus_{\beta\in \BQ_+} V_{(\beta)}$, and  {\it negative root graded} if $V = \oplus_{\beta\in \BQ_-} V_{(\beta)}$.

(ii) Let $V$ be a root graded module. Define $V^w$ to be the subspace of $V((w^{-1}))$ linearly spanned by the $V_{(\beta)}((w^{-1}))$ for $\beta \in \BQ$.
For $\Bp = (\Bp_i(z))_{i\in I}$ a polynomial $\ell$-weight, define the linear automorphism $D_{\Bp,w}^V$ of $V^w$ by 
$$ D_{\Bp,w}^V (g(w)) := g(w) \times \prod_{i\in I}\Bp_i(-w)^{-\langle \varpi_i^{\vee}, \beta \rangle} \qquad \textrm{for $\beta \in \BQ$ and $g(w) \in V_{(\beta)}((w^{-1}))$}. $$
\end{defi}  
Notice that $V^w$ is intermediate between $V[w] \subset V((w^{-1}))$. It is acted upon by the subalgebras $Y_{\mu}(\Glie)[w]$ and $Y_{\mu}(\Glie)[[w^{-1}]]$ of $Y_{\mu}(\Glie)((w^{-1}))$ but not by the full algebra.
The linear automorphism $D_{\Bp,w}^V$ of $V^w$  does not extend to $V((w^{-1}))$ in general. 
\begin{example}   \label{example: motivation root graded Yangian}
(i) If $V$ is top graded of top weight $\beta_0 \in \Hlie^*$, then it is negative root graded by setting $V_{(\beta)}$ to be the weight space $V_{\beta+\beta_0}$. If $V$ is bottom graded of bottom weight $\beta_0$, then it is positive root graded by setting $V_{(\beta)} := V_{\beta+\beta_0}$. If $V$ is a trivial module in the sense of Definition \ref{defi: trivial modules Yangians}, then it is root graded by setting $V = V_{(0)}$.

(ii) The regular representation of $Y_{\mu}(\Glie)$ on itself is not weight graded as a module but root graded, by setting $Y_{\mu}(\Glie)_{(\beta)} := Y_{\mu}(\Glie)_{\beta}$. 

(iii) Let $M$ be a root graded $Y_{\mu}(\Glie)$-module and $N$ be a root graded $Y_{\nu}(\Glie)$-module. The $Y_{\mu+\nu}(\Glie)$-module $M \otimes N$ is root graded by setting $(M\otimes N)_{(\beta)}$ to be the direct sum of the $M_{(\gamma)} \otimes N_{(\beta-\gamma)}$ for $\gamma \in \BQ$. We have $(M \otimes N)^w \supset M^w \otimes N$. However, the linear automorphism $D_{\Bp,w}^M \otimes \mathrm{Id}_N$ of $M^w \otimes N$ does not extend to $(M \otimes N)^w$ in general.
\end{example}
\begin{prop}
Finite-dimensional modules over shifted Yangians are root graded.
\end{prop}
\begin{proof}
Let $V$ be a finite-dimensional module over $Y_{\mu}(\Glie)$. As a module over the commutative subalgebra of $Y_{\mu}(\Glie)$ generated by the $\xi_{i,-\langle \mu, \alpha_i\rangle}$ for $i \in I$, it is a direct sum of {\it generalized weight spaces} $V^{\gamma}$ for $\gamma \in \Hlie^*$ defined by (see also \cite[(5.1)]{BK})
$$ V^{\gamma} := \{v \in V\ |\ \forall i \in I, \exists k \in \BN\ \textrm{such that } (\xi_{i,-\langle\mu,\alpha_i\rangle}) - (\alpha_i, \gamma))^k v = 0 \}. $$ 
We have $Y_{\mu}(\Glie)_{\alpha} V^{\gamma} = V^{\gamma + \alpha}$ for $\alpha \in \BQ$. Let $V^{\langle \gamma\rangle}$ denote the direct sum of subspaces $V^{\gamma+\beta}$ for $\beta \in \BQ$. Since $V$ is finite dimensional, it is a direct sum of finitely many subspaces $V^{\langle \gamma_1\rangle}, V^{\langle \gamma_2\rangle}, \cdots, V^{\langle \gamma_m\rangle}$. For $\beta \in \BQ$ let us set 
$$ V_{(\beta)} := V^{\gamma_1+\beta} \oplus V^{\gamma_2+\beta} \oplus \cdots \oplus V^{\gamma_m + \beta}. $$
This equips $V$ with a root grading. 
\end{proof}

Let $V$ be a root graded module over $Y_{\mu}(\Glie)$ and $\Bp$ be a polynomial $\ell$-weight of coweight $\zeta$. As in Example \ref{example: tensor prefund Yangian} we view $V^w$ as a $Y_{\mu}(\Glie)[w]$-module by scalar extension and take its pullback along the algebra homomorphism $F_{\Bp,w}^{\mu+\zeta}: Y_{\mu+\zeta}(\Glie) \longrightarrow Y_{\mu}(\Glie)[w]$. The resulting $Y_{\mu+\zeta}(\Glie)$-module, denoted by $L(\Bp)_w \tilde{\otimes} V$, is a completion of the ordinary tensor product module $L(\Bp)_w \otimes V$. Similarly the $Y_{\mu+\zeta}(\Glie)$-module $V \tilde{\otimes} L(\Bp)_w$ is defined via $G_{\Bp,w}^{\mu+\zeta}$ as a completion of $V \otimes L(\Bp)_w$. Let the S-series $S_{\Bp}(w) \in Y_{\mu}(\Glie)[[w^{-1}]]$ of Eq.\eqref{def: S series one-dim Yangian} act on $V^w$ as a linear automorphism. We arrive at the main result of this section.
\begin{theorem} \label{thm: Yangian T-series}
Let $\Bp$ be a polynomial $\ell$-weight of coweight $\zeta$ and $V$ be a root graded $Y_{\mu}(\Glie)$-module. Then the composition $D_{\Bp,w}^V \circ S_{\Bp}(w)$ induces a linear automorphism of the space $V^w$, denoted by $T_{\Bp}^V(w)$, which is a $Y_{\mu+\zeta}(\Glie)$-module isomorphism
$$ T_{\Bp}^V(w): L(\Bp)_w \tilde{\otimes} V \longrightarrow V \tilde{\otimes} L(\Bp)_w. $$
\end{theorem}
\begin{proof}
By Example \ref{example: tensor prefund Yangian}, it suffices to prove the following intertwining property
 \begin{gather}
  \quad \begin{cases}
  T_{\Bp}^V(w) \circ x_i^-(z) = \langle \Bp_i(z-w) x_i^-(z) \rangle_+ \circ T_{\Bp}^V(w),  \\
  x_i^+(z) \circ T_{\Bp}^V(w) = T_{\Bp}^V(w) \circ \langle \Bp_i(z-w) x_i^+(z) \rangle_+, \\
  T_{\Bp}^V(w) \circ \xi_i(z) = \xi_i(z) \circ T_{\Bp}^V(w),
  \end{cases}  \label{intertwining T} 
\end{gather}
as equations in $(\mathrm{End} V^w)((z^{-1}))$. Here the $x_i^{\pm}(z)$ and $\xi_i(z)$ are viewed as generating series of $Y_{\mu}(\Glie)$ whose action on $V^w$ is induced from the $Y_{\mu}(\Glie)$-module structure on $V$. Let $i \in I$ and $n \in \BN$.
If $v \in V_{(\beta)}$, then $x_{i,n}^- v \in V_{(\beta-\alpha_i)}$ and
$$ D_{\Bp,w}^V (x_{i,n}^-v) = \Bp_i(-w)  x_{j,n}^- D_{\Bp,w}^V(v). $$
Namely, $D_{\Bp,w}^V \circ x_{i,n}^- = \Bp_i(-w)  x_{i,n}^-\circ D_{\Bp,w}^V$. The factor $\Bp_i(-w)$ cancels with the denominator in the commutation relation between $S_{\Bp}(w)$ and $x_i^-(z)$ in Eq.\eqref{comm: S x general}. We get the desired $Tx^-$ relation. The $x^+T$ and $T\xi$ relations are proved similarly. 
\end{proof}

\begin{rem} 
To simplify notations let us write $T_{\Psi_{i,0}}^V(w)$ as $T_i^V(w)$ for $i \in I$.

(i) The $T_{\Bp}^V(w)$ for all polynomial $\ell$-weights $\Bp$ form a commuting family of linear automorphisms of $V^w$, and each factorization $\Bp = \Psi_{i_1, a_1} \Psi_{i_2, a_2} \cdots \Psi_{i_n,a_n}$ into prefundamental $\ell$-weights induces a decomposition of T-operators: 
$$ T_{\Bp}^V(w) = T_{i_1}^V(w+a_1) \circ T_{i_2}^V(w+a_2)  \circ \cdots \circ T_{i_n}^V(w+a_n) \in \mathrm{Aut}(V^w). $$

(ii) We view the T-operator $T_i^V(w)$ as $S_i(w) \times (-w)^{-\varpi_i^{\vee}}$ where the factor $(-w)^{-\varpi_i^{\vee}}$ is well-defined for root graded modules; this is similar to evaluations of the universal R-matrix for the quantum loop algebra in Proposition \ref{prop: R-matrix general}. The intertwining property \eqref{intertwining T} for $\Bp = \Psi_{i,0}$ has been solved in \cite[Proposition 5.8]{HZ} for $V$ of highest $\ell$-weight and in \cite[\S 4.3]{GW} for $V$ weight graded by finite-dimensional weight spaces. Both solutions depend uniquely on the module structure. In comparison, our T-operator depends also on the additional root grading but can be applied to more general modules. 

(iii) Let $V$ be a weight graded module over $Y_{\mu}(\Glie)$ whose weights belong to $\BQ$. Then $V$ is root graded by setting $V_{(\alpha)} = V_{\alpha}$. Since $a_{i,0}$ acts on $V_{\alpha}$ as the scalar $-\langle \varpi_i^{\vee}, \alpha\rangle$, we obtain from Eq.\eqref{xi S} the following equation as automorphisms of $V^w$:
\begin{align}  \label{xi T Yangian}
\overline{\xi}_i(w) = \prod_{j \in I}  \frac{T_j^V(w-d_{ij})}{T_j^V(w+d_{ij})}.
\end{align}
\end{rem}

\section{Theta series and intertwining property}  \label{sec: Theta Yangian}
In this section we introduce what we call Theta series $\Theta_{\Bp}(w)$, depending on a polynomial $\ell$-weight $\Bp$, in the tensor product of two shifted Yangians. Their evaluation at a tensor product module $M \otimes N$ is interpreted as a module isomorphism $(M \otimes L(\Bp)_w) \otimes N \longrightarrow M \otimes (L(\Bp)_w \otimes N)$, commonly called {\it associators}. We establish various properties of Theta series, notably, the polynomiality.

\subsection{Definition of Theta series} 
In this subsection we define Theta series from the S-series.
Since we deal with several coweights, we add a superscript $\mu$ to the auxiliary series $A_i(z), S_{\Bp}(z)$ to indicate that their coefficients belong to the shifted Yangian $Y_{\mu}(\Glie)$. For example, the shifted homomorphism of Theorem \ref{thm: coproduct Yangian} preserves GKLO series:
$$ \iota^{\mu}_{\epsilon,\eta}(A_i^{\mu}(z)) = A_i^{\mu+\epsilon+\eta}(z) \quad \mathrm{for}\ i \in I. $$
\begin{lem}  \label{lem: coproduct S}
For two coweights $\mu, \nu$ and a polynomial $\ell$-weight $\Bp$, we have
$$ \Delta_{\mu,\nu}(S_{\Bp}^{\mu+\nu}(z)) \equiv S_{\Bp}^{\mu}(z) \otimes S_{\Bp}^{\nu}(z) \ \mathrm{mod}.\ \sum_{\beta \in \BQ_>} Y_{\mu}^{\leq}(\Glie)_{-\beta} \otimes Y_{\nu}^{\geq}(\Glie)_{\beta}. $$
\end{lem}
\begin{proof}
By Eq.\eqref{def: S series one-dim Yangian}, if we write $S_{\Bp}(z) = \exp(\sum_{m\in \BN} s_{\Bp,m} z^{-m-1})$, then  it suffices to show the quasi-primitivity property $\Delta_{\mu,\nu}(s_{\Bp,m}) \equiv s_{\Bp,m} \otimes 1 + 1 \otimes s_{\Bp,m}$. In view of the proof of Proposition \ref{prop: S-series}, each $s_{\Bp,m}$ is a linear combination of the GKLO elements $a_{i,n}$. We are reduced to quasi-primitivity for the GKLO elements. 

Let us write $\overline{\xi}_i(z) = \exp(\sum_{m\in \BN} t_{i,m} z^{-m-1})$ for $i \in I$. By taking logarithms at both sides of the first equation of Definition \ref{defi: GKLO series} we get that the GKLO elements are linear combinations of the $t_{i,m}$. The quasi-primitivity for the GKLO elements is then a direct consequence of $\Delta_{\mu,\nu}(\overline{\xi}_i(z)) \equiv \overline{\xi}_i(z) \otimes \overline{\xi}_i(z)$ in Lemma \ref{lem: Yangian coproduct estimation}.
\end{proof}

\begin{defi}\label{def: Theta series Yangian}
Let $\mu, \nu$ be coweights and $\Bp$ be a polynomial $\ell$-weight. Define the {\it Omega series} $\Omega_{\Bp}^{\mu,\nu}(w)$ to be the power series in $1 + w^{-1}(Y_{\mu}^{\leq}(\Glie) \otimes Y_{\nu}^{\geq}(\Glie))[[w^{-1}]]$ by
\begin{equation*}  
\Delta_{\mu,\nu}(S_{\Bp}^{\mu+\nu}(w)) = (1\otimes S_{\Bp}^{\nu}(w)) \times \Omega_{\Bp}^{\mu,\nu}(w) \times (S_{\Bp}^{\mu}(w) \otimes 1).
\end{equation*}
For $\beta \in \BQ_+$, let $\Omega_{\Bp,\beta}^{\mu,\nu}(w) \in (Y_{\mu}^{\leq}(\Glie)_{-\beta} \otimes Y_{\nu}^{\geq}(\Glie)_{\beta})[[w^{-1}]]$ be the $\beta$-component of $\Omega_{\Bp}^{\mu,\nu}(w)$ and define $\Theta_{\Bp,\beta}^{\mu,\nu}(w)$ to be the Laurent series 
$$ \Theta_{\Bp,\beta}^{\mu,\nu}(w) := \Omega_{\Bp,\beta}^{\mu,\nu}(w) \times \prod_{i\in I} \Bp_i(-w)^{\langle\varpi_i^{\vee}, \beta \rangle}\in (Y_{\mu}^{\leq}(\Glie)_{-\beta} \otimes Y_{\nu}^{\geq}(\Glie)_{\beta})((w^{-1})). $$
Define the {\it Theta series} $\Theta_{\Bp}^{\mu,\nu}(w)$ to be the formal sum
$$ \Theta_{\Bp}^{\mu,\nu}(w) := \sum_{\beta\in \BQ_+} \Theta_{\Bp,\beta}^{\mu,\nu}(w) \in \prod_{\beta\in \BQ_+} (Y_{\mu}^{\leq}(\Glie)_{-\beta} \otimes Y_{\nu}^{\geq}(\Glie)_{\beta})((w^{-1})). $$
\end{defi}
It follows  that $\Omega_{\Bp,0}^{\mu,\nu}(w) = 1$ and $\Omega_{\Bp}^{\mu,\nu}(w)$ is the sum over $\beta \in \BQ_+$ of the $\Omega_{\Bp,\beta}^{\mu,\nu}(w)$.

\begin{rem}
Definition \ref{def: Theta series Yangian} makes sense for the Drinfeld--Cartan series $\xi_i(w)$:
$$ \Delta_{\mu,\nu}(\xi_i(w)) = (1\otimes \xi_i(w)) \times \sum_{\beta \in \BQ_+} \Xi_{i,\beta}^{\mu,\nu}(w) \times (\xi_i(w) \otimes 1)   $$
where $\Xi_{i,\beta}^{\mu,\nu}(w)$ is a power series in $w^{-1}$ with coefficients in $Y_{\mu}^{\leq}(\Glie)_{-\beta} \otimes Y_{\nu}^{\geq}(\Glie)_{\beta}$ and $\Xi_{i,0}^{\mu,\nu}(w) = 1$. As a reformulation of \cite[Lemma 7.1]{HZ} we have
$$ \Xi_{i,\alpha_j}^{\mu,\nu}(w) = -b_{ij} x_j^-(w+d_{ij}) \otimes x_j^+(w+d_{ij}). $$
Therefore the power series $\Xi_{i,\beta}^{\mu,\nu}(w)$ is in general not polynomial. In contrast, the Laurent series $\Theta_{\Bp,\beta}^{\mu,\nu}(w)$ is always polynomial by Theorem \ref{thm: Yangian poly Theta}.
\end{rem}

\subsection{Theta series as module isomorphisms} 
In this subsection we interpret Theta series as associators for completed triple tensor product modules. 

\begin{defi} \label{defi: second completion tensor product}
Let $M$ be a $Y_{\mu}(\Glie)$-module and $N$ be a $Y_{\nu}(\Glie)$-module, both root graded in the sense of Definition \ref{defi: root graded Yangian}. Define the completed tensor product to be
$$ M \otimes_w N := \sum_{\alpha,\beta\in\BQ} \left( \prod_{\gamma \in \BQ_+} (M_{(\alpha-\gamma)} \otimes N_{(\beta+\gamma)})((w^{-1})) \right) \subset \prod_{\alpha,\beta \in \BQ} (M_{(\alpha)} \otimes N_{(\beta)})((w^{-1})).$$
\end{defi}

The completed tensor product $Y_{\mu}(\Glie) \otimes_w Y_{\nu}(\Glie)$ associated to regular representations is an algebra that contains $Y_{\mu}(\Glie) \otimes Y_{\nu}(\Glie)[w]$ as a subalgebra. In the above definition, $M \otimes_w N$ is naturally a module over $Y_{\mu}(\Glie) \otimes_w Y_{\nu}(\Glie)$, and it is a completion, simultaneously at the level of algebras and at the level of representations, of $M\otimes N[w]$ viewed as a $Y_{\mu}(\Glie) \otimes Y_{\nu}(\Glie)[w]$-module by scalar extension.

Notice that the Theta series $\Theta_{\Bp}^{\mu,\nu}(w)$ of Definition \ref{def: Theta series Yangian} belongs to $Y_{\mu}(\Glie) \otimes_w Y_{\nu}(\Glie)$. 

\begin{rem}    \label{rem: comparison two completions}
 Compared to the completion $(M\otimes N)^w$ of $M\otimes N[w]$ in Example \ref{example: motivation root graded Yangian}(iii),  the new completion $M \otimes_w N$ has the advantage that it still contains $M^w \otimes N$ and the linear automorphism $D_{\Bp,w}^M \otimes \mathrm{Id}_N$ of $M^w \otimes N$ extends to $M \otimes_w N$; similar statements hold for $M \otimes N^w$. 
\end{rem}

Let $M$ and $N$ be root graded modules over the shifted Yangians $Y_{\mu}(\Glie)$ and $Y_{\nu}(\Glie)$ respectively and let $\Bp$ be a polynomial $\ell$-weight of coweight $\zeta$. Both $M \otimes N[w]$ and $M\otimes_w N$ are modules over $Y_{\mu}(\Glie) \otimes Y_{\nu}(\Glie)[w]$ by scalar extension, with the latter being a completion of the former. By Example \ref{example: tensor prefund Yangian}, the pullback of $M\otimes N[w]$ along the following algebra homomorphism gives $Y_{\mu+\zeta+\nu}(\Glie)$-module $(M\otimes L(\Bp)_w) \otimes N$:
 $$ (G^{\mu+\zeta}_{\Bp,w} \otimes \mathrm{Id}) \Delta_{\mu+\zeta, \nu}: Y_{\mu+\nu+\zeta}(\Glie) \longrightarrow Y_{\mu}(\Glie) \otimes Y_{\nu}(\Glie)[w]. $$ 
The pullback of $M \otimes_w N$ along the same algebra homomorphism is a module structure over $Y_{\mu+\nu+\zeta}(\Glie)$, denoted by $\overline{(M \otimes L(\Bp)_w) \otimes N}$ to mean a completion of the triple tensor product modules $(M \otimes L(\Bp)_w) \otimes N$. Similarly, on the same underlying space $M \otimes_w N$ the completions of the other triple tensor product modules are defined:
\begin{gather*}
\overline{M \otimes (L(\Bp)_w \otimes N)}, \quad \overline{M \otimes (N \otimes L(\Bp)_w)}, \quad \overline{(M\otimes N) \otimes L(\Bp)_w}, \\
\overline{(L(\Bp)_w \otimes M) \otimes N}, \quad \overline{L(\Bp)_w \otimes (M \otimes N)}.
\end{gather*}
Note that the completed triple tensor product module $\overline{(M\otimes N) \otimes L(\Bp)_w}$ here and the one $(M \otimes N) \tilde{\otimes} L(\Bp)_w$ defined before Theorem \ref{thm: Yangian T-series} are in general two different completions of the module $(M \otimes N) \otimes L(\Bp)_w$.

\begin{theorem}  \label{thm:Yangian associator}
Let $M$ and $N$ be root graded modules over shifted Yangians $Y_{\mu}(\Glie)$ and $Y_{\nu}(\Glie)$ respectively and let $\Bp$ be a polynomial $\ell$-weight of coweight $\zeta$. The action of $\Theta_{\Bp}^{\mu,\nu}(w) \in Y_{\mu}(\Glie) \otimes_w Y_{\nu}(\Glie)$ on $M \otimes_w N$ defines an isomorphism of $Y_{\mu+\zeta+\nu}(\Glie)$-modules
$$ \Theta_{\Bp}^{\mu,\nu}(w):  \overline{(M \otimes L(\Bp)_w) \otimes N} \longrightarrow \overline{M \otimes (L(\Bp)_w \otimes N)}.   $$
\end{theorem}
\begin{proof}
By Remark \ref{rem: comparison two completions}, the linear automorphisms $D_{\Bp,w}^M \otimes \mathrm{Id}_N$ of $M^w \otimes N$ and $\mathrm{Id}_M \otimes D_{\Bp,w}^N$ of $M \otimes N^w$ extend to mutually commuting linear automorphisms of $M \otimes_w N$, denoted by $D_{\Bp}^M \otimes_w \mathrm{Id}_N$ and $\mathrm{Id}_M \otimes_w D_{\Bp}^N$ respectively. Set 
$$D_{\Bp} := (\mathrm{Id}_M \otimes_w D_{\Bp}^N) \circ (D_{\Bp}^M \otimes_w \mathrm{Id}_N) \in \mathrm{Aut}(M \otimes_w N).$$ 
From the definition of root grading we have the following identities of linear endomorphisms of $M \otimes_w N$ for $g(w) \in (Y_{\mu}(\Glie)_{\alpha} \otimes Y_{\nu}(\Glie)_{\beta})((w^{-1}))$:
\begin{align*}
(D_{\Bp}^M \otimes_w \mathrm{Id}_N) \circ g(w) &= \prod_{i\in I}\Bp_i(-w)^{-\langle \varpi_i^{\vee}, \alpha\rangle} \times g(w) \circ (D_{\Bp}^M \otimes_w \mathrm{Id}_N), \\
(\mathrm{Id}_M \otimes_w D_{\Bp}^N) \circ g(w) &=  \prod_{i\in I}\Bp_i(-w)^{-\langle \varpi_i^{\vee}, \beta\rangle} \times g(w)\circ (\mathrm{Id}_M \otimes_w D_{\Bp}^N), \\
D_{\Bp} \circ g(w) &= \prod_{i\in I}\Bp_i(-w)^{-\langle \varpi_i^{\vee},\alpha+ \beta\rangle}  \times g(w)\circ D_{\Bp}.
\end{align*}
Setting $g(w)$ to be $\Omega_{\Bp,\beta}^{\mu,\nu}(w)$ and then to be $\Omega_{\Bp}^{\mu,\nu}(w)$ we have by Definition \ref{def: Theta series Yangian}
$$(D_{\Bp}^M \otimes_w \mathrm{Id}_N) \circ \Omega_{\Bp}^{\mu,\nu}(w) \circ (D_{\Bp}^M \otimes_w \mathrm{Id}_N)^{-1} = \Theta_{\Bp}^{\mu,\nu}(w) \in \mathrm{Aut}(M \otimes_w N). $$

By definition $S_{\Bp}^{\mu}(w) \otimes 1,\ 1 \otimes S_{\Bp}^{\nu}(w)$ and $\Delta_{\mu,\nu}(S_{\Bp}^{\mu+\nu}(w))$ are power series in $w^{-1}$ with coefficients in $\sum_{\beta\in \BQ_+} Y_{\mu}(\Glie)_{-\beta} \otimes Y_{\nu}(\Glie)_{\beta}$. So they are elements of the completed tensor product $Y_{\mu}(\Glie) \otimes_w Y_{\nu}(\Glie)$ and act on $M \otimes_w N$.
Adapting the proof of Theorem \ref{thm: Yangian T-series} to the present situation, we obtain that the following linear automorphisms of $M\otimes_w N$ are $Y_{\mu+\nu+\zeta}(\Glie)$-module isomorphisms:
\begin{align*}
&A= D_{\Bp} \circ\Delta_{\mu,\nu}(S_{\Bp}^{\mu+\nu}(w)): \overline{L(\Bp)_w \otimes (M\otimes N)} \longrightarrow \overline{(M\otimes N) \otimes L(\Bp)_w}, \\
& B= (\mathrm{Id}_M \otimes_w D_{\Bp}^N)\circ (1\otimes S_{\Bp}^{\nu}(w)):  \overline{M\otimes (L(\Bp)_w \otimes N)} \longrightarrow  \overline{M\otimes (N \otimes L(\Bp)_w)},  \\
&C= (D_{\Bp}^M \otimes_w \mathrm{Id}_N)\circ (S_{\Bp}^{\mu}(w) \otimes 1): \overline{(L(\Bp)_w \otimes M) \otimes N} \longrightarrow \overline{(M \otimes L(\Bp)_w) \otimes N}.
\end{align*}
Since $L(\Bp)_w$ is trivial, by Theorem \ref{thm: trivial associativity Yangian} identity maps are module isomorphisms
$$ \overline{(L(\Bp)_w \otimes M) \otimes N} = \overline{L(\Bp)_w \otimes (M \otimes N)}, \quad \overline{M \otimes (N \otimes L(\Bp)_w)} = \overline{(M\otimes N) \otimes L(\Bp)_w}.  $$

In summary, we have the following diagram of $Y_{\mu+\nu+\zeta}(\Glie)$-module isomorphisms:
\begin{gather*} 
 \xymatrixcolsep{5pc} \xymatrix{
\overline{L(\Bp)_w \otimes (M\otimes N)}  \ar[d]_{\mathrm{Id} } \ar[r]^A &  \overline{(M\otimes N) \otimes L(\Bp)_w}  \ar[r]^{\mathrm{Id}} & \overline{M \otimes (N \otimes L(\Bp)_w)}   \\
\overline{(L(\Bp)_w \otimes M) \otimes N} \ar[r]^C  &  \overline{(M \otimes L(\Bp)_w) \otimes N}        & \overline{M\otimes (L(\Bp)_w \otimes N)} \ar[u]_B  } 
\end{gather*}
The composite map $B^{-1} \circ A \circ C^{-1} \in \mathrm{Aut}(M\otimes_w N)$ defines a $Y_{\mu+\nu+\zeta}(\Glie)$-module isomorphism from $\overline{(M \otimes L(\Bp)_w) \otimes N}$ to $\overline{M \otimes (L(\Bp)_w \otimes N)}$. Based on the factorization of $\Delta_{\mu,\nu}(S_{\Bp}^{\mu+\nu}(w))$ in Definition \ref{def: Theta series Yangian} we have:
\begin{align*}
B^{-1} \circ A \circ C^{-1} &= (1\otimes S_{\Bp}^{\nu}(w)^{-1}) \circ (\mathrm{Id}_M \otimes_w D_{\Bp}^N)^{-1} \circ  D_{\Bp} \circ \Delta_{\mu,\nu}(S_{\Bp}^{\mu+\nu}(w)) \\
&\quad \circ (S_{\Bp}^{\mu}(w)^{-1} \otimes 1) \circ (D_{\Bp}^M \otimes_w \mathrm{Id}_N)^{-1} \\
&= (1\otimes S_{\Bp}^{\nu}(w)^{-1})\circ (\mathrm{Id}_M \otimes_w D_{\Bp}^N)^{-1} \circ (\mathrm{Id}_M \otimes_w D_{\Bp}^N) \circ (D_{\Bp}^M \otimes_w \mathrm{Id}_N) \\
&\quad  \circ  (1\otimes S_{\Bp}^{\nu}(w)) \Omega_{\Bp}^{\mu,\nu}(w) (S_{\Bp}^{\mu}(w) \otimes 1)  \circ (S_{\Bp}^{\mu}(w)^{-1} \otimes 1)\circ (D_{\Bp}^M \otimes_w \mathrm{Id}_N)^{-1} \\
&= (D_{\Bp}^M \otimes_w \mathrm{Id}_N) \circ \Omega_{\Bp}^{\mu,\nu}(w) \circ (D_{\Bp}^M \otimes_w \mathrm{Id}_N)^{-1} = \Theta_{\Bp}^{\mu,\nu}(w).
\end{align*}
This proves that $\Theta_{\Bp}^{\mu,\nu}(w)$ gives the desired module isomorphism.
\end{proof}
If both modules $M$ and $N$ are positive root graded or both modules are negative root graded, then the two completions coincide
$$ (M\otimes N)^w = M \otimes_w N = \sum_{\alpha,\beta\in \BQ} (M_{(\alpha)} \otimes N_{(\beta)}) ((w^{-1})). $$
The proof of Theorem \ref{thm:Yangian associator}  gives a factorization of T-operators in $\mathrm{Aut}(M \otimes_w N)$:
\begin{equation}  \label{rel: factorization T Yangian}
T_{\Bp}^{M\otimes N}(w) = (\mathrm{Id}_M \otimes T_{\Bp}^N(w)) \circ \Theta_{\Bp}^{\mu,\nu}(w) \circ (T_{\Bp}^M(w) \otimes \mathrm{Id}).
\end{equation}

\subsection{Intertwining property and polynomiality}
In this subsection we establish polynomiality for the Theta series, as a consequence of an intertwining property. 

The inverse $\Theta_{\Bp}^{\mu,\nu}(w)^{-1}$ of a Theta series is well-defined in the completed algebra $Y_{\mu}(\Glie) \otimes_w Y_{\nu}(\Glie)$ because of the initial condition $\Theta_{\Bp,0}^{\mu,\nu}(w) = 1$. As in Definition \ref{def: Theta series Yangian} let $\widetilde{\Theta}_{\Bp,\beta}^{\mu,\nu}(w) \in (Y_{\mu}^{\leq}(\Glie)_{-\beta} \otimes Y_{\nu}^{\geq}(\Glie)_{\beta})((w^{-1}))$ denote its $\beta$-component for $\beta \in \BQ_+$.

\begin{theorem}  \label{thm: Yangian poly Theta}
Let $\mu, \nu$ be coweights and $\Bp$ be a polynomial $\ell$-weight of coweight $\zeta$. Then for $\beta \in \BQ_+$ the Laurent series $\Theta_{\Bp,\beta}^{\mu,\nu}(w)$ and $\widetilde{\Theta}_{\Bp,\beta}^{\mu,\nu}(w)$ are polynomials in $w$ whose degrees are bounded by $\langle\zeta, \beta \rangle$.
\end{theorem}
It follows from Definition \ref{def: Theta series Yangian} that for $\Bp = \Psi_{i,0}$ and $\beta \in \BQ_+$ the Omega series $\Omega_{\Bp,\beta}^{\mu,\nu}(w)$ is a polynomial in $w^{-1}$ of degree bounded by $\langle \varpi_i^{\vee}, \beta \rangle$.
\begin{proof}
By definition we have 
$$ \sum_{\beta, \gamma \in \BQ_+} \widetilde{\Theta}_{\Bp,\beta}^{\mu,\nu}(w)  \Theta_{\Bp,\gamma}^{\mu,\nu}(w) = 1. $$
So $\widetilde{\Theta}_{\Bp,\beta}^{\mu,\nu}(w)$ is a $\BZ$-linear combination of the monomials $\Theta_{\Bp,\gamma_1}^{\mu,\nu}(w) \Theta_{\Bp,\gamma_2}^{\mu,\nu}(w) \cdots \Theta_{\Bp,\gamma_s}^{\mu,\nu}(w)$ where $s \geq 0$ and $\gamma_1, \gamma_2, \cdots, \gamma_s \in \BQ_+$ are such that $\beta = \gamma_1+\gamma_2+\cdots+\gamma_s$. If each $\Theta_{\Bp,\gamma_j}^{\mu,\nu}(w)$ is a polynomial of degree bounded by $\langle \zeta, \gamma_j\rangle$, each such monomial and $\widetilde{\Theta}_{\Bp,\beta}^{\mu,\nu}(w)$ are polynomials of degree bounded by $\langle \zeta, \gamma_1+\gamma_2+\cdots+\gamma_s\rangle = \langle \zeta, \beta\rangle$. 

It suffices to establish polynomiality of Theta series. This is done in several steps.

\medskip

\noindent {\bf Step 1: intertwining property.}
In Theorem \ref{thm:Yangian associator} let us take $M$ and $N$ to be regular representations of shifted Yangians. Then both completed triple tensor product modules have $Y_{\mu}(\Glie) \otimes_w Y_{\nu}(\Glie)$ as the common underlying space and  an element $x$ of $Y_{\mu+\zeta+\nu}(\Glie)$ acts on these two modules as left multiplications by
$$(G_{\Bp,w}^{\mu+\zeta} \otimes \mathrm{Id}) \circ \Delta_{\mu+\zeta, \nu}(x), \quad (\mathrm{Id} \otimes F_{\Bp,w}^{\zeta+\nu})  \circ \Delta_{\mu,\zeta+ \nu}(x) $$
respectively. As a consequence of module morphism, we have the following {\it intertwining property} in the algebra $Y_{\mu}(\Glie) \otimes_w Y_{\nu}(\Glie)$ for $x \in Y_{\mu+\zeta+\nu}(\Glie)$:
\begin{equation}  \label{intertwining Theta}
 \Theta_{\Bp}^{\mu,\nu}(w) \times (G_{\Bp,w}^{\mu+\zeta} \otimes \mathrm{Id}) \circ \Delta_{\mu+\zeta, \nu}(x) = (\mathrm{Id} \otimes F_{\Bp,w}^{\zeta+\nu})  \circ \Delta_{\mu,\zeta+ \nu}(x) \times \Theta_{\Bp}^{\mu,\nu}(w).
\end{equation}

\medskip

\noindent {\bf Step 2: case $\mu = \nu = 0$.} We specialize Eq.\eqref{intertwining Theta} to the case $\mu = \nu = 0$ and $x = x_{i,0}^+ \in Y_{\zeta}(\Glie)$ for $i \in I$. For the left-hand side, by Lemma \ref{lem: Yangian coproduct estimation} we can write
$$  \Delta_{\zeta,0}(x_{i,0}^+) = x_{i,0}^+ \otimes 1 + \sum_{\gamma \in \mathcal{Q}_i, v\in \mathcal{C}_{\gamma}} v \otimes \varphi_{\gamma}(v). $$
Here  $\mathcal{Q}_i$ is a finite subset of $\BQ_+$ depending on $i$ and for each $\gamma \in \mathcal{Q}_i$ we have a map $\varphi_{\gamma}$ from a finite subset $\mathcal{C}_{\gamma}$ of $\sum_{t\in \BZ}  Y_{\zeta}^-(\Glie)_{-\gamma} \xi_{i,t}$ to $Y_0^+(\Glie)_{\gamma+\alpha_i}$. For the right-hand side, in the commutative diagram \eqref{rel: coproduct vs shift} take $(\mu,\nu,\epsilon,\eta)$ to be $(0,\zeta,0,-\zeta)$ and put $x_{i,0}^+$ at the top-left corner. Then the element at the bottom-left corner is $x_{i,0}^+$, a primitive element of the Hopf algebra $Y_0(\Glie)$. So the element at the top-right corner is
\begin{align*}
\Delta_{0,\zeta}(x_{i,0}^+) = x_{i,0}^+ \otimes 1 + 1 \otimes x_{i,0}^+.
\end{align*}
From Example \ref{example: tensor prefund Yangian} observe that $G_{\Bp,w}^{\zeta}(x_{i,0}^+) = x_{i,0}^+$, each $G_{\Bp,w}^{\zeta}(v)$ for $v \in \mathcal{C}_{\gamma}$ is a polynomial in $w$ of degree bounded by $\langle \zeta, \gamma+\alpha_i\rangle$, and $F_{\Bp,w}^{\zeta}(x_{i,0}^+)$ is a polynomial in $w$ of degree $\langle \zeta, \alpha_i\rangle$. Projecting Eq.\eqref{intertwining Theta} to the component $(Y_0(\Glie)_{\alpha_i-\beta} \otimes Y_0(\Glie)_{\beta})((w^{-1}))$ for $\beta \in \BQ_+$ we obtain the following commutation relation
\begin{equation}   \label{rel: Theta i x i}
\begin{split}
[x_{i,0}^+ \otimes 1, \Theta_{\Bp, \beta}^{0,0}(w)] &=  \sum_{\gamma \in \mathcal{Q}_i, v\in \mathcal{C}_{\gamma}} \Theta_{\Bp, \beta-\gamma-\alpha_i}^{0,0}(w)  (G_{\Bp,w}^{\zeta}(v) \otimes \varphi_{\gamma}(v))  \\
&\qquad - (1 \otimes F_{\Bp,w}^{\zeta}(x_{i,0}^+)) \Theta_{\Bp,\beta-\alpha_i}^{0,0}(w). 
\end{split}
\end{equation}
By convention $\Theta_{\Bp,\beta}^{0,0}(w) := 0$ if $\beta \notin \BQ_+$. The above summation is finite.

\medskip

\noindent {\bf Step 3: polynomiality for $\Theta_{\Bp,\beta}^{0,0}(w)$.} We prove that the Laurent series $\Theta_{\Bp,\beta}^{0,0}(w)$ is a polynomial in $w$ of degree bounded by $\langle \zeta, \beta\rangle$ by induction on the height of $\beta$ defined as $h(\beta) := \langle \sum_{i\in I} \varpi_i^{\vee}, \beta \rangle \in \BN$, namely the sum of the coefficients of $\beta$ as an $\BN$-linear combination of the simple roots. 

The initial case $h(\beta) = 0$ follows from $\Theta_{\Bp,0}^{0,0}(w) = 1 = \Omega_{\Bp,0}^{0,0}(w)$. 

Suppose $h(\beta) > 0$.  By induction hypothesis, $\Theta_{\Bp,\beta-\gamma-\alpha_i}^{0,0}(w)$ is a polynomial in $w$ of degree bounded by $\langle \zeta, \beta-\gamma-\alpha_i\rangle$ for $\gamma \in \BQ_+$. In view of the degree estimations of the polynomials $G_{\Bp,w}^{\zeta}(v)$ and $F_{\Bp,w}^{\zeta}(x_{i,0}^+)$ at Step 2,  each term at the right-hand side of Eq.\eqref{rel: Theta i x i} is a polynomial in $w$ of degree bounded by $\langle \zeta, \beta\rangle$. For $n \in \BZ$, let $f_n \in Y_0(\Glie)_{-\beta} \otimes Y_0(\Glie)_{\beta}$ be the coefficient of $w^n$ in the Laurent series $\Theta_{\Bp,\beta}^{0,0}(w)$. Then 
$$ [x_{i,0}^+ \otimes 1, f_n] = 0 \quad \textrm{for $i \in I$ and ($n > \langle \zeta, \beta\rangle$ or $n < 0$)}. $$
It suffices to show that the joint kernel of the linear maps $\mathrm{ad}_{x_{i,0}^+}: Y_0(\Glie)_{-\beta} \longrightarrow Y_0(\Glie)$ for $i \in I$ is zero, so that $f_n = 0$ for such $n$. Indeed, the adjoint representation of $\Glie$ on the ordinary Yangian $Y_0(\Glie)$ is integrable. If $v \neq 0$ belongs to the joint kernel, then $v$ must be a vector of highest weight $-\beta$, which forces $-\beta$ to be a dominant weight. Together with $\beta \in \BQ_+$ and the fact that the symmetrized Cartan matrix $(b_{ij})_{i,j\in I}$ is positive definite, we must have $\beta = 0$, contradicting $h(\beta) > 0$. 

\medskip

\noindent {\bf Step 4: polynomiality for $\Theta_{\Bp,\beta}^{\mu,\nu}(w)$.} We claim that
\begin{equation} \label{rel: shift Theta}
(\iota_{\epsilon,0}^{\mu} \otimes \iota_{0,\eta}^{\nu})(\Theta_{\Bp,\beta}^{\mu,\nu}(w)) = \Theta_{\Bp,\beta}^{\mu+\epsilon,\nu+\eta}(w)\quad \textrm{for $\epsilon$ and $\eta$ antidominant}. 
\end{equation}
It suffices to prove the identity with the $\Theta_{\Bp,\beta}(w)$ replaced by $\Omega_{\Bp}(w)$. 
In the commutative diagram \eqref{rel: coproduct vs shift}, put $S_{\Bp}^{\mu+\nu}(w)$ at the top-left corner so that at the top-right corner we have $\Delta_{\mu,\nu}(S_{\Bp}^{\mu+\nu}(w))$. From the proof of Lemma \ref{lem: coproduct S} and the fact that shifted homomorphisms preserve GKLO series we obtain that shifted homomorphisms also preserve S-series. This implies that the series at the bottom-right corner is $\Delta_{\mu+\epsilon,\nu+\eta}(S_{\Bp}^{\mu+\nu+\epsilon+\eta}(w))$. Comparing the factorizations of $\Delta(S_{\Bp}(w))$ at the bottom-right corner and at the top-right corner gives the desired identity for Omega series.  

Choose antidominant coweights $\epsilon$ and $\eta$ such that $\mu + \epsilon$ and $\nu + \eta$ are antidominant. Then the polynomiality of $\Theta_{\Bp,\beta}^{\mu+\epsilon,\nu+\eta}(w)$ follows from that of $\Theta_{\Bp,\beta}^{0,0}(w)$ and
$$ \Theta_{\Bp,\beta}^{\mu+\epsilon,\nu+\eta}(w) = (\iota_{\mu+\epsilon,0}^0 \otimes \iota_{0,\nu+\eta}^0)(\Theta_{\Bp,\beta}^{0,0}(w)). $$
This together with Eq.\eqref{rel: shift Theta} in turn implies the polynomiality of $\Theta_{\Bp,\beta}^{\mu,\nu}(w)$ because the shifted homomorphisms $\iota_{\epsilon,0}^{\mu}$ and $\iota_{0,\eta}^{\nu}$ are injective.
\end{proof}

Step 2 of the above proof is similar to \cite[\S\S 4.2, 4.5]{GTLW} which deduces uniqueness, existence and rationality of the triangular part of the universal R-matrix for the ordinary Yangian from an intertwining property. In our case Eq.\eqref{rel: Theta i x i} implies polynomiality and uniqueness of Theta series, but not the existence. 

\begin{example}  \label{example: Yangian Theta sl2}
Fix $\Bp = \Psi_{j,0}$ and $\beta = n \alpha_j$ with $n \in \BZ_{>0}$. For $i \in I$ and $\gamma \in \BQ_+$, in order that $n\alpha_j - \gamma - \alpha_i \in \BQ_+$ we must have $j = i$ and $\gamma \in \BN \alpha_j$. The following coproduct estimation can be deduced from the proof of \cite[Theorem 4.12]{coproduct}:
$$\Delta_{\varpi_j^{\vee},0}(x_{j,0}^+) \equiv x_{j,0}^+ \otimes 1 +  \xi_{j,-1} \otimes x_{j,0}^+ + 1 \otimes x_{j,1}^+ \ \mathrm{mod}. \sum_{\gamma \notin \BN \alpha_j} Y_{\varpi_j^{\vee}}(\Glie) \otimes Y_0(\Glie)_{\gamma}. $$
The recursion formula \eqref{rel: Theta i x i} becomes 
$$ [x_{i,0}^+ \otimes 1, \Theta_{\Bp,n\alpha_j}^{0,0}(w)] = \delta_{ij} [\Theta_{\Bp,(n-1)\alpha_j}^{0,0}(w), 1 \otimes (x_{j,1}^+ - w x_{j,0}^+)] + \delta_{ij} \Theta_{\Bp,(n-1)\alpha_j}^{0,0}(w) (\xi_{j,0} \otimes x_{j,0}^+). $$
It has a unique solution in $Y_0(\Glie)_{-n\alpha_j} \otimes Y_0(\Glie)_{n\alpha_j}[w]$, namely 
$$ \Theta_{\Bp,n\alpha_j}^{0,0}(w) = \frac{1}{n!}  (x_{j,0}^- \otimes x_{j,0}^+)^n. $$
The same formula holds for $\Theta_{\Bp,n\alpha_j}^{\mu,\nu}(w)$. In the rank-one case $\Glie = sl_2$,  we have the following coproduct formula in the Yangian $Y_0(sl_2)$, which is simpler than the coproduct of $\xi_1(z)$ in \cite[Definition 2.24]{Molev} and \cite[(6.9)]{coproduct}.
$$ \Delta(S_1(z)) = (1\otimes S_1(z)) \exp(-z^{-1} x_{1,0}^- \otimes x_{1,0}^+) (S_1(z) \otimes 1). $$
\end{example}

\begin{cor}  \label{cor: nontrivial associator Yangian}
Let $M$ be a $Y_{\mu}(\Glie)$-module, $N$ be a $Y_{\nu}(\Glie)$-module, and $\Bp$ be a polynomial $\ell$-weight of coweight $\zeta$. If either $M$ is positive root graded or $N$ is negative root graded, then the action of $\Theta_{\Bp}^{\mu,\nu}(w)$ on $M \otimes_w N$ restricts to a module isomorphism 
$$\Theta_{\Bp}^{\mu,\nu}(w):  (M \otimes L(\Bp)_w) \otimes N \longrightarrow M \otimes (L(\Bp)_w \otimes N).  $$
\end{cor}
\begin{proof}
Given $v_1 \in M$ and $v_2 \in N$ we have for all but finitely many $\beta \in \BQ_+$:
 $$\Theta_{\Bp,\beta}^{\mu,\nu}(w) (v_1 \otimes v_2) = 0 = \widetilde{\Theta}_{\Bp,\beta}^{\mu,\nu}(w)(v_1 \otimes v_2).$$
 Both $\Theta_{\Bp}^{\mu,\nu}(w) (v_1 \otimes v_2)$ and $\Theta_{\Bp}^{\mu,\nu}(w)^{-1}(v_1 \otimes v_2)$ are polynomial in $w$. So $\Theta_{\Bp}^{\mu,\nu}(w)$ defines a $\BC[w]$-linear automorphism of $M \otimes N[w]$. That the linear map is a module morphism follows directly from Theorem \ref{thm:Yangian associator} or the intertwining property \eqref{intertwining Theta}.
\end{proof}
\subsection{Multiplicativity of Theta series} In this subsection we factorize the Theta series $\Theta_{\Bp}^{\mu,\nu}(z)$ with respect to the polynomial $\ell$-weight $\Bp$.  

\begin{prop}   \label{prop: multip}
Let $\mu, \nu$ be coweights and $\Bm, \Bn$ be polynomial $\ell$-weight.
\begin{itemize}
    \item[(i)] We have an algebra homomorphism $\psi_{\Bm,w}^{\mu,+}: Y_{\mu}^{\geq}(\Glie) \longrightarrow Y_{\mu}^{\geq}(\Glie)[w]$ defined by
$$ x_i^+(z) \mapsto \langle \Bm_i(z-w) x_i^+(z) \rangle_+,\quad \xi_i(z) \mapsto \xi_i(z). $$
     \item[(ii)] We have an algebra homomorphism $\psi_{\Bm,w}^{\mu,-}: Y_{\mu}^{\leq}(\Glie) \longrightarrow Y_{\mu}^{\leq}(\Glie)[w]$ defined by
$$ x_i^-(z) \mapsto \langle \Bm_i(z-w) x_i^-(z) \rangle_+,\quad \xi_i(z) \mapsto \xi_i(z). $$
     \item[(iii)] We have the following multiplicativity of Theta series: 
     $$ \Theta_{\Bm\Bn}^{\mu,\nu}(w) = (\mathrm{Id} \otimes \psi^{\nu,+}_{\Bn,w})(\Theta_{\Bm}^{\mu,\nu}(w)) \times  (\psi^{\mu,-}_{\Bm,w} \otimes \mathrm{Id})(\Theta_{\Bn}^{\mu,\nu}(w)).  $$
\end{itemize}
\end{prop}
\begin{proof}
For (i)--(ii), apply Theorem \ref{thm: Yangian T-series} to the regular representation $V = Y_{\mu}(\Glie)$. We have a module morphism $T_{\Bm}^V(w): L(\Bm)_w \tilde{\otimes} V \longrightarrow V \tilde{\otimes} L(\Bm)_w$. The first and third relations of Eq.\eqref{intertwining T} imply that $T_{\Bm}^V(w)$ sends $Y_{\mu}^{\leq}(\Glie)$ to $Y_{\mu}^{\leq}(\Glie)[w]$ restricts to the desired algebra homomorphism $\psi_{\Bm,w}^{\mu,-}$. Similarly $T_{\Bm}^V(w)^{-1}$ restricts to $\psi_{\Bm,w}^{\mu,+}$.

For (iii), let $\mathrm{Ad}_{S_{\Bm}^{\mu}(w)}$ denote the conjugation on the algebra $Y_{\mu}(\Glie)[[w^{-1}]]$ by the invertible power series $S_{\Bm}^{\mu}(w)$. 
By Eq.\eqref{def: S series one-dim Yangian} and Definition \ref{def: Theta series Yangian} we have 
\begin{align*}
&\Delta_{\mu,\nu}(S_{\Bm\Bn}^{\mu+\nu}(w)) = \Delta_{\mu,\nu}(S_{\Bm}^{\mu+\nu}(w) S_{\Bn}^{\mu+\nu}(w)) = \Delta_{\mu,\nu}(S_{\Bm}^{\mu+\nu}(w)) \Delta_{\mu,\nu}(S_{\Bn}^{\mu+\nu}(w)) \\
&= (1\otimes S_{\Bm}^{\nu}(w)) \Omega_{\Bm}^{\mu,\nu}(w) (S_{\Bm}^{\mu}(w) \otimes 1) (1\otimes S_{\Bn}^{\nu}(w)) \Omega_{\Bn}^{\mu,\nu}(w) (S_{\Bn}^{\mu}(w) \otimes 1) \\
&= (1\otimes S_{\Bm\Bn}^{\nu}(w)) \times (\mathrm{Id} \otimes \mathrm{Ad}_{S_{\Bn}^{\nu}(w)}^{-1})(\Omega_{\Bm}^{\mu,\nu}(w))  (\mathrm{Ad}_{S_{\Bm}^{\mu}(w)} \otimes \mathrm{Id})(\Omega_{\Bn}^{\mu,\nu}(w)) \times  (S_{\Bm\Bn}^{\mu}(w) \otimes 1).
\end{align*}
By uniqueness we obtain the factorization formula for Omega series:
$$ \Omega_{\Bm\Bn}^{\mu,\nu}(w) = (\mathrm{Id} \otimes \mathrm{Ad}_{S_{\Bn}^{\nu}(w)}^{-1})(\Omega_{\Bm}^{\mu,\nu}(w)) \times (\mathrm{Ad}_{S_{\Bm}^{\mu}(w)} \otimes \mathrm{Id})(\Omega_{\Bn}^{\mu,\nu}(w)).  $$
From Eqs.\eqref{comm: S x general}--\eqref{intertwining T} observe that $\psi_{\Bm,w}^{\mu,\pm}$ are related to $\mathrm{Ad}_{S_{\Bm}^{\mu}(w)}^{\mp 1}$ as follows:
\begin{align*}
& \mathrm{Ad}_{S_{\Bm}^{\mu}(w)}^{-1}(x_{i,m}^+) = \frac{1}{\Bm_i(-w)} \psi_{\Bm,w}^{\mu,+}(x_{i,m}^+),\quad \mathrm{Ad}_{S_{\Bm}^{\mu}(w)}^{-1}(\xi_{i,p}) = \psi_{\Bm,w}^{\mu,+}(\xi_{i,p}), \\
& \mathrm{Ad}_{S_{\Bm}^{\mu}(w)}(x_{i,m}^-) = \frac{1}{\Bm_i(-w)}\psi_{\Bm,w}^{\mu,-}(x_{i,m}^-),\quad \mathrm{Ad}_{S_{\Bm}^{\mu}(w)}(\xi_{i,p}) = \psi_{\Bm,w}^{\mu,-}(\xi_{i,p}).
\end{align*}
These additional factors contribute to the passage from Omega series to Theta series in Definition \ref{def: Theta series Yangian}.
\end{proof}
The computation of Theta series is reduced to the prefundamental case $\Bp = \Psi_{i,a}$.
\subsection{Coefficients of Theta series}
For $\mu$ a coweight, recall the subalgebra $Y_{\mu}^{\pm}(\Glie)$ of the shifted Yangian $Y_{\mu}(\Glie)$ generated by the $x_{i,m}^{\pm}$ for $i \in I$ and $m \in \BN$. 
\begin{prop} \label{prop: purity Yangian} 
Let $\mu, \nu$ be coweights, $\beta \in \BQ_+$ and
 $\Bp$ be a polynomial $\ell$-weight. The polynomial $\Theta_{\Bp,\beta}^{\mu,\nu}(z)$  has coefficients in $Y_{\mu}^-(\Glie) \otimes Y_{\nu}^+(\Glie)$.
\end{prop}
This result is inspired by its counterpart for the quantum loop algebra in Theorem \ref{thm: polynomiality Theta}. Its proof is deferred to Appendix \ref{app: coefficients} as it requires a technical result of Gautam, Toledano Laredo and Wendlandt \cite{GTLW} on the lower triangular part of the universal R-matrix of the ordinary Yangian. It is not needed in the rest of the paper.  
\section{Decomposition of R-matrices} \label{sec: R Yangian}
In this section, we construct nonzero module morphisms $M \otimes N \longrightarrow N \otimes M$, commonly called {\it R-matrices}, in terms of the R-matrices of Theorem \ref{thm: poly R Yangian}, the T-operators of Theorem \ref{thm: Yangian T-series} and the Theta series of Definition \ref{def: Theta series Yangian}. In particular, $M$ and $N$ can be modules over the ordinary Yangian. As an application, we compute a particular diagonal entry of the R-matrix of Theorem \ref{thm: poly R Yangian}.

\subsection{Polynomiality of T-operators}
In this subsection we establish polynomiality for the T-operator $T_{\Bp}^V(w)$ of Theorem \ref{thm: Yangian T-series} for tensor products of highest $\ell$-weight modules over various shifted Yangians.

\begin{defi}
Let $V$ be a $Y_{\mu}(\Glie)$-module either top graded or bottom graded. Equip it with the negative root grading or the positive root grading as in Example \ref{example: motivation root graded Yangian}(i). For $\Bp$ be a polynomial $\ell$-weight, the eigenvalue of the $\BC((w^{-1}))$-linear automorphism $T_{\Bp}^V(w)$ of $V^w$ associated to the one-dimensional invariant subspace $V_{(0)}((w^{-1}))$ is denoted by $f_{\Bp}^V(w)$ if $V$ is top graded, and denoted by $g_{\Bp}^V(w)$ if $V$ is bottom graded.
\end{defi}
If $V$ is at the same time top graded and bottom graded, then to compute $f_{\Bp}^V(w)$ we use the T-operator arising from the negative root grading and to compute $g_{\Bp}^V(w)$ we use another T-operator arising from the positive root grading. These two T-operators differ by an overall scalar factor in $\BC((w^{-1}))^{\times}$.

A parenthesized tensor product of highest (respectively lowest) $\ell$-weight modules over various shifted Yangians is top (respectively bottom) graded.
\begin{prop}  \label{prop: poly T Yangian}
Let $\Bp \in \CL_{\zeta}$ be a polynomial $\ell$-weight and $V$ be a $Y_{\mu}(\Glie)$-module.
\begin{itemize}
\item[(i)] If $V$ is a tensor product of highest $\ell$-weight modules, then the normalized operator $f_{\Bp}^V(w)^{-1} T_{\Bp}^V(w)$ maps each nonzero vector of $V_{(-\beta)}$, for $\beta \in \BQ_+$, to a polynomial in $w$ of degree $\langle \zeta,\beta\rangle$. Its restriction to the subspace $V[w]$ is a $Y_{\mu+\zeta}(\Glie)$-module morphism denoted by
$$ \overline{T}_{\Bp}^V(w): L(\Bp)_w \otimes V \longrightarrow V \otimes L(\Bp)_w, $$
and the linear operator $\overline{T}_{\Bp}^V(w-z)$ on $V[z,w]$ is a $Y_{\mu+\zeta}(\Glie)$-module morphism
$$ \overline{T}_{\Bp}^V(w-z): L(\Bp)_w \otimes V_z \longrightarrow V_z \otimes L(\Bp)_w. $$
\item[(ii)] If $V$ is a tensor product of lowest $\ell$-weight modules, then the normalized operator $g_{\Bp}^V(w) T_{\Bp}^V(w)^{-1}$ maps each nonzero vector of $V_{(\beta)}$, for $\beta \in \BQ_+$, to a polynomial in $w$ of degree $\langle \zeta,\beta\rangle$. Its restriction to the subspace $V[w]$ is a $Y_{\mu+\zeta}(\Glie)$-module morphism denoted by
$$ \widetilde{T}_{\Bp}^V(w): V \otimes L(\Bp)_w \longrightarrow L(\Bp)_w \otimes V, $$
and the linear operator $ \widetilde{T}_{\Bp}^V(w-z)$ on $V[z,w]$ is a module morphism
$$  \widetilde{T}_{\Bp}^V(w-z): V_z \otimes L(\Bp)_w \longrightarrow L(\Bp)_w \otimes V_z.  $$
\end{itemize}
\end{prop}
\begin{proof}
(i) Let $\beta_0$ be the top weight of $V$ and $t$ be the number of tensor factors in $V$.  We prove polynomiality of $\overline{T}_{\Bp}^V(z)$ by induction on $n$.

Assume $t = 1$. Choose a nonzero vector $v_0$ of $V_{\beta_0}$. Then $V_{\beta_0-\beta}$ is spanned by the $x_{j_1,m_1}^- x_{j_2,m_2}^- \cdots x_{j_s,m_s}^- v_0$ where $s \in \BN$ and $(j_t, m_t) \in I \times \BN$ for $1\leq t \leq s$ such that
$$ \beta = \alpha_{j_1} + \alpha_{j_2} + \cdots + \alpha_{j_s}. $$
Write $\zeta = \sum_{i\in I} n_i \varpi_i^{\vee}$. Then for $i \in I$ we have a factorization
$$\Bp_i(z) =  (z-a_1)(z-a_2) \cdots (z-a_{n_i}).  $$
 Based on Eq.\eqref{intertwining T} and the definition of $\overline{T}_{\Bp}^V(w)$ we have $\overline{T}_{\Bp}^V(w) v_0 = v_0$ and
$$  \overline{T}_{\Bp}^V(w)\circ x_{i,m}^- = \prod_{k=1}^{n_i} (\sigma_i^- - w - a_k) (x_{i,m}^-)\circ \overline{T}_{\Bp}^V(w). $$ 
At the right-hand side, the first factor is a polynomial of $w$ with coefficients in $Y_{\mu}(\Glie)$ whose dominant term is $(-w)^{n_i} x_{i,m}^-$.
 We obtain as in the proof of \cite[Proposition 5.7(ii)]{HZ} that for $0\neq v \in V_{\beta_0-\beta}$, the Laurent series $\overline{T}_{\Bp}^V(w)v \in V_{\beta_0-\beta}((w^{-1}))$ is a polynomial in $w$ whose dominant term is
 $$ (-w)^{n_{j_1} + n_{j_2} + \cdots + n_{j_s}} v =  (-w)^{\langle \zeta,\beta\rangle} v.  $$ 

Assume $t > 1$. Then $V = M \otimes N$ where $M$ and $N$ are modules over shifted Yangians $Y_{\nu}(\Glie)$ and $Y_{\eta}(\Glie)$ respectively and each module is a tensor product of less than $n$ highest $\ell$-weight modules. The modules $M$ and $N$ being top graded, let $\beta_1$ and $\beta_2$ denote their top weights respectively and equip them with the root grading as in Example \ref{example: motivation root graded Yangian}(i). Then $\beta_0 = \beta_1 + \beta_2$. The induction hypothesis applied to $M$ and $N$, we have polynomiality for the operators 
$\overline{T}_{\Bp}^M(w)$ and $\overline{T}_{\Bp}^N(w)$.

Since $M$ and $N$ are positive root graded, Eq.\eqref{rel: factorization T Yangian} holds true. Applying Eq.\eqref{rel: factorization T Yangian} to $M_{\beta_1} \otimes N_{\beta_2}$ and noticing that the second tensor factor of $\Theta_{\Bp,\beta}^{\nu,\eta}(w)$ for $\beta \in \BQ_>$ lies in $Y_{\eta}(\Glie)_{\beta}$ and annihilates the top weight space $N_{\beta_2}$, we obtain that 
 $$ f_{\Bp}^V(w) = f_{\Bp}^M(w) f_{\Bp}^N(w). $$
 Dividing Eq.\eqref{rel: factorization T Yangian} by the above equation, we get another factorization
$$ \overline{T}_{\Bp}^V(w) = (1 \otimes \overline{T}_{\Bp}^N(w)) \circ \Theta_{\Bp}^{\nu,\eta}(w) \circ (\overline{T}_{\Bp}^M(w) \otimes 1) \in \mathrm{Aut}(V^w).$$

 Let $v_1 \in M$ and $v_2 \in N$ be nonzero vectors of weights $\beta_1 - \gamma_1$ and $\beta_2 - \gamma_2$ respectively, with $\gamma_1, \gamma_2 \in \BQ_+$. By induction hypothesis, $\overline{T}_{\Bp}^M(w) v_1$ is a polynomial of degree $\langle \zeta, \gamma_1\rangle$ with coefficients in $M_{\beta_1-\gamma_1}$ and $\overline{T}_{\Bp}^N(w) v_2$ is a polynomial of degree $\langle \zeta, \gamma_2\rangle$ with coefficients in $N_{\beta_2-\gamma_2}$
 From the proof of Corollary \ref{cor: nontrivial associator Yangian} and the above factorization of $\overline{T}_{\Bp}^V(w)$ we see that the Laurent series
$$ \overline{T}_{\Bp}^V(w)(v_1 \otimes v_2) - \overline{T}_{\Bp}^M(w) v_1 \otimes \overline{T}_{\Bp}^N(w) v_2 $$
 is a finite sum over $\beta \in \BQ_>$ of polynomials with coefficients in $M_{\beta_1-\gamma_1-\beta} \otimes N_{\beta_2-\gamma_2+\beta}$ whose degrees are uniformly bounded by
$$  \langle \zeta, \gamma_2-\beta \rangle +  \langle \zeta, \beta \rangle + \langle \zeta, \gamma_1 \rangle =  \langle \zeta, \gamma_1+\gamma_2 \rangle. $$
 Since  $V_{\beta_0-\gamma}$ is the direct sum of the $M_{\beta_1-\gamma_1} \otimes N_{\beta_2 - \gamma_2}$ for $\gamma_1, \gamma_2 \in \BQ_+$ such that $\gamma = \gamma_1 + \gamma_2$, we get that $\overline{T}_{\Bp}^V(w)v$ is a polynomial of degree $\langle \zeta, \gamma \rangle$ for $0 \neq v \in V_{\beta_0-\gamma}$.

It follows that $\overline{T}_{\Bp}^V(w)$ stabilizes the subspace $V[w]$ of $V^w$. Notice that $V[w]$ is the underlying space of the submodule $L(\Bp)_w \otimes V \subset L(\Bp)_w \tilde{\otimes} V$ and the submodule $V \otimes L(\Bp)_w \subset V \tilde{\otimes} L(\Bp)_w$ in Theorem \ref{thm: Yangian T-series}. We obtain a module morphism from $L(\Bp)_w \otimes V$ to $V \otimes L(\Bp)_w$. From the module identifications in Example \ref{example: tensor prefund Yangian} it follows that $\overline{T}_{\Bp}^V(w-z)$ defines a module morphism from
$L(\Bp)_w \otimes V_z$ to $V_z \otimes L(\Bp)_w$.

The proof of part (ii) is parallel.
\end{proof}

\begin{rem}
(i) Let $\Bp = \Psi_{i,0}$.  If $V$ is a highest $\ell$-weight module, then our $\overline{T}_{\Bp}^V(w)$ equals the operator $R_i^V(w)$ in \cite[Proposition 5.7]{HZ} and 
up to renormalization the abelianized transfer matrix operator $\overline{T}_i(w)|_V$ as shown in \cite[Corollary 4.7]{GW}. 

(ii) If $V$ is a tensor product of finite-dimensional irreducible modules, then it is top graded and bottom graded. We have as linear operators on $V[w]$:
$$ \overline{T}_{\Bp}^V(w) \widetilde{T}_{\Bp}^V(w) = \frac{g_{\Bp}^V(w)}{f_{\Bp}^V(w)} \mathrm{Id}_{V[w]}. $$
\end{rem}
\begin{example}  \label{example: sl2 fund T}
Let $\Glie = sl_2,\ \Bp = \Psi_{1,0}$ and $V = L(\frac{\Psi_{1,-1}}{\Psi_{1,0}})$. Then $V$ is the two-dimensional vector space $\BC e_1 \oplus \BC e_2$ with the action of the generating series of $Y_0(sl_2)$ 
$$ x_1^+(z) = z^{-1} E_{12},  \quad x_1^-(z) = 
z^{-1} E_{21},\quad \xi_1(z) = (1 + z^{-1}) E_{11} + (1 - z^{-1}) E_{22}. $$
Here $E_{ij}$ for $1\leq i, j \leq 2$ denote the elementary matrices in $\mathrm{End} V$.
Fix $a, b \in \BC$. By \cite[Example 5.11]{HZ} we have $\overline{T}_{\Bp}^{V_a}(w) = E_{11} + (a-w) E_{22}$ as a $\BC[w]$-linear operator of $V[w]$. We obtain from Example \ref{example: Yangian Theta sl2} that as $\BC[w]$-linear operators of $V\otimes V[w]$:
\begin{align*}
\Theta_{\Bp}^{0,0}(w)|_{V_a \otimes V_b} & = \exp(E_{21} \otimes E_{12}) = 1 + E_{21} \otimes E_{12}, \\
\overline{T}_{\Bp}^{V_a \otimes V_b}(w) &= (1\otimes \overline{T}_{\Bp}^{V_b}(w)) \circ \Theta_{\Bp}^{0,0}(w)|_{V_a \otimes V_b} \circ (\overline{T}_{\Bp}^{V_a}(w)\otimes 1) \\
&= E_{11} \otimes E_{11} + (a-w)(b-w) E_{22} \otimes E_{22} + (b-w) E_{11} \otimes E_{22} \\
&\qquad + (a-w) E_{22} \otimes E_{11} + E_{21} \otimes E_{12}.
\end{align*}
\end{example}

\subsection{Decomposition of R-matrices}
In this subsection, starting from the polynomial R-matrices of Theorem \ref{thm: poly R Yangian}, and applying the trivial associativity of Theorem \ref{thm: trivial associativity Yangian} and the non-trivial associativity of Corollary \ref{cor: nontrivial associator Yangian}, we obtain more general R-matrices. 

To simplify notations, for $V$ a $Y_{\mu}(\Glie)$-module and $\Bp \in \CL_{\zeta}$ a polynomial $\ell$-weight, let us denote the tensor product modules over $Y_{\mu+\zeta}(\Glie)$ in Example \ref{example: tensor prefund Yangian} as follows:
$$ V_{\Bp,w}^1 := L(\Bp)_w \otimes V, \quad V_{\Bp,w}^2 := V \otimes L(\Bp)_w. $$
Their underlying spaces are $V[w]$, so that $w$ can be evaluated at complex numbers to get two families of $Y_{\mu+\zeta}(\Glie)$-module structures $V_{\Bp,a}^1$ and $V_{\Bp,a}^2$ for $a \in \BC$ on $V$.
\begin{theorem}  \label{thm: T Theta R Yangian}
Let $\Be \in \CL_{\mu}, \Bp \in \CL_{\zeta}$ and $\Bn \in \CL_{-\nu}$ be three $\ell$-weights such that $\Bp$ and $\Bn$ are polynomial and the $Y_{\mu}(\Glie)$-module $M := L(\Be)$ is finite-dimensional. Set $N := L(\frac{1}{\Bn})$. Then we have two $Y_{\mu+\nu+\zeta}(\Glie)$-module morphisms defined by the compositions
\begin{align*}
\check{R}_{M,N}(z) \circ (\widetilde{T}_{\Bp}^M(w-z) \otimes \mathrm{Id}_N) \circ  \Theta_{\Bp}^{\mu,\nu}(w)^{-1}_{M_z \otimes N}: M_z \otimes N_{\Bp,w}^1 \longrightarrow N_{\Bp,w}^1 \otimes M_z, \\
\Theta_{\Bp}^{\nu,\mu}(w)^{-1}_{N\otimes M_z} \circ (\mathrm{Id}_N \otimes \widetilde{T}_{\Bp}^M(w-z)) \circ \check{R}_{M,N}(z): M_z \otimes N_{\Bp,w}^2  \longrightarrow N_{\Bp,w}^2 \otimes M_z.
\end{align*}
\end{theorem}
\begin{proof}
Recall from Theorem \ref{thm: poly R Yangian} the module morphism $\check{R}_{M,N}(z): M_z \otimes N \longrightarrow N \otimes M_z$. Since $N$ is a highest $\ell$-weight module, by Corollary \ref{cor: nontrivial associator Yangian} the inverse of the Theta series $\Theta_{\Bp}^{\mu,\nu}(w)^{-1}_{M_z \otimes N}$ defines a module isomorphism $M_z \otimes (L(\Bp)_w \otimes N)\longrightarrow (M_z \otimes L(\Bp)_w) \otimes N$.
Applying Theorem \ref{thm: trivial associativity Yangian} to the trivial module $L(\Bp)_w$ and Proposition \ref{prop: poly T Yangian}(ii) to $M$, we get the first module morphism as the composition:
\begin{gather*} 
 \xymatrixcolsep{6pc} \xymatrix{
M_z \otimes (L(\Bp)_w \otimes N) \ar[d]^{\Theta_{\Bp}^{\mu,\nu}(w)^{-1}_{M_z \otimes N} }  & (L(\Bp)_w \otimes N) \otimes M_z  \\
(M_z \otimes L(\Bp)_w) \otimes N \ar[d]^{\widetilde{T}_{\Bp}^M(w-z) \otimes \mathrm{Id}_N}  &     L(\Bp)_w \otimes (N \otimes M_z)  \ar[u]_{\mathrm{Id}}           \\
(L(\Bp)_w \otimes M_z) \otimes N \ar[r]^{\mathrm{Id}} & L(\Bp)_w \otimes (M_z \otimes N) \ar[u]_{\check{R}_{M,N}(z)}
 } 
\end{gather*}
The second module morphism is obtained similarly.
\end{proof}
Theorem \ref{thm: T Theta R Yangian} stays true if $M$ is a tensor product $V_1 \otimes V_2 \otimes \cdots \otimes V_n$ of finite-dimensional irreducible modules over the ordinary Yangian. We have a module morphism $\check{R}_{M,N}(z): M_z \otimes N \longrightarrow N \otimes M_z$ defined in terms of the $\check{R}_{V_i,N}(z)$; see \cite[Proposition 5.3(i)]{HZ}. 

\begin{rem}
(i) The ordinary Yangian processes a universal R-matrix \cite{Dr1}, a formal power series in $Y(\Glie)^{\otimes 2}[[z^{-1}]]$. Its evaluation at a pair $(V, W)$ of $Y(\Glie)$-modules produces a modules morphism from $V_z \otimes  W$ to a suitable completion of $W \otimes V_z$. 
 In Theorem \ref{thm: T Theta R Yangian} if we assume $\mu = 0$ and $\zeta = -\nu$ so that $M, N_{\Bp,w}^1$ and $N_{\Bp,w}^2$ are modules over the ordinary Yangian, then it is natural to expect that the two module morphisms of the theorem are evaluations of the universal R-matrix, a proof of which should be along the line of \cite[\S 3.4]{GRW}. In particular, this would imply that our R-matrices in the ordinary Yangian case satisfy the quantum Yang--Baxter equation when $M$ is a tensor product of finite-dimensional irreducible modules.  Later for the quantum loop algebra we shall establish a similar R-matrix decomposition in terms of T-operators and Theta series for evaluations of the universal R-matrix; see Theorem \ref{thm: decomposition R-matrices}. 
 
 (ii) The R-matrix $\check{R}_{M,N}(z)$ of Theorem \ref{thm: poly R Yangian} was shown in \cite[Proposition 5.3(ii)]{HZ} to satisfy the quantum Yang--Baxter equation when $M$ is a tensor product of finite-dimensional irreducible modules over the ordinary Yangian. Its proof used  the monoidal category of representations of antidominantly shifted Yangians. It does not work in the situation of Theorem \ref{thm: T Theta R Yangian}  because we have to consider a larger category of representations of shifted Yangians which is unknown to be monoidal.
\end{rem}

\begin{rem}  \label{conj: T Theta R Yangian}
Motivated by polynomiality results in Theorem \ref{thm: poly R Borel} for the quantum loop algebra, we expect that the $\BC[z]$-linear map $\check{R}_{M,N}(z)$ of Theorem \ref{thm: poly R Yangian} is invertible over $\BC(z)$ and its inverse multiplied by a nonzero polynomial restricts to a module morphism $\mathcal{F}(z): N \otimes M_z \longrightarrow M_z \otimes N$. This would lead to module morphisms  
\begin{gather*}
 \Theta_{\Bp}^{\mu,\nu}(w)_{M_z \otimes N} \circ (\overline{T}_{\Bp}^M(w-z)\otimes \mathrm{Id}_N) \circ \mathcal{F}(z): N_{\Bp,w}^1 \otimes M_z\longrightarrow M_z \otimes N_{\Bp,w}^1,   \\
 \mathcal{F}(z) \circ (\mathrm{Id}_N \otimes \overline{T}_{\Bp}^M(w-z)) \circ \Theta_{\Bp}^{\nu,\mu}(w)_{N\otimes M_z}: N_{\Bp,w}^2 \otimes M_z \longrightarrow M_z \otimes N_{\Bp,w}^2.
\end{gather*}
\end{rem}
\begin{example}
Let $\Glie = sl_2,\ \Bp = \Psi_{1,0},\ N = L(\frac{1}{\Psi_{1,0}})$ and $M = L(\frac{\Psi_{1,-1}}{\Psi_{1,0}})$. We have constructed explicitly the two-dimensional module $M$ in Example \ref{example: sl2 fund T}. Fix a highest $\ell$-weight vector $v_0$ of $N$ and define $v_n := \frac{1}{n!} (x_{1,0}^-)^n v_0$ for $n \in \BN$. Then $(v_n)_{n\in \BN}$ forms a basis of $N$. We refer to \cite[Example 3.11]{HZ} for the action of the generating series of $Y_{-\varpi_1^{\vee}}(sl_2)$ in terms of this basis. In particular,
$$ x_{1,0}^+ v_{n+1} = v_n,\quad x_{1,0}^- v_n = (n+1) v_{n+1} \quad \mathrm{for}\ n \in \BN. $$
 Consider the R-matrix $M_z \otimes N_{\Bp,w}^2 \longrightarrow N_{\Bp,w}^2 \otimes M_z$ of Theorem \ref{thm: T Theta R Yangian}. All the factors being $\BC[z,w]$-linear, with respect to the basis $(e_1, e_2)$ of $M$, they are square matrices whose entries are $\BC[z,w]$-linear operators on $N[z,w]$. By Example \ref{example: sl2 fund T}:
$$ \mathrm{Id}_N \otimes \widetilde{T}_{\Bp}^M(w-z) = \begin{pmatrix}
z-w & 0 \\
0 & 1
\end{pmatrix}. $$ 
As a rewriting of \cite[Example 7.5]{HZ} with $\mathbf{a}^{\pm}$ therein identified with $x_{1,0}^{\mp}$, we have
$$ \check{R}_{M,N}(z) = \begin{pmatrix}
1 & x_{1,0}^- \\
x_{1,0}^+ & z + x_{1,0}^-x_{1,0}^+
\end{pmatrix} = \begin{pmatrix}
1 & \sum\limits_{n\geq 0} (n+1)  E_{n+1,n} \\
\sum\limits_{n>0} E_{n-1,n} & \sum\limits_{n\geq 0} (z+n) E_{nn}
\end{pmatrix}.  $$
By Example \ref{example: Yangian Theta sl2}, we have $\Theta_{\Bp}^{-\varpi_1^{\vee}, 0}(w)^{-1} = \exp(-x_{1,0}^- \otimes x_{1,0}^+)$. Its action on $N \otimes M_z$ is given by $1 - x_{1,0}^- \otimes x_{1,0}^+$. Since $x_{1,0}^+ e_2 = e_1$ in the deformed module $M_z$, we have 
$$\Theta_{\Bp}^{-\varpi_1^{\vee},0}(w)^{-1}|_{N \otimes M_z} = \begin{pmatrix}
1 &  -x_{1,0}^- \\
0 & 1
\end{pmatrix}.  $$
The R-matrix from $M_z \otimes N_{\Bp,w}^2 \longrightarrow N_{\Bp,w}^2 \otimes M_z$ is then given by
\begin{gather*}
\begin{pmatrix}
1 &  -x_{1,0}^- \\
0 & 1
\end{pmatrix}  \begin{pmatrix}
z-w & 0 \\
0 & 1
\end{pmatrix}  \begin{pmatrix}
1 & x_{1,0}^- \\
x_{1,0}^+ & z + x_{1,0}^-x_{1,0}^+
\end{pmatrix} = \begin{pmatrix}
z-w &  -x_{1,0}^- \\
0 & 1
\end{pmatrix}  \begin{pmatrix}
1 & x_{1,0}^- \\
x_{1,0}^+ & z + x_{1,0}^-x_{1,0}^+
\end{pmatrix} \\
= \begin{pmatrix}
\sum\limits_{n\geq 0} (z-w-n)E_{nn} & -\sum\limits_{n\geq 0} (n+1) (w+n) E_{n+1,n} \\
\sum\limits_{n>0} E_{n-1,n} & \sum\limits_{n\geq 0} (z+n) E_{nn}
\end{pmatrix}.
\end{gather*}
If we specialize $w$ to a complex number $\ell \in \BC$, then the resulting matrix is the monodromy matrix $T(z)$ associated to the asymptotic module $\mathcal{W}^{-\ell}$ of \cite[Example A.2]{FZ} by identifying the basis vector $w_n \in \mathcal{W}^{-\ell}$ therein with our $n! v_n \in N$ for $n \in \BN$. A direct computation shows that $\check{R}_{M,N}(z)$ is invertible over $\BC(z)$ with inverse given by
\begin{gather*}
 \check{R}_{M,N}(z)^{-1} = \frac{1}{z-1} \begin{pmatrix}
 \sum\limits_{n\geq 0}(z+n-1)E_{nn} & -\sum\limits_{n\geq 0} (n+1) E_{n+1,n} \\
-\sum\limits_{n>0}E_{n-1,n} & 1
 \end{pmatrix}.
\end{gather*}
As expected in Remark \ref{conj: T Theta R Yangian}, $(z-1) \check{R}_{M,N}(z)^{-1}$ is polynomial.
\end{example}
\subsection{Lowest diagonal entry} In this subsection as an application of Theta series we compute a particular diagonal entry of the R-matrices of Theorem \ref{thm: poly R Yangian}, improving a result in our previous work \cite{HZ}. 
\begin{defi} \cite[Definition 5.4]{HZ}  \label{def: lowest diagonal entry}
Let $V$ be a finite-dimensional irreducible module over $Y_{\mu}(\Glie)$, and $W$ be an irreducible module over $Y_{\nu}(\Glie)$ of highest $\ell$-weight $\frac{1}{\Bp}$ such that $\Bp$ is polynomial. Let $\beta_0$ be the bottom weight of $V$ and choose a nonzero vector $v_0 \in V$ of weight $\beta_0$.  The {\it lowest diagonal entry} of $\check{R}_{V,W}(z)$ is defined to be the $\BC[z]$-linear operator $t_{V,W}(z)$ on $W[z]$ such that for $\omega \in W$ we have
$$ \check{R}_{V,W}(z)(v_0 \otimes \omega) \equiv t_{V,W}(z) \omega \otimes v_0 \ \mathrm{mod.} \sum_{\gamma \in \BQ_>} W \otimes V_{\beta_0+\gamma}[z].  $$
Define the polynomial $\lambda_{V,W}(z) \in \BC[z]$ to be the eigenvalue of $t_{V,W}(z)$ associated to the top weight space of $W$.
\end{defi}

By \cite[Theorem 7.3]{HZ}  the polynomial $\lambda_{V,W}(z)$ is always monic. In \cite[Theorem 7.4]{HZ} we have expressed the lowest diagonal entry $t_{V,W}(z)$ in terms of ratios of the commuting family of operators $\overline{T}_{\Bp}^W(z)$ of Proposition \ref{prop: poly T Yangian}, under the technical assumption that $V$ is defined over the ordinary Yangian. The following result drops this technical assumption. Its proof is more conceptual and relies on Theta series.

\begin{theorem}
Let $\Bp \in \CL_{-\nu}$ be a polynomial $\ell$-weight and set $W := L(\frac{1}{\Bp})$. Let $V$ be a finite-dimensional irreducible $Y_{\mu}(\Glie)$-module whose lowest $\ell$-weight is the ratio $\frac{\Bm}{\Bn}$ of two polynomial $\ell$-weights $\Bm$ and $\Bn$. 
Then as $\BC[z]$-linear operators on $W[z]$ we have:
$$ t_{V,W}(z) \circ \overline{T}_{\Bn}^W(z) = \lambda_{V,W}(z) \overline{T}_{\Bm}^W(z). $$
\end{theorem}
\begin{proof}
  It suffices to prove the identity with $z$ evaluated at generic complex numbers $a \in \BC$. 
By Proposition \ref{prop: poly T Yangian} we have module morphisms 
$$ \overline{T}_{\Bn}^W(a): L(\Bn)_a \otimes W \longrightarrow W \otimes L(\Bn)_a, \quad \overline{T}_{\Bm}^W(a) : L(\Bm)_a \otimes W \longrightarrow W \otimes L(\Bm)_a. $$
Fix a lowest $\ell$-weight vector $v_0$ of $V_a$, a highest $\ell$-weight vector $\omega_0$ of $W$, a nonzero vector $\mathbf{1}$ of $L(\Bn)_a$.
Observe from Eq.\eqref{equ: F G Yangian} that in the tensor product module $L(\Bn)_a \otimes V_a$ the vector $\mathbf{1} \otimes v_0$ spans a one-dimensional submodule which can be identified with $L(\Bm)_a$.
We assume that $a$ satisfies the following condition
\begin{itemize}
\item[(i)] in the triple tensor product module $W \otimes (L(\Bn)_a \otimes V_a)$ all vectors of highest $\ell$-weight $\Bp^{-1} \tau_a(\Bm)$ are proportional to $\omega_0 \otimes (\mathbf{1} \otimes v_0)$.
\end{itemize} 
This condition is generic by the q-character theory. (Indeed, let $A_{i,b} \in \CL_0$ for $i \in I$ and $b \in \BC$ be the generalized simple roots defined in \cite[(3.23)]{HZ}. By \cite[Proposition 5.8]{HZ}, there exists a countable subset $\Gamma_W$ of $\BC$ and a finite subset $\Gamma_V$ of $\BC$ such that the normalized q-character of $W$ is a power series of the $A_{i,b}^{-1}$ for $(i, b) \in I \times \Gamma_W$ and the normalized q-character of $V$ is a polynomial of the $A_{i,c}^{-1}$ for $(i,c) \in I \times \Gamma_V$. Assume 
$$ a  \notin \Gamma_W - \Gamma_V := \{b-c\ |\ b \in \Gamma_W,\ c \in \Gamma_V\}. $$
Then the $\ell$-weight space of $W \otimes (L(\Bn)_a \otimes V_a)$ of $\ell$-weight $\Bp^{-1} \tau_a(\Bm)$ is one-dimensional and spanned by the highest $\ell$-weight vector $\omega_0 \otimes (\mathbf{1} \otimes v_0)$, so Condition (i) is fulfilled.)
 Our goal is to prove the following identity of linear operators on $W$:
$$ \overline{T}_{\Bn}^W(a) \circ t_{V,W}(a) = \lambda_{V,W}(a) \overline{T}_{\Bm}^W(a) \in \mathrm{End}(W). $$

 Equip $V_a$ with the positive root grading and $W$ with the negative root grading as in Example \ref{example: motivation root graded Yangian}(i). By Corollary \ref{cor: nontrivial associator Yangian}, we have a module isomorphism 
$$F: (W \otimes (L(\Bn)_a) \otimes V_a \longrightarrow W \otimes (L(\Bn)_a \otimes V_a)  $$
which maps $(\omega \otimes \mathbf{1}) \otimes v$, for $\omega \in W_{(\alpha)}$ and $v \in (V_a)_{(\beta)}$, to $\omega \otimes (\mathbf{1} \otimes v)$ plus a linear combination of vectors in $W_{(\alpha-\gamma)} \otimes (\mathbf{1} \otimes (V_a)_{(\beta+\gamma)})$ for $\gamma \in \BQ_>$. Consider the following diagram of linear maps where $G$ is the composition of the other maps
\begin{gather*} 
 \xymatrixcolsep{4pc} \xymatrix{
(L(\Bn)_a \otimes V_a) \otimes W \ar[d]_{\mathrm{Id} } \ar@{-->}[r]^G  & W \otimes (L(\Bn)_a \otimes V_a) & (W \otimes L(\Bn)_a) \otimes V_a  \ar[l]_F  \\
L(\Bn)_a \otimes (V_a \otimes W) \ar[r]^{\check{R}_{V,W}(a)}  &   L(\Bn)_a \otimes (W\otimes V_a) \ar[r]^{\mathrm{Id}}          & (L(\Bn)_a \otimes W) \otimes V_a \ar[u]^{\overline{T}_{\Bn}^W(a) \otimes \mathrm{Id}_{V_a} }  } 
\end{gather*}
Since the two identity maps are module morphisms by Theorem \ref{thm: trivial associativity Yangian}, this is a commutative diagram of module isomorphisms. Since $(\mathbf{1} \otimes v_0) \otimes \omega_0$ is a vector of highest $\ell$-weight $\Bp^{-1}\tau_a(\Bm)$, so is its image by $G$ and from Condition (i) we obtain
$$ G((\mathbf{1} \otimes v_0) \otimes \omega_0) \in \BC \omega_0 \otimes (\mathbf{1}\otimes v_0). $$
 The right-hand side is included in the submodule $W \otimes (\mathbf{1} \otimes v_0)$ of the codomain of $G$. 
From \cite[Theorem 4.8]{HZ} or Eq.\eqref{equ: F G Yangian} observe that $(\mathbf{1} \otimes v_0) \otimes \omega_0$ generates the submodule $(\mathbf{1} \otimes v_0) \otimes W$ of the domain of $G$. It follows that 
$$ G((\mathbf{1} \otimes v_0) \otimes W) \subset W \otimes (\mathbf{1}\otimes v_0). $$
By definition of the lowest diagonal entry and the uni-triangular property of $F$, for $\beta \in \BQ_-$ and $\omega \in W_{(\beta)}$, the vector $G((\mathbf{1} \otimes v_0) \otimes \omega)$ is $\overline{T}_{\Bn}^W(a) (t_{V,W}(a) \omega) \otimes (\mathbf{1} \otimes v_0)$ plus a linear combination of vectors in $W_{(\beta-\gamma)} \otimes (\mathbf{1} \otimes V_{(\gamma)})$ for $\gamma \in \BQ_>$. Together with the relation $G((\mathbf{1} \otimes v_0) \otimes \omega) \in W \otimes (\mathbf{1} \otimes V_{(0)})$ we must have
$$ G((\mathbf{1} \otimes v_0) \otimes \omega) = \overline{T}_{\Bn}^W(a) ( t_{V,W}(a) \omega)  \otimes (\mathbf{1}\otimes v_0) \quad \mathrm{for}\ \omega \in W. $$
Under the identifications $W \otimes L(\Bm)_a = W \otimes (\mathbf{1} \otimes v_0)$ and $L(\Bm)_a \otimes W = (\mathbf{1} \otimes v_0) \otimes W$, the module morphism $G$ restricts to a module morphism
$$ \overline{T}_{\Bn}^W(a) \circ t_{V,W}(a): L(\Bm)_a \otimes W \longrightarrow W \otimes L(\Bm)_a $$
which as a linear operator on $W$ sends $\omega_0$ to $\lambda_{V,W}(a) \omega_0$. Such a module morphism is unique and equal to $\lambda_{V,W}(a) \overline{T}_{\Bm}^W(a)$ because the tensor product module $L(\Bm)_a \otimes W$ is generated by the highest $\ell$-weight vector $\mathbf{1} \otimes \omega_0$.
\end{proof}
In the proof, the genericity condition (i) implies $G(L(\Bm)_a \otimes W) \subset W \otimes L(\Bm)_a$. Since $G$ is polynomial in $a$, the inclusion holds true for all $a \in \BC$. This should explain the miraculous cancellation $\tilde{E}_a = 0$ at Step 3 of the proof of \cite[Theorem 7.4]{HZ}.

\section{Generalities on shifted quantum affine algebras}  \label{sec: quantum}
In this section we recall basic algebraic properties of shifted quantum affine algebras and their representation theory. 

Fix $q \in \BC^{\times}$ which is not a root of unity. For $t \in \BZ$ and $n \in \BN$ set
$$   (t)_q := \frac{q^{2t}-1}{q^2-1},\quad [t]_q := \frac{q^t-q^{-t}}{q-q^{-1}}, \quad (n)_q! := \prod_{m=1}^n (m)_q,\quad \dstirling{t}{n}_q := \prod_{m=1}^n [t-m+1]_q.  $$
Recall $I = \{1,2,\cdots, r\}$ and the symmetrized Cartan matrix $(b_{ij})_{i,j\in I}$.
 Define the symmetric quantum Cartan matrix to be $B(q) := ([b_{ij}]_q)_{i,j\in I}$. It is invertible because $q$ is generic, and its inverse is denoted by $\widetilde{B}(q) = (\widetilde{B}_{ij}(q))_{i,j\in I}$. 
 
Let $\theta = \sum_{i=1}^r a_i \alpha_i \in \BQ$ be the highest positive root of $\mathfrak{g}$. 
We enlarge the finite-type Cartan matrix $(c_{ij})_{1\leq i,j \leq r}$ to an affine-type Cartan matrix $(c_{ij})_{0\leq i,j\leq r}$ as follows:
$$c_{00} := 2, \quad c_{i0} = -\frac{2(\alpha_i, \theta)}{(\alpha_i, \alpha_i)},\quad  c_{0i} := -\frac{2(\theta, \alpha_i)}{(\theta, \theta)}\quad \mathrm{for}\ 1 \leq i \leq r. $$
The Kac--Moody algebra associated to the affine Cartan matrix $(c_{ij})_{0\leq i, j \leq r}$ without derivation operator is denoted by $\Gaff$. Its root lattice is $\BZ \delta \oplus \BQ$ where $\delta$ is the standard imaginary root. Its simple roots are $(\alpha_i)_{0\leq i \leq r}$ with $\alpha_0 = \delta-\theta$. As a Lie algebra it is a central extension of the loop algebra $L \Glie := \Glie \otimes \BC[t,t^{-1}]$. 

Set $q_i := q^{d_i}$ for $1\leq i \leq r$ and $q_0 := q^{d_0}$ where $d_0 := \frac{(\theta,\theta)}{2}$.
\subsection{The quantum loop algebra}
The {\it quantum loop algebra} $U_q(L\mathfrak{g})$ is the associative algebra generated by $e_i, f_i, k_i^{\pm 1}$ for $0 \leq i \leq r$ and with relations for $0 \leq i, j \leq r$:
\begin{gather*}
 k_ik_i^{-1} = 1 = k_i^{-1}k_i,\quad k_i k_j = k_j k_i, \quad k_0 k_1^{a_1} k_2^{a_2} \cdots k_r^{a_r} = 1, \\
k_i e_j = q_i^{c_{ij}} e_j k_i,\quad k_i f_j = q_i^{-c_{ij}} f_j k_i, \quad [e_i, f_j] = \delta_{ij} \frac{k_i - k_i^{-1}}{q_i-q_i^{-1}}, \\
\sum_{s=0}^{1-c_{ij}} (-1)^s \dstirling{1-c_{ij}}{s}_{q_i} x_i^{1-c_{ij}-s} x_j x_i^s = 0  \quad \textrm{if $i \neq j$ and $x \in \{e, f\}$}.
\end{gather*}
The algebra $\qaf$ has a Hopf algebra structure with coproduct given by:
\begin{equation}  \label{def: coproduct DJ}
\Delta(e_i) = e_i \otimes 1 + k_i \otimes e_i,\quad \Delta(f_i) = 1 \otimes f_i + f_i \otimes k_i^{-1},\quad \Delta(k_i) = k_i \otimes k_i.
\end{equation}
It contains two Hopf subalgebras, the {\it upper Borel subalgebra} $\Borel$  generated by $e_i, k_i$ for $0 \leq i \leq r$ and the {\it lower Borel subalgebra} $U_q(\mathfrak{c})$ generated by $f_i, k_i$ for $0 \leq i \leq r$.

We occasionally write the upper Borel subalgebra $\Borel$ as $U_q(\mathfrak{b}^+)$ and the lower Borel subalgebra $U_q(\mathfrak{c})$ as $U_q(\mathfrak{b}^-)$ to unify results about these two algebras.
\begin{rem}  \label{rem: affine}
In the above definition if we drop the relation $k_0 k_1^{a_1} k_2^{a_2} \cdots k_r^{a_r} = 1$, then we obtain the quantum affine algebra $U_q(\Gaff)$, the Drinfeld--Jimbo quantum group associated to the affine Kac--Moody algebra $\Gaff$.
\end{rem}

With respect to the conjugate action of the $k_i$, the Hopf algebra $\qaf$ is graded by the root lattice $\BQ$: an element $x \in \qaf$ is of weight $\beta \in \BQ$, and we write $\wt(x) = \beta$, if $k_i x k_i^{-1} = q^{(\alpha_i,\beta)} x$ for $1\leq i \leq r$. In terms of generators, we have
\begin{gather*}
 \wt(e_i) = \alpha_i, \quad  \wt(f_i) = -\alpha_i,\quad \wt(k_i) = 0 \quad \mathrm{for}\ 1\leq i \leq r; \\
 \wt(e_0) = -\theta,\quad \wt(f_0) = \theta,\quad \quad \wt(k_0) = 0.
\end{gather*}
The Hopf algebra $\qaf$ is also $\BZ$-graded by setting 
$$ \deg(e_i) = \delta_{0i},\quad \deg(f_i) = -\delta_{0i},\quad \deg(k_i) = 0 \quad \mathrm{for}\ 0 \leq i \leq r. $$
 It induces a $\BC[w,w^{-1}]$-algebra automorphism $\tau_w$ of $\qaf[w,w^{-1}]$ such that (we choose $w^{-m}$ below instead of commonly used $w^m$ to simplify the formulas)
 \begin{equation} \label{def: spectral parameter}
  \tau_w(x) = w^{-m} x\quad \textrm{for $x \in \qaf$ of degree $m$}.
 \end{equation}
Specializing $w$ to nonzero complex numbers, we obtain a one-parameter family of Hopf algebra automorphisms $\tau_a$ of $\qaf$ for $a \in \BC^{\times}$ satisfying $\tau_a \circ \tau_b = \tau_{ab}$. Notice that both the weight grading and the $\BZ$-grading descend to Borel subalgebras. 
For $\beta \in \BQ$, let $\qaf_{\beta}$ denote the subspace of elements in $\qaf$ of weight $\beta$. The subspaces $U_q(\mathfrak{b}^{\pm})_{\beta}$ are defined similarly.

We shall need Drinfeld's description of the square of the antipode $S^2$ as an algebra automorphism of $\qaf$. For this purpose, let $\rho \in \frac{1}{2} \BQ$ be half the sum of positive roots of $\Glie$. It is the element $\rho \in \Hlie^*$ characterized by
$ (\rho, \alpha_i) = d_i$ for $1\leq i \leq r$.
As in \cite[\S 3.5]{GW}, let $h^{\vee}$ denote the dual Coxeter number of $\Glie$ and set 
$$\kappa := \frac{1}{2} h^{\vee} \max(d_1, d_2, \cdots, d_r) = \frac{1}{4} (\theta, \theta + 2\rho). $$
Then the square of the antipode $S^2$ is computed as follows \cite[(1.6)]{FM}:
 \begin{equation}  \label{equ: antipode}
 S^2(x) = q^{-(2\rho, \beta) - 4 \kappa m} x \quad \textrm{for $x \in \qaf$ of weight $\beta$ and of degree $m$}.
 \end{equation}

\subsection{Drinfeld loop realization}   \label{ss: Beck isomorphism}
The quantum loop algebra $\qaf$ has a second presentation by Drinfeld generators $x_{i,m}^{\pm}, \phi_{i,m}^{\pm}$ for $(i,m) \in I \times \BZ$ subject to the following relations for $(i,j,m,n) \in I^2 \times \BZ^2$ and $\varepsilon \in \{+, -\}$:
\begin{gather} 
\phi_{i,m}^+ = 0 \ \mathrm{if}\ m < 0,\quad \phi_{i,m}^- = 0 \ \mathrm{if}\ m > 0, \quad \phi_{i,0}^+ \phi_{i,0}^- = 1,  \label{Drinfeld rel: initial}  \\
[\phi_{i,m}^{\pm}, \phi_{j,n}^{\varepsilon}] = 0, \quad [x_{i,m}^+, x_{j,n}^-] = \delta_{ij} \frac{\phi_{i,m+n}^+-\phi_{i,m+n}^-}{q_i-q_i^{-1}}, \label{Drinfel rel: Cartan}  \\
\phi_{i,m+1}^{\varepsilon} x_{j,n}^{\pm} -q^{\pm b_{ij}} \phi_{i,m}^{\varepsilon} x_{j,n+1}^{\pm} = q^{\pm b_{ij}} x_{j,n}^{\pm} \phi_{i,m+1}^{\varepsilon} - x_{j,n+1}^{\pm} \phi_{i,m}^{\varepsilon},  \label{Drinfeld rel: Drinfeld-Cartan} \\
x_{i,m+1}^{\pm} x_{j,n}^{\pm} -q^{\pm b_{ij}} x_{i,m}^{\pm} x_{j,n+1}^{\pm} = q^{\pm b_{ij}} x_{j,n}^{\pm} x_{i,m+1}^{\pm} - x_{j,n+1}^{\pm} x_{i,m}^{\pm},  \label{Drinfeld rel: Drinfeld} \\
\sum_{s=0}^{1-c_{ij}} (-1)^s \dstirling{1-c_{ij}}{s}_{q_i}(x_{i,0}^{\pm})^{1-c_{ij}-s} x_{j,0}^{\pm} (x_{i,0}^{\pm})^s = 0 \quad \mathrm{if}\ i \neq j. \label{Drinfeld rel: Serre}
\end{gather}

For $i \in I$ and $y \in \{x^{\pm}, \phi^{\pm}\}$ let us define the generating series $y_i(z) := \sum_{m\in \BZ} y_{i,m} z^m$. 
Since the series $\phi_i^{\pm}(z)$ mutually commute by Eq.\eqref{Drinfel rel: Cartan}, we define Drinfeld--Cartan elements $h_{i,s}$ for $i \in I$ and $s \in \BZ_{\neq 0}$ by (we follow the convention of \cite{FR1})
\begin{equation}   \label{def: h}
\phi_i^{\pm}(z) = \phi_{i,0}^{\pm} \exp(\pm (q-q^{-1}) \sum_{\pm s>0} h_{i,s} z^s).
\end{equation}
The $h_{i,s}$ mutually commute and we have the following relation from Eq.\eqref{Drinfeld rel: Drinfeld-Cartan}
\begin{equation} \label{Drinfeld rel: h x}
[h_{i,s}, x_{j,m}^{\pm}] = \frac{[b_{ij}s]_q}{s} x_{j,m+s}^{\pm} \quad \mathrm{for}\ (i,j,s,m) \in I^2 \times \BZ_{\neq 0} \times \BZ.
\end{equation}

The quantum loop algebra $\qaf$ admits a {\it triangular decomposition}. Let us define three subalgebras by generating subsets. The first $U_q^+(L\Glie)$ is generated by all the  $x_{i,m}^+$, the second $U_q^-(L\Glie)$ generated by the $x_{i,m}^-$, and the third $U_q^0(L\Glie)$ by the $\phi_{i,m}^{\pm}$. Then the multiplication map induces a vector space isomorphism
\begin{equation} \label{Drinfeld rel: triangular decomposition}
U_q^-(L\Glie) \otimes U_q^0(L\Glie) \otimes U_q^+(L\Glie) \longrightarrow \qaf,\quad x \otimes y \otimes z \mapsto xyz.
\end{equation}
Furthermore, the three subalgebras inherit the Drinfeld loop realization from $\qaf$. (To present the subalgebras $U_q^-(L\Glie)$ and $U_q^+(L\Glie)$ one should use the more general form of the Serre relation than \eqref{Drinfeld rel: Serre}, as in \cite{Beck}.)

\begin{rem} 
The correspondence between the Drinfeld--Jimbo generators and Drinfeld generators is due to Beck \cite{Beck}. We have for $i \in I$ (see also Proposition \ref{prop: root vectors}):
\begin{align*}
e_i = x_{i,0}^+, \quad f_i = x_{i,0}^-, \quad k_i = \phi_{i,0}^+,\quad k_0^{-1} e_0 \in U_q^-(L\Glie),\quad f_0 k_0  \in U_q^+(L\Glie). 
\end{align*}
For $m \in \BN$, the positive mode Drinfeld generators $x_{i,m}^+,\ x_{i,m+1}^-,\ h_{i,m+1}$ and $\phi_{i,m}^+$ belong to the upper Borel subalgebra $\Borel$, while the negative mode Drinfeld generators $x_{i,-m}^-,\ x_{i,-m-1}^+,\ h_{i,-m-1}$ and $\phi_{i,-m}^-$ belong to the lower Borel subalgebra $\lBorel$.
\end{rem}

The weight grading and the $\BZ$-grading in terms of Drinfeld generators are:
\begin{equation}  \label{def: bigrading}
\wt(x_{i,m}^{\pm}) = \pm \alpha_i,\quad \wt(\phi_{i,m}^{\pm}) = 0,\quad \deg(x_{i,m}^{\pm}) = m,\quad \deg(\phi_{i,m}^{\pm}) = m. 
\end{equation}
It follows from Eq.\eqref{equ: antipode} that for $i \in I$ and $m \in \BZ$:
\begin{equation} \label{Drinfeld rel: antipode}
S^2(x_{i,m}^{\pm}) = q^{-4\kappa m\mp 2d_i} x_{i,m}^{\pm},\quad S^2(\phi_{i,m}^{\pm}) = q^{-4\kappa m} \phi_{i,m}^{\pm}.
\end{equation}

\subsection{Shifted quantum affine algebras} 
For $\mu = \sum_{i\in I} n_i \varpi_i^{\vee}$ a coweight, the {\it shifted quantum affine algebra} $\CU_{\mu}(\Gaff)$ is the associative algebra generated by $ x_{i,m}^{\pm}, \phi_{i,m}^{\pm}$ for $(i, m) \in I \times \BZ$
subject to Eqs.\eqref{Drinfel rel: Cartan}--\eqref{Drinfeld rel: Serre} and the following relations \cite{FT}:
\begin{equation}  \label{Drinfeld rel: shift}
\phi_{i,m}^+ = 0 \ \mathrm{if}\ m < 0,\quad \phi_{i,m}^- = 0 \ \mathrm{if}\ m > n_i, \quad \phi_{i,0}^+ \phi_{i,n_i}^- \textrm{ is central and invertible}.
\end{equation}

The quantum loop algebra $\qaf$ is the quotient of the zero-shifted quantum affine algebra $\CU_0(\Gaff)$ by the ideal generated by the central elements $\phi_{i,0}^+ \phi_{i,0}^- - 1$ for $i \in I$. 

We have formal power series $x_i^{\pm}(z)$ and $\phi_i^{\pm}(z)$ with coefficients in $\CU_{\mu}(\Gaff)$.
By definition,  $\phi_i^+(z)$ is a power series in $z$ of leading term $\phi_{i,0}^+$, while $\phi_i^-(z)$ is a Laurent series in $z^{-1}$ of leading term $\phi_{i,n_i}^- z^{n_i}$. 

Define the three subalgebras $\CU_{\mu}^+(\Gaff),\ \CU_{\mu}^-(\Gaff)$ and $\CU_{\mu}^0(\Gaff)$ of $\CU_{\mu}(\Gaff)$ as in the case of the quantum loop algebra. Then as in \eqref{Drinfeld rel: triangular decomposition} we have a vector space isomorphism
$$ \CU_{\mu}^-(\Gaff) \otimes \CU_{\mu}^0(\Gaff) \otimes \CU_{\mu}^+(\Gaff) \longrightarrow \CU_{\mu}(\Gaff), \quad a \otimes b \otimes c \mapsto abc.  $$
 Furthermore, we have canonical identifications of algebras
 \begin{equation}  \label{identification subalgebras}
U_q^{\pm}(L\Glie) \cong \CU_{\mu}^{\pm}(\Gaff),\quad x_{i,m}^{\pm} \mapsto x_{i,m}^{\pm}.
 \end{equation}

Eq.\eqref{def: bigrading} equips the algebra $\CU_{\mu}(\Gaff)$ with a weight grading and a $\BZ$-grading. Notably, the $\BZ$-grading induces a one-parameter family of algebra automorphisms $\tau_a$ of $\CU_{\mu}(\Gaff)$ for $a \in \BC^{\times}$ and a $\BC[w,w^{-1}]$-algebra automorphism $\tau_w$ of $\CU_{\mu}(\Gaff)[w,w^{-1}]$ as in Eq.\eqref{def: spectral parameter}. 
An element $x \in \CU_{\mu}(\Gaff)$ is of weight $\beta \in \BQ$ if and only if $\phi_{i,0}^+ x = q^{(\beta,\alpha_i)} x \phi_{i,0}^+$ for all $i \in I$, if and only if $\phi_{i,n_i}^- x = q^{-(\beta,\alpha_i)} x \phi_{i,n_i}^-$ for all $i \in I$. 

In \cite[\S 10(vii)]{FT}, different shifted quantum affine algebras are connected by the so-called {\it shifted homomorphisms}. The proof of injectivity was given in \cite[Appendix I]{FT} in type A and can be extended to general types using the shuffle realization of the subalgebras $U_q^{\pm}(L\Glie)$ recently obtained in \cite[Theorem 1.7]{NT}.
\begin{prop} \cite{FT}  \label{prop: shifted homomorphism}
Let $\epsilon$ and $\eta$ be antidominant coweights. For $\mu$ a coweight we have an injective algebra homomorphism $\iota_{\epsilon,\eta}^{\mu}:  \CU_{\mu}(\Gaff) \longrightarrow \CU_{\mu+\epsilon+\eta}(\Gaff)$ defined by
\begin{gather*}
x_i^+(z) \mapsto (1-z)^{b_i} x_i^+(z) ,  \quad x_i^-(z) \mapsto (1-z)^{c_i} x_i^-(z), \quad \phi_i^{\pm}(z) \mapsto (1-z)^{b_i+c_i} \phi_i^{\pm}(z).
\end{gather*}
Here $b_i := -\langle \epsilon, \alpha_i\rangle$ and $c_i := -\langle \eta, \alpha_i\rangle$ for $i \in I$.
\end{prop}

\subsection{Representations of shifted quantum affine algebras}  \label{ss: highest weight}
 We review results on highest weight representations of shifted quantum affine algebras from \cite{H}. All the constructions and definitions of Subsection \ref{ss: rep shifted Yangians} for shifted Yangians, except tensor product modules, carry over to the present situation after small modifications.

 Following \cite[\S 3.1]{HJ}, let $\mathfrak{t}^* := (\BC^{\times})^I$ be the set of $I$-tuples of nonzero complex numbers. For $i \in I$, let $\lambda(i)$ denote the $i$-th component of $\lambda$. Equip $\mathfrak{t}^*$ with an {\it additive} group structure by component-wise multiplication: $(\lambda+\omega)(i) := \lambda(i) \omega(i)$.
 
The root lattice $\BQ$, as the additive group freely generated by the simple roots $\alpha_j$ for $j \in I$, is viewed as a subgroup of $\mathfrak{t}^*$ so that $\alpha_j= (q^{b_{ji}})_{i\in I}$. This is an embedding because $q$ is not a root of unity and the matrix $(b_{ij})_{i,j\in I}$ is invertible. 

Fix a coweight $\mu = \sum_{i\in I} n_i \varpi_i^{\vee}$. Let $M$ be a module over the shifted quantum affine algebra $\CU_{\mu}(\Gaff)$. For $\lambda \in \mathfrak{t}^*$, set
$$ M_{\lambda} := \{v \in M \ |\ \phi_{i,0}^+ v = \lambda(i) v \quad \mathrm{for}\ i \in I \}. $$
If $M_{\lambda}$ is nonzero, then it is called the weight space of weight $\lambda$. By definition, for $x$ an element in $\CU_{\mu}(\Gaff)$ of weight $\beta \in \BQ$, we have $x M_{\lambda} \subset M_{\lambda+\beta}$.  

Call $M$ {\it weight graded} if $M$ is a direct sum of the weight spaces. 
Call $M$ {\it top graded} if it is weight graded and there exists $\lambda_0 \in \mathfrak{t}^*$ such that 
 $$ \textrm{$\dim M_{\lambda_0} = 1$ and $M_{\lambda} \neq \{0\}$ only if $\lambda_0 - \lambda \in  \BQ_+$.}$$
Call $\lambda_0$ the top weight and $M_{\lambda_0}$ is called the top weight space. Bottom graded modules are defined by replacing the condition $\lambda_0 - \lambda \in  \BQ_+$ with $\lambda - \lambda_0 \in  \BQ_+$.
  
  Define $\mathfrak{t}_{\mu}^*$ to be the set of $I$-tuples $(\Bf_i^{\pm}(z))_{i\in I}$ where
$\Bf_i^+(z) \in \BC[[z]]^{\times}$ is an invertible power series in $z$ and $\Bf_i^-(z) \in \BC((z^{-1}))$ is a Laurent series in $z^{-1}$ whose leading term is of degree $n_i$ for $i \in I$. Such an $I$-tuple is called an {\it $\ell$-weight} of coweight $\mu$. 
  
Given  $\Bf = (\Bf_i^{\pm}(z))_{i\in I} \in \mathfrak{t}_{\mu}^*$, define the highest weight {\it Verma module} to be the $\CU_{\mu}(\Gaff)$-module generated by $\omega$ subject to relations
$$ \phi_i^{\pm}(z) \omega = \Bf_i^{\pm}(z) \omega,\quad x_i^+(z) \omega = 0 \quad \mathrm{for}\ i \in I. $$
It has a unique irreducible quotient, denoted by $L_{\mu}(\Bf)$. Highest $\ell$-weight modules are defined in the obvious way as nonzero quotients of Verma modules. We add a subscript $\mu$ to the irreducible module over the shifted quantum affine algebra $\CU_{\mu}(\Gaff)$ because later we will be mainly concerned with the highest $\ell$-weight irreducible module over the upper Borel subalgebra $\Borel$ denoted by $L(\Bf)$.

Call an $\ell$-weight $\Bf \in \mathfrak{t}_{\mu}^*$ {\it rational} if for $i \in I$:
the series $\Bf_i^+(z)$ and $\Bf_i^-(z)$ are Laurent expansions of the same rational function, denoted by $\Bf_i(z)$, around $z = 0$ and $z = \infty$ respectively.
By abuse of language we identify $\Bf$ with the $I$-tuple $(\Bf_i(z))_{i\in I}$ of rational functions. Call $\Bf$ {\it polynomial} if it is rational and each $\Bf_i(z)$ is polynomial in $z$.  
Let $\mathfrak{r}_{\mu}$ denote the subset of $\mathfrak{t}_{\mu}^*$ formed of rational $\ell$-weights.

\begin{prop}\cite[Theorem 4.12]{H}   \label{prop: rationality shifted}
Let $\Bf \in \mathfrak{t}_{\mu}^*$. Then $\Bf$ is rational if and only if all weight spaces of the $\CU_{\mu}(\Gaff)$-module $L_{\mu}(\Bf)$ are finite-dimensional. $\Bf$ is polynomial if and only if $L_{\mu}(\Bf)$ is one-dimensional. 
\end{prop}
The disjoint union of the $\mathfrak{t}_{\mu}^*$ over all coweights $\mu$ forms a multiplicative group under component-wise multiplication. We have $\mathfrak{t}_{\mu}^* \mathfrak{t}_{\nu}^* \subset \mathfrak{t}_{\mu+\nu}^*$. The subset of rational $\ell$-weights, namely the disjoint union of the $\mathfrak{r}_{\mu}$ over all coweights $\mu$, forms a subgroup.
\begin{defi} \label{defi: prefund}
 For $j \in I$ and $a \in \BC^{\times}$, the {\it prefundamental $\ell$-weight} $\Psi_{j,a}$ is the polynomial $\ell$-weight of coweight $\varpi_j^{\vee}$ whose $i$th component is $1-za \delta_{ij}$ for $i\in I$. For $\zeta$ a dominant coweight, let $\mathfrak{d}_{\zeta}$ denote the subset of $\mathfrak{r}_{\zeta}$ formed of $\ell$-weights which are monomials of the $\Psi_{i,a}$. Equivalently, $(\Bf_i(z))_{i\in I} \in \mathfrak{d}_{\zeta}$ if and only if each $\Bf_i(z)$ is a polynomial in $z$ of degree $\langle \zeta, \alpha_i\rangle$ and of constant term 1 (also called Drinfeld polynomial). Let $\mathfrak{d}$ be the disjoint union of the $\mathfrak{d}_{\zeta}$ over all dominant coweights $\zeta$. 
\end{defi}
\begin{defi}  \label{defi: deformed module}
Let $V$ be a $\CU_{\mu}(\Gaff)$-module. The Laurent polynomial space $V[w,w^{-1}]$ being a module over the algebra $\CU_{\mu}(\Gaff)[w,w^{-1}]$ by scalar extension, its pullback along the algebra homomorphism $\tau_w: \CU_{\mu}(\Gaff) \longrightarrow \CU_{\mu}(\Gaff)[w,w^{-1}]$ defines a $\CU_{\mu}(\Gaff)$-module, denoted by $V_w$ and called {\it deformed module}. We refer to $w$ as {\it spectral parameter}.
\end{defi}
Evaluating $w$ at nonzero complex numbers gives a one-parameter family of representations $V_a = \tau_a^*V$, for $a \in \BC^{\times}$, on the same underlying space as $V_1 = V$.  If $V$ is of highest $\ell$-weight $\Bf \in \mathfrak{t}_{\mu}^*$, then $V_a$ is of highest $\ell$-weight $(\Bf_i^{\pm}(za^{-1}))_{i\in I}$, denoted by $\tau_a(\Bf)$.

\section{Further properties of Borel subalgebras}  \label{sec: Borel}
We recall further properties of the upper Borel subalgebra of the quantum loop algebra: root vectors, highest/lowest weight representations, the universal R-matrix and the monodromy matrices. 
\subsection{Root vectors}   \label{ss: root vectors}
The finite-type Lie algebra $\mathfrak{g}$ has simple roots $\alpha_i \in \Hlie^*$ for $1\leq i \leq r$, whose $\BZ$-span is the root lattice $\BQ$. Let $\Phi \subset \BQ$ be the set of {\it positive} roots of $\mathfrak{g}$. Consider the affine Kac--Moody algebra $\Gaff$ with Cartan matrix $(c_{ij})_{0\leq i,j \leq r}$ and root lattice $\BZ \delta \oplus \BQ$. 
The set $\hat{\Phi}$ of real positive roots of $\Gaff$ is a disjoint union of the following two subsets of the affine root lattice:
$$ \hat{\Phi}_+ := \{n \delta+\beta\ |\ n \geq 0,\ \beta \in \Phi \},\quad \hat{\Phi}_- := \{s \delta - \beta\ |\ s > 0,\ \beta \in \Phi\}. $$
In \cite{Beck} depending on a reduced expression of a particular element in the extended affine Weyl group, $\hat{\Phi}$ has a linear ordering $\prec$ satisfying $\alpha \prec \alpha'$ for $\alpha \in \hat{\Phi}_+$ and $\alpha' \in \hat{\Phi}_-$. To each positive real root $\alpha$ are attached the {\it root vectors} $E_{\alpha} \in \Borel$ and $F_{\alpha} \in \lBorel$ whose weights and degrees are given by
$$ \wt(E_{n\delta+\beta}) = \beta = - \wt(F_{n\delta+\beta}),\quad \deg (E_{n\delta+\beta}) = n = -\deg (F_{n\delta+\beta}). $$
In particular, $E_{\alpha_i} = e_i$ and $F_{\alpha_i} = f_i$ for $0\leq i \leq r$.

 Define the following three subalgebras of Borel subalgebras
$$ U_q^{\bullet}(\mathfrak{b}^{\pm}) := U_q^{\bullet}(L\Glie) \cap U_q(\mathfrak{b}^{\pm})\quad \mathrm{for}\ \bullet \in \{-, 0, +\}. $$
Then the triangular decomposition \eqref{Drinfeld rel: triangular decomposition} restricts to vector space isomorphisms
$$ U_q^-(\mathfrak{b}^{\pm}) \otimes U_q^0(\mathfrak{b}^{\pm}) \otimes U_q^+(\mathfrak{b}^{\pm}) \longrightarrow U_q(\mathfrak{b}^{\pm}),\quad a \otimes b \otimes c \mapsto abc. $$

Let $\beta\mapsto k_{\beta}$ be the unique map $\BQ \longrightarrow \qaf$ such that $k_{\alpha_i} = k_i$ for $i \in I$ and $k_{\beta+\gamma} = k_{\beta}k_{\gamma}$ for $\beta, \gamma \in \BQ$.
We summarize the basic properties of the root vectors. 
\begin{prop}\cite[Proposition 9.3]{Damiani2}  \label{prop: root vectors}
\begin{itemize}
\item[(i)] Ordered monomials in the $E_{\alpha}$ for $\alpha \in \hat{\Phi}_+$ form a basis of $U_q^+(\mathfrak{b})$.
\item[(ii)] Ordered monomials in the $k_{\beta} E_{s\delta-\beta}$ for $s\delta-\beta \in \hat{\Phi}_-$ form a basis of $U_q^-(\mathfrak{b})$.
\item[(iii)]  Ordered monomials in the $F_{\alpha}$ for $\alpha \in \hat{\Phi}_+$ form a basis of $U_q^-(\mathfrak{c})$.
\item[(iv)] Ordered monomials in the $F_{s\delta-\beta}k_{\beta}^{-1}$ for $s\delta-\beta \in \hat{\Phi}_-$ form a basis of $U_q^+(\mathfrak{c})$.
\end{itemize}
\end{prop}

We adapt the construction of the root vectors to shifted quantum affine algebras.
Given a coweight $\mu$, let $\beta\mapsto \phi_{\beta}^+$ be the unique map  $\BQ \longrightarrow \CU_{\mu}(\Gaff)$ such that $\phi_{\alpha_i}^+ = \phi_{i,0}^+$ for $i \in I$ and $\phi^+_{\beta+\gamma} = \phi^+_{\beta}\phi^+_{\gamma}$ for $\beta, \gamma \in \BQ$. 

\begin{defi} \label{def: shifted root vectors}
For $\mu = \sum_{i\in I} n_i \varpi_i^{\vee}$ a coweight, define in the shifted quantum affine algebra $\CU_{\mu}(\Gaff)$ the {\it Drinfeld--Cartan} elements $h_{i,s}$ for $(i,s) \in I \times \BZ_{\neq 0}$ by
\begin{equation*}  
\phi_i^+(z) = \phi_{i,0}^+ \exp((q-q^{-1}) \sum_{s>0} h_{i,s} z^s),\quad \phi_i^-(z) = \phi_{i,n_i}^- z^{n_i} \exp((q^{-1}-q) \sum_{s<0} h_{i,s} z^s).
\end{equation*}
For $\alpha \in \hat{\Phi}$ define the the {\it shifted root vector} $E_{\alpha}^{\mu} \in \CU_{\mu}(\Gaff)$  as follows.
\begin{itemize}
\item[(i)] If $\alpha \in \hat{\Phi}_+$, then by Proposition \ref{prop: root vectors}(i) the root vector $E_{\alpha}$ belongs to $U_q^+(L\Glie)$ and under the identification \eqref{identification subalgebras} corresponds to an element $E_{\alpha}^{\mu}$ of $\CU_{\mu}^+(\Gaff)$.
\item[(ii)] If $\alpha = s\delta - \beta  \in \hat{\Phi}_-$ with $\beta \in \Phi$, then by Proposition \ref{prop: root vectors}(ii) the normalized root vector $k_{\beta} E_{\alpha}$ belongs to $U_q^-(L\Glie)$ and under the identification \eqref{identification subalgebras} corresponds to an element $\overline{E}_{\alpha}^{\mu}$ of $\CU_{\mu}^-(\Gaff)$. Set $E_{\alpha}^{\mu} := \phi_{-\beta}^+ \overline{E}_{\alpha}^{\mu}$. 
\end{itemize}
\end{defi}

Antidominantly shifted quantum affine algebras contain the upper Borel subalgebra as a subalgebra, as shown in \cite[Proposition H.1(a)]{FT} for $\Glie$ of type A and in \cite[Proposition 3.4]{H} for general types. We state their results in a different form. 

\begin{prop}\cite{FT,H}  \label{prop: shifted upper Borel}
For $\mu$ an antidominant coweight, we have an injective algebra homomorphism $\jmath_{\mu}: \Borel \longrightarrow \CU_{\mu}(\Gaff)$ defined by
 \begin{gather*}
E_{\alpha} \mapsto E_{\alpha}^{\mu}, \quad h_{i,s} \mapsto h_{i,s},\quad k_i \mapsto \phi_{i,0}^+ \quad \mathrm{for}\ i \in I,\ s \in \BZ_{>0}\ \mathrm{and}\ \alpha \in \hat{\Phi}.
 \end{gather*}
\end{prop}
\begin{proof} 
We explain why root vectors are preserved.
    Recall from \cite[\S 2.2]{HJ} the asymptotic algebra defined as the subalgebra of the quantum loop algebra generated by 
$$x_{i,m}^+,\quad k_i^{-1} x_{i,m}^-, \quad k_i^{-1} \phi_{i,m}^{\pm},\quad k_i^{-1} \quad \textrm{for $i \in I$ and $m \in \BZ$}.  $$
Let $\mathcal{I}$ denote the intersection of the asymptotic algebra with the upper Borel subalgebra. 
By Proposition \ref{prop: root vectors}, the vectors $E_{\alpha}$ and $h_{i,s}$ for $\alpha \in \hat{\Phi}$ and $(i, s) \in I \times \BZ_{>0}$ belong to $\mathcal{I}$. It is known from \cite[Proposition 2.1]{HJ} that $\Borel$ can be obtained from $\mathcal{I}$ by localization at the $k_i^{-1}$. By \cite[Proposition 3.3]{H}, for $\mu$ antidominant there is a natural algebra homomorphism from $\mathcal{I}$ to $\CU_{\mu}(\Gaff)$ which maps 
$$E_{\alpha}\mapsto E_{\alpha}^{\mu},\quad h_{i,s}\mapsto h_{i,s},\quad k_i^{-1} \mapsto (\phi_{i,0}^+)^{-1}.  $$ 
After localization, this induces the desired algebra homomorphism. 
\end{proof}

\subsection{Representations of the upper Borel subalgebra} We review results on highest/lowest representations of the upper Borel subalgebra $\Borel$ from \cite{HJ}. The basic notions and constructions in Subsection \ref{ss: highest weight}, including weight spaces, weight graded modules, top/bottom graded modules and highest/lowest $\ell$-weight modules, carry over to modules over $\Borel$ with the following modifications. 

 The set of $\ell$-weights is $\mathfrak{t}_{\ell}^* := (\BC[[z]]^{\times})^I$, namely the set of $I$-tuples of invertible power series in $z$. It is a group under component-wise multiplication. An $\ell$-weight is {\it rational} if each component is the Taylor expansion at $z = 0$ of a rational function. The rational $\ell$-weights form a subgroup of $\mathfrak{t}_{\ell}^*$, denoted by $\mathfrak{r}$. It is identified with the disjoint union of the $\mathfrak{r}_{\mu}$ over all coweights $\mu$ defined after Proposition \ref{prop: rationality shifted}.
 
 Given $\Be = (\Be_i(z))_{i\in I} \in \mathfrak{t}_{\ell}^*$, the highest weight Verma module $M(\Be)$ is defined to be the $\Borel$-module generated by $\omega$ subject to relations
$$ \phi_i^+(z) \omega = \Be_i(z) \omega,\quad E_{\alpha} \omega = 0 \quad \textrm{for $i \in I$ and $\alpha \in \hat{\Phi}_+$}. $$
Its irreducible quotient is $L(\Be)$. Replacing the condition $\alpha \in \hat{\Phi}_+$ with $\alpha \in \hat{\Phi}_-$, we obtain the lowest weight Verma module $M'(\Be)$ and its irreducible quotient $L'(\Be)$.

\begin{prop}\cite{HJ,H}  \label{prop: rationality}
Let $\mu$ be an antidominant coweight and $\Be \in \mathfrak{t}_{\ell}^*$.
\begin{itemize}
\item[(i)] The $\ell$-weight $\Be$ is rational if and only if all weight spaces of the $\Borel$-module $L(\Be)$ are finite-dimensional. The same statement holds for $L'(\Be)$.
\item[(ii)] If $\Be \in \mathfrak{r}_{\mu}$, then the pullback of the $\CU_{\mu}(\Gaff)$-module $L_{\mu}(\Be)$ along $\jmath_{\mu}$ is irreducible and isomorphic to the $\Borel$-module $L(\Be)$.
\end{itemize}
\end{prop}

Let $M$ be a weight graded $\Borel$-module whose weight spaces are finite-dimensional. Its {\it graded dual} $M^{\vee}$ is the vector space $\oplus_{\lambda \in \mathfrak{t}^*} M_{\lambda}^*$  with $\Borel$-action:
$$ \langle x f, v\rangle := \langle f, S(x) v \rangle \quad \mathrm{for}\ x \in \Borel,\ f \in M^{\vee},\ v \in M  $$
where $S: \Borel \longrightarrow \Borel$ is the antipode and $\langle, \rangle: M^{\vee} \times M \longrightarrow \BC$ is the natural pairing. From $S(k_i) = k_i^{-1}$ we obtain that $(M^{\vee})_{\lambda} = (M_{-\lambda})^*$ for $\lambda \in \mathfrak{t}^*$ and the graded dual is weight graded by finite-dimensional weight spaces.  Replacing $S$ with $S^{-1}$ we get on the same underlying space as $M^{\vee}$ another module structure, denoted by $M^{\wedge}$.

\begin{rem}  \label{rem: dual}
Let $\Be \in \mathfrak{r}$ be a rational $\ell$-weight. We have $\Borel$-module isomorphisms
$$ (L'(\Be))^{\wedge} \cong L(\Be^{-1}),\quad (L'(\Be))^{\vee}  \cong L(\tau_{q^{4\kappa}}(\Be)^{-1}). $$
The first isomorphism is \cite[Proposition 3.18]{HJ}, and the second follows from Eq.\eqref{Drinfeld rel: antipode} and the observation that  $M^{\vee}$ is the pullback module $(S^2)^* (M^{\wedge})$.
\end{rem}

Recall from Definition \ref{defi: prefund} the set $\mathfrak{d}$ of monomials in the prefundamental $\ell$-weights $\Psi_{i,a}$ for $i \in I$ and $\BC^{\times}$. For $\Bp \in \mathfrak{d}$ and $(i, s) \in I \times \BZ_{>0}$ let $\lambda_{i,s}^{\Bp}$ denote the eigenvalue of $h_{i,s}$ on a highest $\ell$-weight vector of the irreducible $\Borel$-module $L(\Bp)$, namely,
\begin{equation}  \label{def: lambda}
    \lambda_{i,s}^{\Bp} := \frac{c_1^s + c_2^s + \cdots + c_k^s}{s(q^{-1}-q)} \quad \mathrm{if}\ \Bp_i(z) = (1-zc_1)(1-zc_2) \cdots (1-zc_k).
\end{equation}
We record the following additional $\BN$-grading on the $\Borel$-modules $L(\Bp)$ and $L'(\Bp)$, obtained first in \cite[Theorem 6.1]{FH} for $\Psi_{i,a}$ and later in \cite[Proposition 5.11]{FJMM}. 

\begin{theorem}\cite{FH, FJMM}  \label{thm: polynomiality prefund}
Let $\zeta$ be a coweight, $\Bp \in \mathfrak{d}_{\zeta}$ and $M$ be the $\Borel$-module $L'(\Bp)$ with $\omega$ a lowest $\ell$-weight vector. The exists a $\BN$-grading $M = \oplus_{n\in \BN} M_n$ such that
\begin{itemize}
\item[(i)] each $M_n$ is stable by the $k_i^{\pm 1}$ for $i \in I$ and $M_0 = \BC \omega$;
\item[(ii)] for $\beta \in \Phi,\ i \in I,\ s >0$ and $t \geq 0$ we have (set $M_n = \{0\}$ for $n<0$)
\begin{gather*}
E_{s\delta-\beta} M_n \subset M_{n-s}, \quad (h_{i,s} - \lambda_{i,s}^{\Bp}) M_n \subset M_{n-s},  \quad E_{t\delta+\beta} M_n \subset \sum_{p=0}^{\langle\zeta,\beta \rangle} M_{n-t+p}.
\end{gather*}
\end{itemize}
Similar $\BN$-grading exists for $M = L(\Bp)$, by replacing $M_{n-s}$ and $M_{n-t+p}$ respectively with $M_{n+s}$ and $M_{n+t-p}$ at the right-hand sides of the above relations.  
\end{theorem}

\begin{example}  \label{example: prefund sl2} \cite{HJ} 
Let $\mathfrak{g} = sl_2$ so that $\kappa = 1$. On the infinite-dimensional vector space $\oplus_{j\in \BN} \BC v_j$ is realized the lowest $\ell$-weight $\Borel$-module $L'(\Psi_{1,1})$ with action $(n\in \BN)$
\begin{gather*}
x_{1,n}^+ v_j = \delta_{n0} v_{j+1},\quad x_{1,n+1}^- v_j = \delta_{n0} \frac{(j)_q}{q-q^{-1}}  v_{j-1}, \quad \phi_1^+(z) v_j = q^{2j}(1-z) v_j.
\end{gather*}  
On the same space is realized the highest $\ell$-weight $\Borel$-module $L(\Psi_{1,1})$:
\begin{gather*}
x_{1,n}^+ v_j = \delta_{n0} v_{j-1},\quad x_{1,n+1}^- v_j = \delta_{n0} \frac{(j+1)_{q^{-1}}}{q^{-1}-q} v_{j+1},\quad \phi_1^+(z) v_j = q^{-2j} (1-z) v_j.
\end{gather*}
 By setting $v_j$ to be of degree $j$, we get the $\BN$-grading of Theorem \ref{thm: polynomiality prefund}.
The graded dual $L'(\Psi_{1,1})^{\vee}$ is the highest $\ell$-weight $\Borel$-module $L(\Psi_{1,q^{-4}}^{-1})$ with action 
\begin{gather*} 
   \phi_1^+(z) v_j^* = q^{-2j}\frac{1-zq^{-2}}{(1-zq^{-2-2j})(1-zq^{-4-2j})} v_j^*, \\
    x_{1,m}^+ v_j^* = q^{-2m-2jm}(-q^{-2j}) v_{j-1}^*,  \quad x_{1,m}^- v_j^* = q^{-4m-2jm} \frac{q-q^{2j+3}}{(q-q^{-1})^2} v_{j+1}^*.
\end{gather*}
This also defines the $\CU_{-\varpi_1^{\vee}}(\Gaff)$-module $L_{-\varpi_1^{\vee}}(\Psi_{1,q^{-4}}^{-1})$.
\end{example}

\subsection{The universal R-matrix}  
The quantum loop algebra $\qaf$ is quasi-triangular as a Hopf algebra. 
We review basic properties of the universal R-matrix.

Recall the inverse $(\widetilde{B}_{ij}(q))_{i,j\in I}$ of the symmetric quantum Cartan matrix $([b_{ij}]_q)_{i,j\in I}$ and the Drinfeld--Cartan elements $h_{i,s}$ from Eq.\eqref{def: h}. The {\it abelian part} and the two {\it triangular parts} of the universal R-matrix are power series in $z$ defined as an exponential and two ordered infinite products of $q$-exponentials \cite{Damiani} 
\begin{align}  
    \CR_0(z) &:= \exp \left(\sum_{i,j\in I}\sum_{s>0} \frac{ s(q^{-1}-q)\widetilde{B}_{ij}(q^s)}{[s]_q} h_{i,s} \otimes h_{j,-s} z^s \right),  \label{def: R0}  \\
    \CR_{\pm}(z) &:= \prod_{n \in \BN,\ \beta \in \BQ:\ n\delta+\beta \in \hat{\Phi}_{\pm}}^{\prec} \exp_{q_{\beta}} ((q_{\beta}^{-1}-q_{\beta})   E_{n\delta+\beta} \otimes F_{n\delta+\beta} z^n). \label{def: R+-}
\end{align}
 Here $q_{\beta} := q^{\frac{(\beta,\beta)}{2}}$ for $\beta \in \BQ$ and the $q$-exponential is defined by $ \exp_q(x) :=  \sum_{n\geq 0} \frac{1}{(n)_q!} x^n$. The abelian part $\CR_0(z)$ and the lower triangular part $\CR_-(z)$ have coefficients in the ordinary tensor product algebra  $\Borel \otimes \lBorel$. The upper triangular part $\CR_+(z)$ has coefficients in a completion of $\Borel \otimes \lBorel$ with respect to the weight grading:
$$ \Borel \stimes \lBorel := \sum_{\alpha,\beta\in\BQ} \left( \prod_{\gamma \in \BQ_+} \Borel_{\alpha+\gamma} \otimes \lBorel_{\beta-\gamma} \right) \subset \prod_{\alpha,\beta \in \BQ} \Borel_{\alpha} \otimes \lBorel_{\beta}. $$
 It is an algebra and contains $\Borel \otimes \lBorel$ as a subalgebra by setting $\gamma = 0$. More generally, $M \stimes N$ makes sense for $\BQ$-graded vector spaces $M$ and $N$. 
 
The {\it reduced part} of the universal R-matrix is given by the factorization \cite{Damiani}:
\begin{equation} \label{def: R bar}
\barR(z) := \CR_+(z) \CR_0(z) \CR_-(z) \in (\Borel \stimes \lBorel)[[z]] 
\end{equation}
We need a further completion of $\Borel \stimes \lBorel$, called $q$-completion, by adding an invertible element $q^{t_{\infty}}$, whose inverse is denoted by $q^{-t_{\infty}}$, such that for $\beta, \gamma \in \BQ$:
\begin{equation}  \label{def: t infinity}
q^{-t_{\infty}} (x \otimes y) q^{t_{\infty}} = q^{-(\beta,\gamma)} x k_{-\gamma} \otimes y k_{-\beta} \quad \textrm{if $x \in \Borel_{\beta}$ and $y \in \lBorel_{\gamma}$}.
\end{equation}
See the paragraph above Proposition \ref{prop: root vectors} for the definition of $k_{-\gamma}$ and $k_{-\beta}$.

The Drinfeld--Jimbo coproduct for the Drinfeld--Cartan elements $h_{i,s}$ can be computed in terms of $\CR_{\pm}(z)$ by a twistor construction. Let $\sigma: x \otimes y \mapsto y \otimes x$ be the flip map and $\Delta^{\mathrm{cop}} := \sigma \circ \Delta$ be the opposite coproduct. Recall $\BQ_+ = \sum_{i\in I} \BN \alpha_i$ and $\BQ_> = \BQ_+ \setminus \{0\}$. 
\begin{prop}\cite{Damiani, KT, EKP} \label{prop: twist}
Let $i \in I$ and $s \neq 0$ we have
\begin{gather*}
(\mathrm{Id} \otimes \tau_z) \circ \Delta^{\mathrm{cop}}(h_{i,s}) = \CR_+(z) \times (h_{i,s} \otimes 1 + z^{-s} \otimes h_{i,s}) \times \CR_+(z)^{-1}, \\
(\mathrm{Id} \otimes \tau_z) \circ \Delta(h_{i,s}) = q^{t_{\infty}}\CR_-(z)^{-1} \times (h_{i,s} \otimes 1 +z^{-s} \otimes h_{i,s}) \times \CR_-(z)q^{-t_{\infty}}, \\
\Delta(h_{i,s}) \equiv 1 \otimes h_{i,s} + h_{i,s}\otimes 1 \ \mathrm{mod}. \sum_{\beta \in \BQ_>} \qaf_{-\beta} \otimes \qaf_{\beta}.
\end{gather*}
\end{prop} 
\begin{proof}
The first two equations are \cite[Proposition 3.8]{EKP} and \cite[Theorem 8.1]{KT}. Their right-hand sides are Laurent seires in $z$ with coefficients in the completed algebra $\qaf \stimes \qaf$.
The third relation follows from \cite[Proposition 7.1]{Damiani}.
\end{proof}
The full universal R-matrix is a power series in $z$ with coefficients in the $q$-completion:
\begin{equation} \label{def: R}
\CR(z) := \barR(z) q^{-t_{\infty}}.
\end{equation}
In what follows, we shall evaluate the first tensor factor of $\CR(z)$ at a $\Borel$-module $N$. To make sense of the $q$-completion at the level of representations, we need to put extra conditions on $N$. The following notion of root grading is weaker than the Yangian case in Definition \ref{defi: root graded Yangian} but suffices for our purpose.
\begin{defi} \label{defi: root graded}
A module over $\Borel$ is called {\it root graded} if it is weight graded and all its weights belong to $\BQ$.  Positive/negative root graded modules are defined in the obvious way using $\BQ_{\pm}$. If $N$ is a root graded $\Borel$-module and $V$ is a $\lBorel$-module, then define the linear operators $q^{\pm t_{\infty}}$ on $N\otimes V$ by
$$q^{\pm t_{\infty}}(u \otimes v) := u \otimes k_{\beta}^{\pm 1}(v) \qquad \textrm{for $\beta \in \BQ,\ u \in N_{\beta}$ and $v \in V$.} $$ 
\end{defi}
The above definition is made so that Eq.\eqref{def: t infinity} holds true as operators on $N \otimes V$. 

Given a $\qaf$-module $V$, one attaches as in Definition \ref{defi: deformed module} its deformed module structure $V_z$ on $V[z,z^{-1}]$.
Let $N$ be a $\Borel$-module. Then the $\Borel$-module $N \otimes V_z$ is the pullback of $N \otimes V[z,z^{-1}]$, viewed as a $\Borel\otimes \qaf[z,z^{-1}]$-module by scalar extension, along the algebra homomorphism 
$$ (\mathrm{Id} \otimes \tau_z) \Delta: \Borel \longrightarrow \Borel \otimes \qaf[z,z^{-1}].  $$
The space $(N \otimes V)((z))$ of Laurent series being a module over $\Borel\otimes \qaf[z,z^{-1}]$ again by scalar extension, its pullback along the above algebra homomorphism defines a $\Borel$-module, denoted by $\overline{N \otimes V_z}$, which is a completion of the tensor product module $N \otimes V_z$.  Similarly, the $\Borel$-module $\overline{V_z\otimes N}$ is defined via $(\tau_z \otimes \mathrm{Id}) \Delta$. 

The next proposition follows from the quasi-triangularity property \cite[(4.5)]{FR0}
$$(\mathrm{Id} \otimes \tau_z) \circ \Delta^{\mathrm{cop}}(x) = \CR(z) \times (\mathrm{Id} \otimes \tau_z) \circ \Delta(x) \times \CR(z)^{-1} \quad \mathrm{for}\ x \in \Borel. $$
\begin{prop}\cite[(4.10)]{FR0} \label{prop: R-matrix general}
Let $N$ be a negative root graded $\Borel$-module and $V$ be a $\qaf$-module. Then $\CR(z)$ sends $N \otimes V$ to $(N\otimes V)[[z]]$ and its composition with the flip map $\sigma$ induces a $\BC((z))$-linear isomorphism of $\Borel$-modules
$$ \check{R}_{N,V}(z) =  \sigma\circ \CR(z)|_{N\otimes V}: \overline{N \otimes V_z} \longrightarrow \overline{V_z \otimes N}. $$
\end{prop}
 The assumption \lq\lq negative" implies that $\CR_+(z)$ maps $N \otimes V$ to $(N\otimes V)[[z]]$. 
 
\subsection{Monodromy matrices}  \label{ss: monodromy matrix}
In this subsection as a modification of Proposition \ref{prop: R-matrix general} we evaluate the first tensor factor of the universal R-matrix at more general root graded representations of the upper Borel subalgebra. The main examples we will study later will be highest $\ell$-weight modules and lowest $\ell$-weight modules. 

\begin{defi}\cite[(4.32)]{FR0} \label{defi: monodromy matrix}
Let $M$ be a $\Borel$-module which is root graded by finite-dimensional weight spaces and equipped with a weight basis $\mathcal{B}$. The {\it monodromy matrix} associated to $M$ is the matrix $(t_{b_1,b_2}(z))_{b_1,b_2 \in \mathcal{B}}$ with entries in $\lBorel[[z]]$ such that 
 \begin{equation} \label{R: monodromy matrix}
\CR(z) (b_2 \otimes 1) = \sum_{b_1 \in \mathcal{B}} b_1 \otimes t_{b_1,b_2}(z)  \quad \mathrm{for}\ b_2 \in \mathcal{B}. 
\end{equation}
\end{defi}
To make sense of the definition, consider the completed tensor product
$$ M \stimes \lBorel = \sum_{\alpha,\beta\in\BQ} \left( \prod_{\gamma \in \BQ_+} M_{\alpha+\gamma} \otimes \lBorel_{\beta-\gamma} \right) \subset \prod_{\alpha,\beta \in \BQ} M_{\alpha} \otimes \lBorel_{\beta}. $$
It is a module over the completed algebra $\Borel \stimes \lBorel$ whose action is induced by the $\Borel$-module structure on $M$ and the regular representation of $\lBorel$. Furthermore, Definition \ref{defi: root graded} equips it with an action of $q^{\pm t_{\infty}}$. Therefore the full universal R-matrix $\CR(z)$ acts on $(M \stimes \lBorel)[[z]]$. For $b \in \mathcal{B}$ let $\wt(b) \in \BQ$ denote its weight. Since $\barR(z)$ is of total weight zero and each weight space of $M$ is finite-dimensional, both sides of Eq.\eqref{R: monodromy matrix} lie in the following subspace of $(M \stimes \lBorel)[[z]]$:
\begin{align*}
(\prod_{\alpha \in \BQ} M_{\alpha+\wt(b_2)} \otimes \lBorel_{-\alpha})[[z]] =\prod_{\alpha \in \BQ} M_{\alpha+\wt(b_2)} \otimes  \lBorel_{-\alpha}[[z]].
\end{align*}
In particular $t_{b_1,b_2}(z)$ belongs to $\lBorel_{\wt(b_2)-\wt(b_1)}[[z]]$. 

The coproduct of the monodromy matrix takes the usual matrix form \cite[(4.34)]{FR0}
\begin{equation}  \label{R: coproduct monodromy} 
\Delta(t_{b_1,b_2}(z)) = \sum_{b_3 \in \mathcal{B}} t_{b_3,b_2}(z) \otimes t_{b_1,b_3}(z) \in (\lBorel \otimes \lBorel)[[z]].
\end{equation}
 This is a consequence of the quasi-triangularity property \cite[(4.5)]{FR0}
$$ (\mathrm{Id} \otimes \Delta)(\mathcal{R}(z)) = \mathcal{R}_{13}(z) \mathcal{R}_{12}(z) $$
where $\CR_{12}(z) := \CR(z) \otimes 1$ and $\CR_{13}(z) := (\mathrm{Id} \otimes \sigma)(\CR_{12}(z))$.

Consider the graded dual module $M^{\vee}$. For $b \in \mathcal{B}$, let $b^{\vee}$ be  the linear form $b_1 \mapsto \delta_{b_1b}$ on $M$.  Then $\mathcal{B}^{\vee} := \{b^{\vee}: b \in \mathcal{B}\}$ forms a basis of the module $M^{\vee}$ and the weight of $b^*$ is opposite to the weight of $b$. Therefore the module $M^{\vee}$ is root graded by finite-dimensional weight spaces. Applying Definition \ref{defi: monodromy matrix} to $(M^{\vee}, \mathcal{B}^{\vee})$ we obtain another matrix  $(t_{b_1,b_2}^{\vee}(z))_{b_1,b_2 \in \mathcal{B}}$ with entries in $\lBorel[[z]]$ such that for $b_2 \in \mathcal{B}$:
\begin{equation}  \label{R: dual monodromy}
\CR(z) (b_2^{\vee} \otimes 1) = \sum_{b_1 \in \mathcal{B}} b_1^{\vee} \otimes t^{\vee}_{b_1,b_2}(z)  \in (M^{\vee} \stimes \lBorel)[[z]]. 
\end{equation}
The monodromy matrix of $M^{\vee}$ can be computed alternatively \cite[(4.13)]{FR0} 
\begin{equation} \label{R: inverse monodromy}
\CR(z)^{-1} (b_2 \otimes 1) = \sum_{b_1 \in \mathcal{B}} b_1 \otimes t^{\vee}_{b_2,b_1}(z)\in (M \stimes \lBorel)[[z]].
\end{equation}
It is a consequence of the following quasi-triangularity property \cite[(4.12)]{FR0}
$$ (S \otimes \mathrm{Id})(\CR(z)) = \CR(z)^{-1} = q^{t_{\infty}} \CR_-(z)^{-1} \CR_0(z)^{-1} \CR_+(z)^{-1}. $$

Consider the graded dual $M^{\wedge}$.
For $b\in \mathcal{B}$ let $b^{\wedge} \in M^*$ be the linear form $b_1 \mapsto \delta_{b_1b}$. Then $\{b^{\wedge}: b \in \mathcal{B}\}$ forms a basis of $M^{\wedge}$. The above quasi-triangularity implies:
\begin{equation}  \label{R: inverse dual}
    \CR(z)^{-1} (b_2^{\wedge} \otimes 1) = \sum_{b_1 \in \mathcal{B}} b_1^{\wedge} \otimes t_{b_2,b_1}(z) \in (M^{\wedge} \stimes \lBorel)[[z]].
\end{equation}
\section{Theta series from monodromy matrices}  \label{sec: Theta}
In this section we define and study Theta series for the quantum loop algebra. These are obtained from the coproduct of the T-series introduced by Frenkel--Hernandez \cite{FH} in their study of Baxter's Q-operators for quantum integrable systems. The main result of this section is polynomiality of Theta series.

Recall the Drinfeld--Cartan elements $h_{i,s}$ for $i \in I$ and $s \in \BZ_{\neq 0}$ from Eq.\eqref{def: h} and the inverse $(\widetilde{B}_{ij}(q))_{i,j\in I}$ of the symmetric quantum Cartan matrix $([b_{ij}]_q)_{i,j\in I}$. The T-series $T_i(z)$, for $i \in I$, is an invertible power series in $z$ with coefficients in the lower Borel subalgebra  \cite[Proposition 5.5]{FH}:
\begin{equation}  \label{def: Ti}
T_i(z) := \exp\left(\sum_{s>0} \frac{\tilde{h}_{i,-s}}{[s]_q} z^s \right) \quad \mathrm{where}\ \tilde{h}_{i,-s} :=  \sum_{j\in I} \widetilde{B}_{ij}(q^s) h_{j,-s}. 
\end{equation}
We refer to the $\tilde{h}_{i,-s}$ for $i \in I$ and $s > 0$ as modified Drinfeld--Cartan elements. Actually the other half $\tilde{h}_{i,s}$ can be defined but they are not used in this paper.
\begin{lem}   \label{lem: T x h}
For $i, j \in I$ and $m \in \BZ$  the following relations hold in $\qaf[[z]]$:
\begin{gather*}
T_i(z) x_{j,m}^- T_i(z)^{-1} = x_{j,m}^- - \delta_{ij} z x_{j,m-1}^-, \quad T_i(z)^{-1} x_{j,m}^+ T_i(z) = x_{j,m}^+ - \delta_{ij} z x_{j,m-1}^+, \\
\exp\left((q^{-1}-q) \sum_{s>0} h_{i,-s} z^s \right)= \prod_{k\in I} \frac{T_k(zq^{-b_{ik}})}{T_k(zq^{b_{ik}})}.
\end{gather*}
\end{lem}
\begin{proof}
For $s > 0$, by Eqs.\eqref{Drinfeld rel: h x} and \eqref{def: Ti} we have (see also \cite[Lemma 3.1]{FM})
\begin{align*}
[\frac{\tilde{h}_{i,-s}}{[s]_q}, x_{j,m}^{\pm}] &= \sum_{k\in I} \widetilde{B}_{ik}(q^s) [\frac{h_{k,-s}}{[s]_q}, x_{j,m}^{\pm}] =  \pm\sum_{k\in I} \widetilde{B}_{ik}(q^s) \frac{[s b_{kj}]_q}{s[s]_q} x_{j,m-s}^{\pm}  \\
&= \pm\sum_{k\in I} \widetilde{B}_{ik}(q^s) \frac{[b_{kj}]_{q^s}}{s} x_{j,m-s}^{\pm} = \pm \frac{\delta_{ij}}{s} x_{j,m-s}^{\pm}.
\end{align*}
 Therefore
$T_i(z)^{\pm 1} x_{j,m}^{\mp} T_i(z)^{\mp 1} = \sum_{n\geq 0} a_n x_{j,m-n}^{\mp} z^n$
 where the coefficients $a_n$ are encoded in the following power series
\begin{align*}
\sum_{n \geq 0} a_n z^n = \exp(- \sum_{s>0} \frac{\delta_{ij}}{s} z^s) = 1 - \delta_{ij} z.
\end{align*}
The third equation follows from 
$  h_{i,-s} = \sum_{k\in I} [b_{ik}]_{q^s} \tilde{h}_{k,-s}$.
\end{proof}
The third equation of the lemma is a quantum affine analog of Eq.\eqref{xi T Yangian}  at the level of algebras. In the Yangian case we only have T-operators at the level of representations.
 
\begin{defi}
To a polynomial $\ell$-weight $\Bp = \Psi_{i_1,a_1} \Psi_{i_2,a_2} \cdots \Psi_{i_n, a_n} \in \mathfrak{d}$ we associate its {\it T-series} and {\it Theta series} as the following power series in $z$ with leading term 1:
\begin{align}
T_{\Bp}(z) &:= T_{i_1}(za_1) T_{i_2}(za_2) \cdots T_{i_n}(za_n)  \in \lBorel[[z]],  \label{def: T-series}  \\
\Theta_{\Bp}(z) &:= (1\otimes T_{\Bp}(z)^{-1}) \times \Delta(T_{\Bp}(z)) \times (T_{\Bp}(z)^{-1} \otimes 1) \in \lBorel^{\otimes 2}[[z]].  \label{def: Theta series}
\end{align}
For $\beta \in \BQ_+$, let $\Theta_{\Bp,\beta}(z) \in (\lBorel_{-\beta} \otimes \lBorel_{\beta})[[z]]$ be the $\beta$-component of $\Theta_{\Bp}(z)$. Similarly, let $\widetilde{\Theta}_{\Bp,\beta}(z)$ be the $\beta$-component of the inverse Theta series $\Theta_{\Bp}(z)^{-1}$.
\end{defi}
Since $\tilde{h}_{i,-s}$ is a linear combination of the $h_{j,-s}$, its coproduct is quasi-primitive as in the case of $h_{j,-s}$ in Proposition \ref{prop: twist}. We obtain therefore
$$\Theta_{\Bp,0}(z) = 1,\quad \Theta_{\Bp}(z) = \sum_{\beta \in \BQ_+} \Theta_{\Bp,\beta}(z),\quad\Theta_{\Bp}(z)^{-1} = \sum_{\beta\in \BQ_+} \widetilde{\Theta}_{\Bp,\beta}(z). $$
As a consequence of Lemma \ref{lem: T x h}, for $i \in I$ we have in $\qaf[[z, w, w^{-1}]]$:
\begin{equation}  \label{rel: T x}
T_{\Bp}(z) x_i^-(w) = \Bp_i(zw)  x_i^-(w) T_{\Bp}(z),\quad  x_i^+(w) T_{\Bp}(z) = \Bp_i(zw) T_{\Bp}(z) x_i^+(w).
\end{equation}
Let $\mathrm{Ad}_{T_{\Bp}(z)}$ denote the conjugation on the algebra $\lBorel[[z]]$ by the invertible element $T_{\Bp}(z)$. As in Proposition \ref{prop: multip}(iii), Theta series are multiplicative:
$$ \Theta_{\Bm\Bn}(z) = (\mathrm{Id} \otimes \mathrm{Ad}_{T_{\Bn}(z)}^{-1})(\Theta_{\Bm}(z)) \times  (\mathrm{Ad}_{T_{\Bm}(z)} \otimes \mathrm{Id})(\Theta_{\Bn}(z))\quad \mathrm{for}\ \Bm, \Bn \in \mathfrak{d}. $$

\begin{rem} \label{rem: from R to T}
The T-series $T_{\Bp}(z)$ is obtained from $\CR_0(z)$ by substituting in the first tensor factor $h_{i,s}$, for $i \in I$ and $s \in \BZ_{>0}$,  with the complex number $\lambda_{i,s}^{\Bp}$ defined in Eq.\eqref{def: lambda}. In other words, if $\omega$ is a highest $\ell$-weight vector of $L(\Bp)$, then 
$$ \CR_0(z) (\omega \otimes 1) = \omega \otimes T_{\Bp}(z) \in L(\Bp) \otimes \lBorel[[z]]. $$
\end{rem}

\subsection{Polynomiality of monodromy matrices} \label{ss: poly mono}
 In this subsection we establish the polynomiality for the monodromy matrix associated to the lowest $\ell$-weight irreducible $\Borel$-module of Theorem \ref{thm: polynomiality prefund}. This is inspired by the work of Frenkel--Hernandez \cite{FH} on the transfer matrix (which is a weighted sum of the diagonal entries of the monodromy matrix specialized at a finite-dimensional $\qaf$-module).

Throughout this subsection, we fix a dominant coweight $\zeta$ and a polynomial $\ell$-weight $\Bp \in \mathfrak{d}_{\zeta}$. Take $M = L'(\Bp)$ to be the irreducible $\Borel$-module of lowest $\ell$-weight $\Bp$ as in Theorem \ref{thm: polynomiality prefund}. Let $b_0$ be a lowest $\ell$-weight vector of $M$, and extend it to a weight basis $\mathcal{B}$  compatible with the $\BN$-grading. Each basis vector $b \in \mathcal{B}$ is equipped with the weight $\wt(b) \in \BQ_+$ and the degree $p(b) \in \BN$. Notably, $\wt(b_0) = 0$ and $p(b_0) = 0$.

Since each weight space of $M$ is finite-dimensional by Proposition \ref{prop: rationality}, one can apply the monodromy-matrix construction of Definition \ref{defi: monodromy matrix} to $M$. As a result, we have the power series $t_{b_1,b_2}(z) \in \lBorel[[z]]$ for $b_1, b_2 \in \mathcal{B}$ defined by Eq.\eqref{R: monodromy matrix}. Introduce also the power series $t_{b_1,b_2}^{\bullet}(z) \in \lBorel[[z]]$ for $\bullet \in \{+,0,-\}$ by
\begin{align*}   
\CR_{\bullet}(z) (b_2 \otimes 1) &= \sum_{b_1 \in \mathcal{B}} b_1 \otimes t_{b_1,b_2}^{\bullet}(z).
\end{align*}
From the factorization \eqref{def: R}, we derive the following Gauss decomposition for the monodromy matrix $t$, as in \cite[line 6 of \S 7.4]{FH}:
\begin{equation}
t_{b_1,b_2}(z) = \sum_{b_3,b_4 \in \mathcal{B}} t^+_{b_1,b_3}(z) t^0_{b_3,b_4}(z) t_{b_4,b_2}^-(z) k_{\wt(b_2)}^{-1}.  \label{Gauss decomposition}
\end{equation}
 The terminology follows \cite{DF}; see \cite[Lemma 3.6]{FM0} for a similar Gauss decomposition for the monodromy matrix of a finite-dimensional irreducible $\qaf$-module. All the matrices in \cite{DF,FM0} are of finite size and their entries are power series in $z$. In contrast, our matrices are of infinite size and the entries of $t^{\pm}$ will be shown in Proposition \ref{prop: poly mono} below to be polynomial.

Recall from Subsection \ref{ss: root vectors} the subalgebras $U_q^{\bullet}(\mathfrak{c}) = \lBorel \cap U_q^{\bullet}(L\Glie)$ for $\bullet \in \{+,0,-\}$.
Let $\mathcal{F}$ denote the subalgebra of $\lBorel$ generated by the $F_{\alpha}$ for $\alpha \in \hat{\Phi}_-$. Then comparing with the basis of $U_q^+(\mathfrak{c})$ in Proposition \ref{prop: root vectors}(iv) we obtain 
\begin{equation}  \label{lower Borel positive}
\mathcal{F} \cap \lBorel_{\beta} =: \mathcal{F}_{\beta} = U_q^+(\mathfrak{c})_{\beta} k_{\beta} \quad \mathrm{for}\ \beta \in \BQ.
\end{equation}

We summarize the basic properties of the monodromy matrix. Its proof is quite close to the arguments in \cite[\S 7.4]{FH}. The main difference is that we do not specialize the matrix entries to any representation of $\qaf$. Recall that $b_0 \in \mathcal{B}$ is fixed to be a lowest $\ell$-weight vector of $M$.
\begin{prop} \label{prop: poly mono}
Fix $b_1, b_2 \in \mathcal{B}$. Set $\beta := \wt(b_2) - \wt(b_1)$ and $d := p(b_2) - p(b_1)$.
\begin{itemize}
\item[(i)] If $t_{b_1,b_2}^+(z) \neq 0$, then $\beta \in \BQ_-$ and $t_{b_1,b_2}^+(z)$ is a $U_q^-(\mathfrak{c})_{\beta}$-valued polynomial in $z$ of degree bounded by $d - \langle \zeta, \beta\rangle$. If $\beta = 0$, then $t_{b_1,b_2}^+(z) = \delta_{b_1,b_2}$.
\item[(ii)] If $t_{b_1,b_2}^0(z) \neq 0$, then $\beta = 0$ and $T_{\Bp}(z)^{-1} t_{b_1,b_2}^0(z)$ is a $U_q^0(\mathfrak{c})$-valued polynomial of degree $d$. Moreover, $t_{b_1,b_1}^0(z) = T_{\Bp}(z)$.
\item[(iii)] If $t_{b_1,b_2}^-(z) \neq 0$, then $\beta \in \BQ_+$ and $t_{b_1,b_2}^-(z)$ is an $\mathcal{F}_{\beta}$-valued polynomial of degree $d$. If $\beta = 0$, then $t_{b_1,b_2}^-(z) = \delta_{b_1,b_2}$.
\item[(iv)] We have $t_{b_1,b_0}(z) =  t_{b_1,b_0}^+(z) T_{\Bp}(z)$ and $t_{b_0,b_2}(z) = T_{\Bp}(z) t_{b_0,b_2}^-(z) k_{\wt(b_2)}^{-1}$.
\end{itemize}
\end{prop}
\begin{proof}
Let $\pi_b: M \longrightarrow \BC b$ denote the projection with respect to the basis vector $b \in \mathcal{B}$.

\medskip

\noindent {\bf Part (i).}
From the definition \eqref{def: R+-} of $\CR_+(z)$ we get that $b_1 \otimes t_{b_1,b_2}^+(z)$ is a formal sum, with complex coefficients, of the following monomials 
$$ \pi_{b_1}(E_{s_1\delta+\beta_1} E_{s_2\delta+\beta_2} \cdots E_{s_m\delta+\beta_m} b_2) \otimes F_{s_1\delta+\beta_1} F_{s_2\delta+\beta_2} \cdots F_{s_m\delta+\beta_m} z^{s_1+s_2+\cdots+s_m} $$
where $m \geq 0, s_1,s_2,\cdots,s_m \geq 0$ and $\beta_1,\beta_2,\cdots \beta_m \in \Phi$ such that 
\begin{gather*}
\begin{cases}
\wt(b_1) = \wt(b_2) + \beta_1 + \beta_2 + \cdots + \beta_m, \\
p(b_1) \leq p(b_2) - s_1-s_2-\cdots-s_m + \langle \zeta, \beta_1+\beta_2+\cdots+\beta_m\rangle.
\end{cases}
\end{gather*}
The last condition comes from the relation $E_{s\delta+\beta} M_n \subset \sum_{k=0}^{\langle \zeta, \beta \rangle} M_{n-s+k}$ of Theorem \ref{thm: polynomiality prefund}. For fixed $b_1, b_2 \in \mathcal{B}$ there are finitely many  positive integers $s_1, s_2, \cdots, s_m$ and positive roots $\beta_1, \beta_2, \cdots, \beta_m$ satisfying the above conditions. Hence $t_{b_1,b_2}^+(z)$ is a finite sum of polynomials of degree bounded by $p(b_2) - p(b_1) + \langle \zeta, \wt(b_1)-\wt(b_2) \rangle$. 

If $\wt(b_1) = \wt(b_2)$, then necessarily  $m = 0$ and $b_1 \otimes t_{b_1,b_2}^+(z)$ is the leading term $1\otimes 1$ of $\CR_+(z)$ applied to $b_1\otimes 1$ and $t_{b_1,b_2}^+(z) = \delta_{b_1,b_2}$.

\medskip

\noindent {\bf Part (ii).} 
Recall from Eq.\eqref{def: lambda} the complex numbers $\lambda_{i,s}^{\Bp}$ for $(i, s) \in I \times \BZ_{>0}$ and set $\overline{h}_{i,s} := h_{i,s} - \lambda_{i,s}^{\Bp}$. Let $\barR_0(z)$ be obtained from $\CR_0(z)$ by substituting $h_{i,s} \mapsto \overline{h}_{i,s}$ in the first tensor factor. Define the power series $\overline{t}_{b_1,b_2}^0(z) \in U_q^0(\mathfrak{c})[[z]]$ in the obvious way:
$$\barR_0(z) (b_2 \otimes 1) = \sum_{b_1 \in \mathcal{B}} b_1 \otimes \overline{t}_{b_1,b_2}^0(z)  \quad \mathrm{for}\ b_2 \in \mathcal{B}.$$ 
By Remark \ref{rem: from R to T} we have
  $$  \CR_0(z) =  (1\otimes T_{\Bp}(z)) \times \barR_0(z). $$
Such a factorization implies that 
$$ t_{b_1,b_2}^0(z) = T_{\Bp}(z) \overline{t}_{b_1,b_2}^0(z).  $$
From the relation $\overline{h}_{i,s} M_n \subset M_{n-s}$ of Theorem \ref{thm: polynomiality prefund} and the definition \eqref{def: R0} of $\CR_0(z)$ we get that $b_1 \otimes \overline{t}_{b_1,b_2}^0(z)$ is a linear combination of the monomials
$$ \pi_{b_1}(\overline{h}_{i_1,s_1} \overline{h}_{i_2, s_2} \cdots \overline{h}_{i_m, s_m} b_2) \otimes h_{j_1,-s_1} h_{j_2,-s_2} \cdots h_{j_m,-s_m} z^{s_1+s_2+\cdots+s_m} $$
where $m\geq 0,\ i_1, i_2, \cdots, i_m, j_1, j_2, \cdots, j_m \in I$ and $s_1, s_2, \cdots, s_m > 0$ such that
$$ p(b_1) = p(b_2) - s_1- s_2 - \cdots - s_m,\quad \wt(b_1) = \wt(b_2). $$
This shows that $\overline{t}_{b_1,b_2}^0(z)$, if nonzero, is a polynomial in $z$ with coefficients in $U_q^0(\mathfrak{c})$ of degree $p(b_2) - p(b_1)$. If $b_1 = b_2$, then $m = 0$ and $b_1 \otimes \overline{t}_{b_1,b_1}^0(z)$ is the leading term $1\otimes 1$ of $\barR_0(z)$ applied to $b_1 \otimes 1$. This proves $\overline{t}_{b_1,b_1}^0(z) = 1$.   

\medskip

\noindent {\bf Part (iii).} From the relation $E_{s\delta-\beta} M_n \subset M_{n-s}$ of Theorem \ref{thm: polynomiality prefund} and the definition \eqref{def: R+-} of $\CR_-(z)$ we get that $b_1 \otimes t_{b_1,b_2}^-(z)$ is a linear combination of the monomials
$$ \pi_{b_1}(E_{s_1\delta-\beta_1} E_{s_2\delta-\beta_2} \cdots E_{s_m\delta-\beta_m} b_2) \otimes F_{s_1\delta-\beta_1} F_{s_2\delta-\beta_2} \cdots F_{s_m\delta-\beta_m} z^{s_1+s_2+\cdots+s_m} $$
where $m \geq 0,\ s_1, s_2,\cdots, s_m > 0$ and $\beta_1,\beta_2,\cdots \beta_m \in \Phi$ such that 
\begin{gather*}
\begin{cases}
\wt(b_1) = \wt(b_2) - \beta_1 - \beta_2 - \cdots - \beta_m, \\
p(b_1) = p(b_2) - s_1-s_2-\cdots-s_m.
\end{cases}
\end{gather*}
Part (iii) is now proved in the same way as part (ii).

\medskip

\noindent {\bf Part (iv).} Let us first take $b_1 = b_0$ in Eq.\eqref{Gauss decomposition}. Since $M$ is bottom graded, we have $\wt(b) \in \BQ_+$ and $\wt(b) = 0$ if and only if $b = b_0$. This together with part (i) forces $t_{b_0,b_3}^+(z) = \delta_{b_0,b_3}$. So in Eq.\eqref{Gauss decomposition} only the terms with $b_3 = b_0$ contribute. Again for weight reason and using part (ii), we have $t^0_{b_0,b_4}(z) = \delta_{b_0,b_4} T_{\Bp}(z)$. Therefore in Eq.\eqref{Gauss decomposition} only the single term with $b_3 = b_4 = b_0$ contributes and gives
\begin{equation*}  
t_{b_0,b_2}(z) = t_{b_0,b_0}^+(z) t^0_{b_0,b_0}(z) t_{b_0,b_2}^-(z) k_{\wt(b_2)}^{-1} = T_{\Bp}(z) t_{b_0,b_2}^-(z) k_{\wt(b_2)}^{-1}.
\end{equation*}

Next take $b_2 = b_0$ in Eq.\eqref{Gauss decomposition}. Only the single term with $b_3 = b_4 = b_0$ contributes:
 \begin{equation*}  
 t_{b_1,b_0}(z) =  t_{b_1,b_0}^+(z)  t_{b_0,b_0}^0(z) t_{b_0,b_0}^-(z) k_{\wt(b_0)}^{-1} = t_{b_1,b_0}^+(z)T_{\Bp}(z).
 \end{equation*}
This proves part (iv).
\end{proof}

\subsection{Polynomiality of Theta series} Recall the T-series and Theta series from Eqs.\eqref{def: T-series}--\eqref{def: Theta series}. In this subsection we establish polynomiality of Theta series and as a first application polynomiality of T-series acting on a tensor product of highest $\ell$-weight modules over $\qaf$.
\begin{theorem} \label{thm: polynomiality Theta}
Let $\zeta$ be a coweight and $\Bp \in \mathfrak{d}_{\zeta}$. For $\beta \in \BQ_+$, the power series $\Theta_{\Bp,\beta}(z)$ and $\widetilde{\Theta}_{\Bp,\beta}(z)$ are polynomials in $z$ with coefficients in $U_q^-(\mathfrak{c})_{-\beta} \otimes U_q^+(\mathfrak{c})_{\beta}$ of degree bounded by $\langle \zeta, \beta\rangle$.
\end{theorem}
\begin{proof}
Combining Proposition \ref{prop: poly mono}(iv) with Eq.\eqref{R: coproduct monodromy} we get 
\begin{align*} 
\Delta(T_{\Bp}(z)) &= \Delta(t_{b_0,b_0}(z)) = \sum_{b\in \mathcal{B}} t_{b,b_0}(z) \otimes t_{b_0,b}(z) \\
&= (1\otimes T_{\Bp}(z)) (\sum_{b\in \mathcal{B}} t_{b,b_0}^+(z) \otimes t_{b_0, b}^-(z)k_{\wt(b)}^{-1}) (T_{\Bp}(z) \otimes 1).
\end{align*}
This leads to the following formula for $\beta \in \BQ_+$:
\begin{equation}  \label{equ: Theta i beta}
\Theta_{\Bp,\beta}(z) = \sum_{b\in \mathcal{B}: \wt(b) = \beta} t_{b,b_0}^+(z) \otimes t_{b_0,b}^-(z) k_{\wt(b)}^{-1}.
\end{equation}
The right-hand side is a finite sum of polynomials in $U_q^-(\mathfrak{c})_{-\beta} \otimes U_q^+(\mathfrak{c})_{\beta}[z]$ of degree bounded by $-p(b)+\langle\zeta, \beta\rangle+p(b)$ by Proposition \ref{prop: poly mono}(i)--(iii) and Eq.\eqref{lower Borel positive}.

The polynomiality of the inverse $\Theta_{\Bp}(z)^{-1}$ follows from the polynomiality of $\Theta_{\Bp}(z)$ as in the beginning of the proof of Theorem \ref{thm: Yangian poly Theta}. 
\end{proof}

\begin{example}
Fix $i \in I$ and $\Bp = \Psi_{i,1}$. For $j \in I$ and $n \in \BN$ we claim that
$$ \Theta_{\Bp,n\alpha_j}(z) = \frac{\delta_{ij}}{(n)_{q_i}!} \left((q_i-q_i^{-1}) x_{i,0}^- \otimes x_{i,-1}^+z\right)^n. $$
Consider the $\Borel$-module $M := L'(\Psi_{1,1})$ as in Subsection \ref{ss: poly mono}. If $j \neq i$ then $ n\alpha_j$ is not a weight of $M$ and so $\Theta_{\Bp,n\alpha_j}(z) = 0$. On the other hand $M_{n\alpha_i}$ is one-dimensional and spanned by the vector $v_n := (x_{i,0}^+)^n b_0$ which we assume to be in $\mathcal{B}$. It suffices to compute $\Theta_{\Bp,n\alpha_i}(z) = t_{v_n,v_0}^+(z) \otimes t_{v_0,v_n}^-(z) k_i^{-n}$. 
Notice first that
$$ k_i v_m = q_i^{2m} v_m,\quad x_{i,n+1}^- v_m = \delta_{n0} \frac{(m)_{q_i}}{q_i-q_i^{-1}} v_{m-1} \quad \mathrm{for}\ n, m \in \BN. $$
To compute $t_{v_0, v_n}^-(z)$, we need to find the $v_0$-component of $\CR_-(z)(v_n \otimes 1)$, which for weight reason amounts to evaluating the $n\alpha_i$-weight component of 
\begin{align*}
\exp_{q_{\alpha_i}}((q_{\alpha_i}^{-1}-q_{\alpha_i}) E_{\delta-\alpha_i} \otimes F_{\delta-\alpha_i})  = \exp_{q_i}((q_i^{-1}-q_i)  k_i^{-1} x_{i,1}^- \otimes x_{i,-1}^+ k_i z), \\
\Rightarrow  v_0 \otimes t_{v_0,v_n}^-(z) = \frac{(q_i^{-1}-q_i)^n }{ (n)_{q_i}! }(k_i^{-1} x_{i,1}^- z)^n v_n \otimes (x_{i,-1}^+ k_i)^n = v_0 \otimes (-1)^n (x_{i,-1}^+)^n k_1^n z^n.
\end{align*}
Similarly, the $v_n$-component of $\CR_+(z)(v_0 \otimes 1)$ equals the $n\alpha_i$-weight component of 
\begin{align*}
\exp_{q_{\alpha_i}}((q_{\alpha_i}^{-1}-q_{\alpha_i}) E_{\alpha_i} \otimes F_{\alpha_i}) (v_n \otimes 1) = \exp_{q_i}((q_i^{-1}-q_i) x_{i,0}^+\otimes x_{i,0}^-)(v_n \otimes 1), \\
\Rightarrow v_n \otimes t_{v_n,v_0}^+(z) =\frac{(q_i^{-1}-q_i)^n }{ (n)_{q_i}! } (x_{i,0}^+)^n v_0 \otimes (x_{i,0}^-)^n = v_n \otimes \frac{(q_i^{-1}-q_i)^n }{ (n)_{q_i}! } (x_{i,0}^-)^n. 
\end{align*}
This gives the desired formulas for $\Theta_{\Bp,n\alpha_i}(z)$ in view of Eq.\eqref{equ: Theta i beta}. In the rank-one case we have the following coproduct formula in $U_q(L sl_2)$:
$$\Delta(T_1(z)) = (1\otimes T_1(z)) \exp_q((q-q^{-1}) x_{1,0}^- \otimes x_{1,-1}^+z) (T_1(z) \otimes 1). $$
\end{example}
In the Yangian case the polynomiality of Theta series follows from the intertwining property. In the present situation, the polynomiality of Theta series follows from the polynomiality of the monodromy matrix associated to $\Borel$-module $L'(\Bp)$. 

We close this section with a polynomiality result for the action $T_{\Bp}(z)$ on a tensor product of highest $\ell$-weight $\qaf$-modules. This result was known for tensor products of finite-dimensional irreducible $\qaf$-modules \cite[Theorem 5.17]{FH} and fusion products of irreducible modules in category $\mathcal{O}$ for shifted quantum affine algebras \cite[Theorem 9.12]{H}. The proof in \cite{FH} required a nontrivial reduction to thin modules depending on the type of $\Glie$. The proof in \cite{H} required an asymptotic limit construction. Our proof is uniform in all types and is an adaptation of  proof of Proposition \ref{prop: poly T Yangian} in the Yangian case based on the factorization \eqref{def: Theta series} and the commutation relations \eqref{rel: T x}.

\begin{defi}  \label{defi: f and g}
Fix $\Bp \in \mathfrak{d}$ and let $V$ be a $\qaf$-module. If $V$ is top graded, let $f_{\Bp}^V(z) \in 1 + z \BC[[z]]$ denote the eigenvalue of $T_{\Bp}(z) \in \qaf[[z]]$ associated to the top weight space of $V$. If $V$ is bottom graded, let $g_{\Bp}^V(z) \in 1 + z \BC[[z]]$ denote the eigenvalue of $T_{\Bp}(z)$ associated to the bottom weight space of $V$.
\end{defi}

\begin{prop}  \label{prop: poly T}
Let $\zeta$ be a coweight, $\Bp \in \mathfrak{d}_{\zeta},\ \lambda_0 \in \mathfrak{t}^*$ and $V$ be a $\qaf$-module.
\begin{itemize}
\item[(i)] If $V$ is a tensor product of highest $\ell$-weight $\qaf$-modules with top weight $\lambda_0$, then on each weight space $V_{\lambda_0 -\beta}$ with $\beta \in \BQ_+$, the operator $f_{\Bp}^V(z)^{-1} T_{\Bp}(z)$ is a polynomial in $z$ of degree $\langle \zeta,\beta\rangle$. 
\item[(ii)] If $V$ is a tensor product of lowest $\ell$-weight $\qaf$-modules with bottom weight $\lambda_0$, then on each weight space $V_{\lambda_0+\beta}$ with $\beta \in \BQ_+$, the operator $g_{\Bp}^V(z) T_{\Bp}(z)^{-1}$ is a polynomial in $z$ of degree $\langle \zeta,\beta\rangle$. 
\end{itemize} 
\end{prop}
\begin{proof}
In view of the proof of Proposition \ref{prop: poly T Yangian}, the only nontrivial part is the initial step: if $V$ is a highest $\ell$-weight module then the normalized operator $\overline{T}_{\Bp}^V(z) := f_{\Bp}^V(z)^{-1} T_{\Bp}(z)$ restricted to each weight space has the desired degree. For $i \in I$, let $c_i w^{n_i}$ denote the dominant term of the polynomial $\Bp_i(z)$. Then $n_i = \langle \zeta, \alpha_i\rangle$.  Eq.\eqref{rel: T x} implies that $T_{\Bp}(z) x_{i,m}^- T_{\Bp}(z)^{-1}$ is a polynomial in $z$ of dominant term $c_i x_{i,m-n_i}^- z^{n_i}$. 

Fix $0 \neq \omega_0 \in V_{\lambda_0}$.
The weight space $V_{\lambda_0-\beta}$ is spanned by the $x_{i_1,m_1}^- x_{i_2,m_2}^- \cdots x_{i_s,m_s}^- \omega_0$ such that $\beta = \alpha_{i_1}+\alpha_{i_2} + \cdots + \alpha_{i_s}$. Applying $\overline{T}_{\Bp}^V(z)$ to such a vector, we get 
$$ c_{i_1} z^{n_{i_1}} c_{i_2} z^{n_{i_2}} \cdots c_{i_s} z^{n_{i_s}} x_{i_1,m_1-n_{i_1}}^- x_{i_2,m_2-n_{i_2}}^- \cdots x_{i_s,m_s-n_{i_s}}^- \omega_0   $$
plus a polynomial in $z$ of degree strictly lower than $n_{i_1} + n_{i_2} + \cdots + n_{i_s} = \langle \zeta, \beta\rangle$. If the degree of $\overline{T}_{\Bp}^V(z)|_{V_{\lambda_0-\beta}}$ is strictly lower than $\langle \zeta, \beta\rangle$, then we have:
$$x_{i_1,m_1-n_{i_1}}^- x_{i_2,m_2-n_{i_2}}^- \cdots x_{i_s,m_s-n_{i_s}}^- \omega_0 = 0  $$
for all $m_1, m_2, \cdots, m_s \in \BZ$ and $i_1, i_2, \cdots, i_s \in I$ such that $\beta = \alpha_{i_1}+\alpha_{i_2} + \cdots + \alpha_{i_s}$. This implies $V_{\lambda_0-\beta} = \{0\}$, contradicting the hypothesis that $\lambda_0-\beta$ is a weight of $V$.
\end{proof}

\section{R-matrices for tensor product modules}  \label{sec: R dec}
In this section, for $V$ an arbitrary finite-dimensional module over the quantum loop algebra and for $W$ a suitable tensor product module over an antidominantly shifted quantum affine algebra, we establish a decomposition formula for the R-matrix $\check{R}_{W,V}(z)$ in terms of T-series, Theta series and another R-matrix. It is inspired by the decomposition formula for shifted Yangians in Theorem \ref{thm: T Theta R Yangian}, but their proofs are different.

For $\nu$ a coweight, as in Definition \ref{defi: trivial modules Yangians} let $\BCU_{\nu}(\Gaff)$ denote the quotient of $\CU_{\nu}(\Gaff)$ by the two-sided ideal generated by the $x_{i,m}^{\pm}$ for $i \in I$ and $m \in \BZ$. Let $\phi_i(z)$ be the image of $\phi_i^+(z)$ in the quotient; it is a polynomial in $z$. Let $\pi_{\nu}: \CU_{\nu}(\Gaff) \longrightarrow \BCU_{\nu}(\Gaff)$ be the quotient map. Call a $\CU_{\nu}(\Gaff)$-module trivial if it factorizes through $\pi_{\nu}$.

The following result shows that the tensor product of a $\CU_{\mu}(\Gaff)$-module with a $\CU_{\nu}(\Gaff)$-module is naturally a $\CU_{\mu+\nu}(\Gaff)$-module if one of the tensor factors is a trivial module.  This differs from the fusion product \cite[\S 5.3]{H}.

\begin{lem}  \label{lem: tensor trivial}
Let $\mu$ and $\nu$ be coweights. Then we have two algebra homomorphisms $F_{\mu,\nu}: \CU_{\mu+\nu}(\Gaff) \longrightarrow \BCU_{\mu}(\Gaff) \otimes \CU_{\nu}(\Gaff)$ and $G_{\mu,\nu}: \CU_{\mu+\nu}(\Gaff) \longrightarrow \CU_{\mu}(\Gaff) \otimes \BCU_{\nu}(\Gaff)$ defined by
\begin{align*}
F_{\mu,\nu}: &\ x_i^+(z) \mapsto \phi_i(z) \otimes x_i^+(z), \quad x_i^-(z) \mapsto 1 \otimes x_i^-(z),\quad \phi_i^{\pm}(z) \mapsto \phi_i(z) \otimes \phi_i^{\pm}(z), \\
G_{\mu,\nu}:&\ x_i^+(z) \mapsto x_i^+(z) \otimes 1, \quad x_i^-(z) \mapsto  x_i^-(z) \otimes \phi_i(z),\quad \phi_i^{\pm}(z) \mapsto \phi_i^{\pm}(z) \otimes \phi_i(z).
\end{align*}
\end{lem}
\begin{proof}
Recall the Drinfeld formal coproduct $\widetilde{\Delta}_{\mu,\nu}: \CU_{\mu+\nu}(\Gaff) \longrightarrow \CU_{\mu}(\Gaff)\ \widehat{\otimes}\ \CU_{\nu}(\Gaff)$ from \cite[Lemma 10.3]{FT}. Passing to the quotients $\CU_{\mu}(\Gaff) \longrightarrow \BCU_{\mu}(\Gaff)$ and $\CU_{\nu}(\Gaff) \longrightarrow \BCU_{\nu}(\Gaff)$ give the desired algebra homomorphisms.
\end{proof}

\begin{example} \label{example: tensor prefund}
Let $\nu, \zeta$ be coweights, $\Bp \in \mathfrak{d}_{\zeta}$ and $N$ be a $\CU_{\nu}(\Gaff)$-module. Consider the $\CU_{\nu+\zeta}(\Gaff)$-modules $L_{\zeta}(\Bp)_{w^{-1}} \otimes N$ and $N \otimes L_{\zeta}(\Bp)_{w^{-1}}$. 
As in Example \ref{example: tensor prefund Yangian}, these are pullbacks of the polynomial space $N[w,w^{-1}]$, viewed as a $\CU_{\nu}(\Gaff)[w,w^{-1}]$-module by scalar extension, along the algebra homomorphisms $F_{\Bp,w}^{\nu+\zeta}$ and $G_{\Bp,w}^{\nu+\zeta}$ respectively from $\CU_{\nu+\zeta}(\Gaff)$ to $\CU_{\nu}(\Gaff)[w,w^{-1}]$ defined by: 
\begin{align*}
F_{\Bp,w}^{\nu+\zeta}:&\ x_i^+(z) \mapsto \Bp_i(zw) x_i^+(z), \quad x_i^-(z) \mapsto x_i^-(z),\quad \phi_i^{\pm}(z) \mapsto \Bp_i(zw) \phi_i^{\pm}(z), \\
G_{\Bp,w}^{\nu+\zeta}:& \ x_i^+(z) \mapsto  x_i^+(z), \quad x_i^-(z) \mapsto \Bp_i(zw)x_i^-(z),\quad \phi_i^{\pm}(z) \mapsto \Bp_i(zw) \phi_i^{\pm}(z).
\end{align*}
 If we take the underlying space to be the space of Laurent series $N((w))$ as in Proposition \ref{prop: R-matrix general}, we get completed tensor product modules over $\CU_{\nu+\zeta}(\Gaff)$: 
$$\overline{L_{\zeta}(\Bp)_{w^{-1}} \otimes N}, \qquad\overline{N \otimes L_{\zeta}(\Bp)_{w^{-1}}}.  $$
\end{example}

Eqs.\eqref{def: Ti}--\eqref{def: T-series} make sense in $\CU_{\nu}(\Gaff)$ and define the power series $T_{\Bp}^{\nu}(z) \in \CU_{\nu}(\Gaff)[[z]]$ for $\Bp \in \mathfrak{d}$; the case $\Bp = \Psi_{i,1}$ corresponds to $T_i^-(z)$ in \cite[\S 9.2]{H}. The commutation relations of Eqs.\eqref{Drinfeld rel: h x} and \eqref{rel: T x} still hold true in shifted quantum affine algebras, from which we obtain the following  quantum affine analog of Theorem \ref{thm: Yangian T-series}.
\begin{prop}  \label{prop: one-dim R}
Let $\nu, \zeta$ be coweights, $\Bp \in \mathfrak{d}_{\zeta}$ and $N$ be a module over $\CU_{\nu}(\Gaff)$. The action of $T_{\Bp}^{\nu}(w) \in \CU_{\nu}(\Gaff)[[w]]$ on $N((w))$ defines an isomorphism of $\CU_{\nu+\zeta}(\Gaff)$-modules 
$$ T_{\Bp}^{\nu}(w)|_N:  \overline{L_{\zeta}(\Bp)_{w^{-1}} \otimes N} \longrightarrow \overline{N \otimes L_{\zeta}(\Bp)_{w^{-1}}}. $$
\end{prop}

In Example \ref{example: tensor prefund} assume $\nu+\zeta$ to be antidominant so that both tensor products are modules over the upper Borel subalgebra. To the end of this section, we will evaluate the first tensor factor of the universal R-matrix to these tensor product modules. For that purpose, we need a shifted version of the universal R-matrix defined as follows.

Let $\nu$ be a coweight. Recall from Definition \ref{def: shifted root vectors} the shifted root vectors $E_{\alpha}^{\nu}$ and the modified Drinfeld--Cartan elements $h_{i,s}$ in the shifted quantum affine algebra $\CU_{\nu}(\Gaff)$. In Eqs.\eqref{def: R0}--\eqref{def: R bar} for $\CR_0(z), \CR_{\pm}(z)$ and $\barR(z)$, we substitute $h_{i,s}, E_{\alpha} \in \Borel$ in their first tensor factor with $h_{i,s}, E_{\alpha}^{\nu} \in \CU_{\nu}(\Gaff)$. The resulting series are denoted by 
$$ \CR_0^{\nu}(z),\ \CR_{\pm}^{\nu}(z),\ \barR^{\nu}(z) \in  (\CU_{\nu}(\Gaff) \stimes \lBorel)[[z]]. $$
Here the completion is as in the Borel case based on the weight grading on $\CU_{\nu}(\Gaff)$.

\begin{lem} \label{lem: deformed shifted homo}
For $\nu, \zeta$ coweights and $\Bp \in \mathfrak{d}_{\zeta}$, we have in $(\CU_{\nu}(\Gaff) \stimes \lBorel)[[z,w]]$:
\begin{align}
&\begin{cases}
(F_{\Bp,w}^{\nu+\zeta} \otimes \mathrm{Id})(\CR_+^{\nu+\zeta}(z)) = (T_{\Bp}^{\nu}(w)^{-1} \otimes 1) \times \CR_+^{\nu}(z) \times (T_{\Bp}^{\nu}(w) \otimes 1), \\
(F_{\Bp,w}^{\nu+\zeta} \otimes \mathrm{Id})(\CR_0^{\nu+\zeta}(z)) = (1 \otimes T_{\Bp}(zw)) \times \CR_0^{\nu}(z), \\
(F_{\Bp,w}^{\nu+\zeta} \otimes \mathrm{Id})(\CR_-^{\nu+\zeta}(z)) = \CR_-^{\nu}(z), 
\end{cases}   \label{F R} \\
&\begin{cases}
(G_{\Bp,w}^{\nu+\zeta} \otimes \mathrm{Id})(\CR_+^{\nu+\zeta}(z)) = \CR_+^{\nu}(z), \\
 (G_{\Bp,w}^{\nu+\zeta} \otimes \mathrm{Id})(\CR_0^{\nu+\zeta}(z)) =  \CR_0^{\nu}(z) \times (1\otimes T_{\Bp}(zw)), \\
 (G_{\Bp,w}^{\nu+\zeta} \otimes \mathrm{Id})(\CR_-^{\nu+\zeta}(z)) =  (T_{\Bp}^{\nu}(w) \otimes 1) \times \CR_-^{\nu}(z) \times (T_{\Bp}^{\nu}(w)^{-1} \otimes 1).
\end{cases}  \label{G R}
\end{align}
\end{lem}  
\begin{proof}
We shall prove Eq.\eqref{F R} as the case of \eqref{G R} is parallel.
Write $\mu := \nu + \zeta$. By Eq.\eqref{rel: T x} and Remark \ref{rem: from R to T} we have for $(i,m,s) \in I \times \BZ \times \BZ_{>0}$:
\begin{gather*}
F_{\Bp,w}^{\mu}(x_{i,m}^+) = T_{\Bp}^{\nu}(w)^{-1} x_{i,m}^+ T_{\Bp}^{\nu}(w), \quad F_{\Bp,w}^{\mu}(h_{i,s}) =  h_{i,s} + \lambda_{i,s}^{\Bp} w^s, \\
F_{\Bp,w}^{\mu}(\phi_{i,0}^+) = \phi_{i,0}^+,\quad F_{\Bp,w}^{\mu}(x_{i,m}^-) = x_{i,m}^-
\end{gather*}
By Definition \ref{def: shifted root vectors} if $\alpha \in \hat{\Phi}_+$ then we have $F_{\Bp,w}^{\mu}(E_{\alpha}^{\mu}) = T_{\Bp}^{\nu}(w)^{-1} E_{\alpha}^{\nu} T_{\Bp}^{\nu}(w)$; if $\alpha \in \hat{\Phi}_-$ then $F_{\Bp,w}^{\mu}(E_{\alpha}^{\mu}) = E_{\alpha}^{\nu}$. Eq.\eqref{F R} follows from Eqs.\eqref{def: R0}--\eqref{def: R+-}. 
\end{proof}

By Eq.\eqref{def: Theta series} and Theorem \ref{thm: polynomiality Theta}, for $\Bp \in \mathfrak{d}$ we have a power series $\Theta_{\Bp}(z)$ in $z$ with coefficients in $U_q^-(\mathfrak{c}) \otimes U_q^+(L\Glie)$, which corresponds to a power series with coefficients in $U_q^-(\mathfrak{c}) \otimes \CU_{\nu}^+(\Gaff)$, denoted by $\Theta_{\Bp}^{0,\nu}(z)$. It is the formal sum of the polynomials 
$$\Theta_{\Bp,\beta}^{0,\nu}(z) \in U_q^-(\mathfrak{c})_{-\beta} \otimes \CU_{\nu}^+(\Gaff)_{\beta}[z] \quad \mathrm{for}\ \beta \in \BQ_+. $$ 
Similarly, the power series $\Theta_{\Bp}^{\nu,0}(z) \in (\CU_{\nu}^-(\Gaff) \otimes U_q^+(\mathfrak{c}))[[z]]$ and its polynomial components are defined. Below we shall need the modified series:
$$ \Theta_{\Bp}^{\infty}(z) := \sum_{\beta \in \BQ_+} q^{(\beta,\beta)} \Theta_{\Bp,\beta}(z) (k_{-\beta} \otimes k_{\beta}), \quad \Theta_{\Bp}^{\nu,\infty}(z)  := \sum_{\beta \in \BQ_+} q^{(\beta,\beta)} \Theta_{\Bp,\beta}^{\nu,0}(z)(\phi_{-\beta}^+\otimes k_{\beta}). $$
The first power series has coefficients in $\lBorel \otimes \lBorel$, and the second $\CU_{\nu}(\Gaff) \otimes \lBorel$. In terms of the $q$-completion of Eq.\eqref{def: t infinity}, we have 
\begin{equation}  \label{q-completion Theta}
\Theta_{\Bp}^{\infty}(z) = q^{-t_{\infty}} \Theta_{\Bp}(z) q^{t_{\infty}}.
\end{equation}

For $X = \sum_i a_i \otimes b_i$ let us denote $X_{21} := \sigma(X) = \sum_i b_i \otimes a_i$.
\begin{lem}  
For $\nu$ a coweight, we have in $(\CU_{\nu}(\Gaff)\stimes \lBorel)[[z,w]]$:
\begin{align}    
\begin{split}  \label{Theta R+}
&\CR_+^{\nu}(z) \times (T_{\Bp}^{\nu}(w) \otimes T_{\Bp}(zw))\times \CR_+^{\nu}(z)^{-1}   \\
            = &\ (T_{\Bp}^{\nu}(w) \otimes 1) \times \left[(\tau_z \otimes \mathrm{Id}) \left(\Theta_{\Bp}^{0,\nu}(w)\times (T_{\Bp}(w) \otimes 1)\right)\right]_{21},   
\end{split} \\
\begin{split} \label{Theta R-}
& \CR_-^{\nu}(z)^{-1} \times (T_{\Bp}^{\nu}(w) \otimes T_{\Bp}(zw)) \times \CR_-^{\nu}(z)  \\
= &\   (\mathrm{Id} \otimes \tau_z)\left(  (1\otimes T_{\Bp}(w)) \times  \Theta_{\Bp}^{\nu,\infty}(w)\right) \times  (T_{\Bp}^{\nu}(w) \otimes 1).
\end{split}
\end{align}
\end{lem}
\begin{proof}
Each of the factors in Eq.\eqref{Theta R+} is  actually a power series in $z, w$; indeed the first tensor factor of $\Theta_{\Bp}^{0,\nu}(w)$ lies in $\lBorel$, which is generated by the elements $f_j$ for $0\leq j \leq r$ of negative degrees. 
By Proposition \ref{prop: twist} and Eq.\eqref{def: Theta series}:
\begin{align*}
 \CR_+(z) \times (T_{\Bp}(w) \otimes T_{\Bp}(zw)) \times \CR_+(z)^{-1} = \sigma \circ (\tau_z \otimes \mathrm{Id}) \circ\Delta(T_{\Bp}(zw)) \\
= (T_{\Bp}(w) \otimes 1) \times \left[ (\tau_z \otimes \mathrm{Id}) \left(\Theta_{\Bp}(w)   \times (T_{\Bp}(w) \otimes 1)\right)  \right]_{21}.
\end{align*}
Here we used  $\tau_z(T_{\Bp}(w)) = T_{\Bp}(zw)$. 
This proves the Borel analog of Eq.\eqref{Theta R+} in $(\mathcal{X} \stimes \lBorel)[[z,w]]$, where $\mathcal{X}$ is the subalgebra of $\qaf$ generated by the $x_{i,m}^+$ and $h_{i,s}$ for $(i, m, s) \in I \times \BZ\times \BZ_{\neq  0}$. Here we used the fact that the first tensor factor of $\CR_+(z)$ lies in $U_q^+(\mathfrak{b}) \subset \mathcal{X}$, the second tensor factor of $\Theta_{\Bp}(w)$ lies in $U_q^+(\mathfrak{c}) \subset \mathcal{X}$, and the coefficients of $T_{\Bp}(w)$ lie in the subalgebra generated by the $h_{i,s} \in \mathcal{X}$. 

Define the subalgebra $\mathcal{X}_{\nu}$ of $\CU_{\nu}(\Gaff)$ in the same way. Then the natural identification of \eqref{identification subalgebras} extends to $\mathcal{X} \cong \mathcal{X}_{\nu}$; indeed as a consequence of triangular decomposition both $\mathcal{X}$ and $\mathcal{X}_{\nu}$ are generated by $x_{i,m}^+, h_{i,s}$ subject to definition relations \eqref{Drinfeld rel: Drinfeld}, \eqref{Drinfeld rel: Serre} and \eqref{Drinfeld rel: h x}. Under this identification, the equation in $(\mathcal{X} \stimes \lBorel)[[z,w]]$ corresponds to an equation in $(\mathcal{X}_{\nu} \stimes \lBorel)[[z,w]]$, which is precisely Eq.\eqref{Theta R+}.

To prove Eq.\eqref{Theta R-}, by Proposition \ref{prop: twist}, Eqs.\eqref{def: t infinity} and\eqref{q-completion Theta}:
\begin{align*}
&\CR_-(z)^{-1} \times (T_{\Bp}(w) \otimes T_{\Bp}(zw)) \times \CR_-(z) =  q^{-t_{\infty}} \times (\mathrm{Id} \otimes \tau_z)\Delta(T_{\Bp}(w)) \times q^{t_{\infty}}\\
&= (\mathrm{Id} \otimes \tau_z)\left( (1\otimes T_{\Bp}(w)) \times \Theta_{\Bp}^{\infty}(w)\right) \times (T_{\Bp}(w) \otimes 1).
\end{align*}
As in the case of $\CR_+$, the first tensor factors of the above equations can be shifted, resulting in the desired Eq.\eqref{Theta R-}.
\end{proof}

\begin{cor} 
For $\nu, \zeta$ coweights and $\Bp \in \mathfrak{d}_{\zeta}$, we have in $(\CU_{\nu}(\Gaff) \stimes \lBorel)[[z,w]]$:
\begin{gather}
(F_{\Bp,w}^{\nu+\zeta} \otimes \mathrm{Id})(\barR^{\nu+\zeta}(z)) =\left[(\tau_z \otimes \mathrm{Id})\left(\Theta_{\Bp}^{0,\nu}(w) \times (T_{\Bp}(w) \otimes 1) \right) \right]_{21} \times \barR^{\nu}(z), \label{dec: cyclic} \\
(G_{\Bp,w}^{\nu+\zeta} \otimes \mathrm{Id})(\barR^{\nu+\zeta}(z)) = \barR^{\nu}(z) \times (\mathrm{Id} \otimes \tau_z)\left((1\otimes T_{\Bp}(w)) \times \Theta_{\Bp}^{\nu,\infty}(w)\right).    \label{dec: cocyclic}
\end{gather}
\end{cor}
\begin{proof}
Eq.\eqref{dec: cyclic} follows from a sequence of identities
\begin{align*}
& (F_{\Bp,w}^{\nu+\zeta} \otimes \mathrm{Id})(\barR^{\nu+\zeta}(z)) = (F_{\Bp,w}^{\nu+\zeta} \otimes \mathrm{Id})(\CR_+^{\nu+\zeta}(z) \times \CR_0^{\nu+\zeta}(z) \times \CR_-^{\nu+\zeta}(z))   \\
=&\ (T_{\Bp}^{\nu}(w)^{-1} \otimes 1) \times \CR_+^{\nu}(z) \times ( T_{\Bp}^{\nu}(w) \otimes 1) \times (1 \otimes T_{\Bp}(zw)) \times \CR_0^{\nu}(z)  \times  \CR_-^{\nu}(z)  \\
=&\  (T_{\Bp}^{\nu}(w)^{-1} \otimes 1) \times \CR_+^{\nu}(z)\times  ( T_{\Bp}^{\nu}(w) \otimes T_{\Bp}(zw)) \times \CR_0^{\nu}(z)  \times  \CR_-^{\nu}(z)  \\
 =&\ \left[(\tau_z \otimes \mathrm{Id})\left(\Theta_{\Bp}^{0,\nu}(w) \times (T_{\Bp}(w) \otimes 1) \right) \right]_{21} \times \CR_+^{\nu}(z) \times \CR_0^{\nu}(z)  \times  \CR_-^{\nu}(z)  \\
 =&\ \left[(\tau_z \otimes \mathrm{Id})\left(\Theta_{\Bp}^{0,\nu}(w) \times (T_{\Bp}(w) \otimes 1) \right) \right]_{21} \times \barR^{\nu}(z).
\end{align*}
The second line follows from Eq.\eqref{F R} and the fourth line from Eq.\eqref{Theta R+}. Eq.\eqref{dec: cocyclic} is proved similarly based on Eqs.\eqref{G R} and \eqref{Theta R-}.
\end{proof}
We specialize the decomposition formulas \eqref{dec: cyclic}--\eqref{dec: cocyclic} to modules over antidominantly shifted quantum affine algebras. Their module structures can be restricted to the upper Borel subalgebra via the embedding of Proposition \ref{prop: shifted upper Borel}. To simplify notations, in Example \ref{example: tensor prefund} let us write the tensor product modules as
$$ N_{\Bp,w}^1 := L_{\zeta}(\Bp)_{w^{-1}} \otimes N,\quad N_{\Bp,w}^2 := N \otimes L_{\zeta}(\Bp)_{w^{-1}}.$$ 

\begin{theorem}  \label{thm: decomposition R-matrices}
Let $\nu, \zeta$ be coweights and $\Bp \in \mathfrak{d}_{\zeta}$ with $\nu+\zeta$ antidominant. 
Let $N$ be a $\CU_{\nu}(\Gaff)$-module which is negative root graded as a $\Borel$-module. Then we have the following decomposition of R-matrices for any finite-dimensional $\qaf$-module $V$:
\begin{align}
\check{R}_{N_{\Bp,w}^1, V}(z) &= \left[\Theta_{\Bp}^{0,\nu}(w) \times (T_{\Bp}(w) \otimes 1) \right]_{V_z \otimes N}  \circ \check{R}_{N,V}(z), \label{R cyclic}  \\
\check{R}_{N_{\Bp,w}^2, V}(z) &= \check{R}_{N, V}(z) \circ \left[(1 \otimes T_{\Bp}(w)) \times \Theta_{\Bp}^{\nu,0}(w) \right]_{N \otimes V_z}. \label{R cocyclic}
\end{align}
\end{theorem}
\begin{proof}
Set $\mu := \nu + \zeta$. By Example \ref{example: tensor prefund} and Proposition \ref{prop: shifted upper Borel} we have 
 $$F_{\Bp,w}^{\mu}(\phi_{i,0}^+) = G_{\Bp,w}^{\mu}(\phi_{i,0}^+) = \phi_{i,0}^+ = \jmath_{\mu}(k_i) \quad \textrm{for $i \in I$}. $$
 This means that the weight space decompositions of the $\Borel$-modules $N_{\Bp,w}^1$ and $N_{\Bp,w}^2$ are simply the $\BC[w,w^{-1}]$-scalar extension of the weight space decomposition of $N$: 
 $$ (N_{\Bp,w}^1)_{\lambda} = (N_{\Bp,w}^2)_{\lambda} = N_{\lambda}[w,w^{-1}]\quad \mathrm{for}\ \lambda \in \mathfrak{t}^*. $$
 So both modules $N_{\Bp,w}^1$ and $N_{\Bp,w}^2$ are negative root graded, and the left-hand sides of Eq.\eqref{R cyclic}--\eqref{R cocyclic} are well-defined by Proposition \ref{prop: R-matrix general}. Moreover, the actions of $q^{-t_{\infty}}$ on $N_{\Bp,w}^1 \otimes V$ and $N_{\Bp,w}^2\otimes  V$ are induced from the action of $q^{-t_{\infty}}$ on $N\otimes V$.
 
Since $V$ is finite-dimensional and all tensor factors of Theta series are of negative degrees, the series $(\tau_z \otimes \mathrm{Id})(\Theta_{\Bp}^{0,\nu}(w))$ maps $V \otimes N$ to $V\otimes N[w,z]$ and extends to a $\BC[w,w^{-1}]((z))$-linear operator on $(V\otimes N)[w,w^{-1}]((z))$, the underlying space of the completed tensor product $\overline{V_z \otimes N_{\Bp,w}^1}$ in Proposition \ref{prop: R-matrix general}. In the same way, all the other factors at the right-hand sides of Eqs.\eqref{R cyclic}--\eqref{R cocyclic} are well-defined.

 Since $\mu$ is antidominant, $\jmath_{\mu} \otimes \mathrm{Id}$ sends $\barR(z)$ to $\barR^{\mu}(z)$ by Proposition \ref{prop: shifted upper Borel}. Therefore the action of $\barR(z)$ on $N_{\Bp,w}^1 \otimes V$ is given by the action of $\barR^{\mu}(z)$ on $N_{\Bp,w}^1 \otimes V$. By definition this is the action of $(F_{\Bp,w}^{\mu} \otimes \mathrm{Id})(\barR^{\mu}(z))$ on $N\otimes V$ extended by $\BC[w,w^{-1}]$-linearity to $N[w,w^{-1}] \otimes V$. We are ready to apply Eq.\eqref{dec: cyclic}: 
\begin{align*}
 \check{R}_{N_{\Bp,w}^1,V}(z) &= \sigma \circ (F_{\Bp,w}^{\mu} \otimes \mathrm{Id})(\barR^{\mu}(z))|_{N\otimes V}  \circ q^{-t_{\infty}}|_{N \otimes V} \\
&= \left[ (\tau_z \otimes \mathrm{Id})\left(\Theta_{\Bp}^{0,\nu}(w) \times ( T_{\Bp}(w) \otimes 1)  \right) \right]_{V\otimes N} \circ \sigma \circ \barR^{\nu}(z) q^{-t_{\infty}}|_{N \otimes V}   \\
 &=  \left[\Theta_{\Bp}^{0,\nu}(w) \times (T_{\Bp}(w) \otimes 1) \right]_{V_z \otimes N}  \circ \check{R}_{N,V}(z).
 \end{align*}
 At the third line we used the deformed module structure to drop $\tau_z$ and identified $\barR(z)|_{N\otimes V}$ with $\barR^{\nu}(z)|_{N\otimes V}$ because is $\nu$ antidominant. This proves Eq.\eqref{R cyclic}.
 
 The action of $\Theta_{\Bp}^{\nu,\infty}(w)$ on the $\CU_{\nu}(\Gaff) \otimes \qaf$-module $N[w,w^{-1}] \otimes V_z$ equals the action of $q^{-t_{\infty}}\circ \Theta_{\Bp}^{\nu,0}(w) \circ q^{t_{\infty}}$ on $N[w,w^{-1}]\otimes V_z$ where $q^{\pm t_{\infty}}$ is computed as in Definition \ref{defi: root graded} with respect to the $\Borel \otimes \qaf$-module structure on $N\otimes V$. Together with the commutativity of $q^{-t_{\infty}}$ with $1\otimes T_{\Bp}(w)$ acting on $N\otimes V_z$, we have
\begin{align*}
 \check{R}_{N_{\Bp,w}^2,V}(z) &= \sigma \circ (G_{\Bp,w}^{\mu} \otimes \mathrm{Id})(\barR^{\mu}(z))|_{N \otimes V} \circ q^{-t_{\infty}}|_{N \otimes V}  \\
 &= \sigma \circ \barR^{\nu}(z)|_{N\otimes V} \circ \left[(\mathrm{Id} \otimes \tau_z) \left( (1\otimes T_{\Bp}(w)) \times \Theta_{\Bp}^{\nu,\infty}(w) \right)\right]_{N\otimes V} \circ q^{-t_{\infty}}|_{N\otimes V} \\
 &=  \sigma \circ \barR^{\nu}(z)q^{-t_{\infty}}|_{N\otimes V}  \circ \left[ (1 \otimes T_{\Bp}(w) \times \Theta_{\Bp}^{\nu,0}(w)\right]_{N\otimes V_z}  \\
 &= \check{R}_{N,V}(z) \circ \left[(1 \otimes T_{\Bp}(w)) \times \Theta_{\Bp}^{\nu,0}(w) \right]_{N\otimes V_z}.
\end{align*} 
At the second line we applied Eq.\eqref{dec: cocyclic}. This proves Eq.\eqref{R cocyclic}.
\end{proof}
\begin{rem}  \label{rem: R tensor trivial}
Let $V$ be a tensor product of finite-dimensional highest $\ell$-weight modules over $\qaf$. 
Let $\alpha_{N,V}(z) \in \BC[[z]]$ be such that 
$\alpha_{N,V}(z) \check{R}_{N,V}(z)$ restricts to a polynomial R-matrix from $N \otimes V_z$ to $V_z \otimes N$; see Theorem \ref{thm: poly R Borel} for general examples.  By Proposition \ref{prop: poly T} and Theorem \ref{thm: decomposition R-matrices} we have two more polynomial R-matrices
\begin{align*}
 \frac{\alpha_{N,V}(z)}{f_{\Bp}^V(zw)}  \check{R}_{N_{\Bp,w}^s, V}(z)&: N_{\Bp,w}^s \otimes V_z \longrightarrow V_z \otimes N_{\Bp,w}^s \quad \mathrm{for}\ s \in \{1, 2\}.
\end{align*}
Specializing $w$ to an arbitrary nonzero complex number, we obtain from $N_{\Bp,w}^2$ asymptotic representations of \cite{Z1} (see Eq.\eqref{asym quantum} for the precise statement) and therefore polynomial R-matrices between an asymptotic representation and $V$. 
\end{rem}
\begin{example}
Fix $\Glie = sl_2$.  Let $N$ be the irreducible $\CU_{-\varpi_1^{\vee}}(\Gaff)$-module of highest $\ell$-weight $\Psi_{1,q^{-4}}^{-1}$, constructed in Example \ref{example: prefund sl2} with basis $(v_n^*: n \in \BN)$. Let $V = \BC e_1 \oplus \BC e_2$ be the irreducible $\qaf$-module defined by ($m \in \BZ$ and $s \in \BZ_{\neq 0}$):
\begin{gather*}
    k_1 = \begin{pmatrix}
    q & 0 \\
    0 & q^{-1}
    \end{pmatrix},\  x_{1,m}^+ = \begin{pmatrix}
    0 & 1 \\
    0 & 0
    \end{pmatrix},\  x_{1,m}^- = \begin{pmatrix}
    0 & 0 \\
    1 & 0
    \end{pmatrix},\  h_{1,s} = \begin{pmatrix}
    \frac{1-q^{-2s}}{s(q-q^{-1})} & 0 \\
    0 & \frac{1-q^{2s}}{s(q-q^{-1})}
    \end{pmatrix}. 
\end{gather*}
 Our goal is to compute the R-matrix of Eq.\eqref{R cyclic} for $\Bp = \Psi_{1,1}$. Let $d$ be the linear operator on $N$ sending $v_j^*$ to $q^j v_j^*$. Then for $v \in N$ we have
 $$q^{-t_{\infty}}(v \otimes e_1)  = d(v) \otimes e_1, \quad  q^{-t_{\infty}}(v \otimes e_2)  = d^{-1}(v) \otimes e_2. $$
  Consider the T-series  $T_{\Bp}(z) = \exp(\sum_{s>0} \frac{q-q^{-1}}{q^{2s}-q^{-2s}} h_{1,-s} z^s)$  acting on $V$:
$$ f_{\Bp}^V(z) = \exp(\sum_{s>0} \frac{1}{s(1+q^{-2s})}  z^s),\quad g_{\Bp}^V(z) = \exp(\sum_{s>0} \frac{-1}{s(q^{2s}+1)} z^s) = f_{\Bp}^V(z) (1-z).   $$
Next we evaluate the second tensor factor of $\CR(z)$ at $V$. With respect to the basis $(e_1, e_2)$ of $V$, it is a square matrix whose entries lie in $\Borel[[z]]$. Introduce the half currents $x_{1,>}^{\pm}(z) := \sum_{m>0}x_{1,m}^{\pm} z^m$ and $x_{1,\geq}^{\pm}(z) := \sum_{m\geq 0}x_{1,m}^{\pm} z^m$. We have
\begin{align*}
\CR_-(z)|_{-\otimes V} &=  \prod_{m>0}^{\prec} \exp_q ((q^{-1}-q)   k_1^{-1} x_{1,m}^-z^m \otimes x_{1,-m}^+ k_1) \\ 
    &= \prod_{m>0}^{\prec} \exp_q \begin{pmatrix} 
    0 & (q^{-2}-1)   k_1^{-1} x_{1,m}^-z^m  \\
    0 & 0
    \end{pmatrix}  = \begin{pmatrix}
    1 & (q^{-2}-1) k_1^{-1} x_{1,>}^-(z) \\
    0 & 1
    \end{pmatrix}, \\
\CR_0(z)|_{-\otimes V} &=  \exp ( - \sum_{m>0}  \frac{m(q-q^{-1})^2}{q^{2m}-q^{-2m}} h_{1,m} z^m \otimes h_{1,-m} )   \\              
              &= \begin{pmatrix} \exp( \sum\limits_{m>0} \frac{q^{-1}-q}{1+q^{-2m}} h_{1,m} z^m) & 0 \\
              0 & \exp(\sum\limits_{m>0} \frac{q-q^{-1}}{1+q^{2m}} h_{1,m} z^m) 
               \end{pmatrix}, \\
\CR_+(z)|_{-\otimes V} &= \prod_{m\geq 0}^{\prec} \exp_q((q^{-1}-q)  x_{1,m}^+z^m \otimes x_{1,-m}^-) \\
               &=  \prod_{m\geq 0}^{\prec} \exp_q\begin{pmatrix}
    0 & 0 \\
    (q^{-1}-q)  x_m^+z^m & 0
    \end{pmatrix} =  \begin{pmatrix}
    1 & 0 \\
    (q^{-1}-q)x_{\geq}^+(z) & 1
    \end{pmatrix},    \\
\CR(z)|_{-\otimes V} &= \CR_+(z)\CR_0(z)\CR_-(z) q^{-t_{\infty}}|_{-\otimes V} = \begin{pmatrix}
    A(z) d & B(z) d^{-1} \\
    C(z) d & D(z) d^{-1}
    \end{pmatrix} \quad \mathrm{where} \\
    A(z) &= \exp( \sum_{m>0} \frac{q^{-1}-q}{1+q^{-2m}} h_{1,m} z^m), \\
    B(z) &= (q^{-2}-1) k_1^{-1} A(z)  x_{1,>}^-(z), \quad    C(z) = (q^{-1}-q)x_{1,\geq}^+(z) A(z), \\
    D(z) &= (q^{-1}-q) x_{1,\geq}^+(z) B(z) + \exp(\sum_{m>0} \frac{q-q^{-1}}{1+q^{2m}} h_{1,m} z^m).    
\end{align*}
Evaluate the ABCD series at $v_j^*$ for $j \in \BN$. We obtain from Example \ref{example: prefund sl2}:
\begin{align*}
g_{\Bp}^V(z) A(z) d v_j^* &= (1-zq^{-2-2j}) q^j v_j^* = (q^j - zq^{-j-2}) v_j^*, \\
g_{\Bp}^V(z) B(z) d^{-1}v_j^* &= g_{\Bp}^V(z) A(z) (q^{-2}-1) k_1^{-1}  x_{1,>}^-(z)q^{-j} v_j^*  = \frac{z(q^j-q^{-j-2})}{q-q^{-1}}  v_{j+1}^*, \\
g_{\Bp}^V(z) C(z) d v_j^* &= (q^{-1}-q)x_{1,\geq}^+(z) g_{\Bp}^V(z) A(z) q^j v_j^* = (q-q^{-1}) q^{-j} v_{j-1}^*, \\
g_{\Bp}^V(z) D(z) d^{-1} v_j^* &=  (q^{-1}-q) x_{1,\geq}^+(z) g_{\Bp}^V(z) B(z) d^{-1} v_j^*  \\
& \qquad+ g_{\Bp}^V(z) \exp(\sum_{m>0} \frac{q-q^{-1}}{1+q^{2m}} h_{1,m} z^m)q^{-j} v_j^*  =  q^{-j} v_j^*.               \end{align*}
Therefore, the R-matrix $\check{R}_{N,V}(z)$ multiplied by $g_{\Bp}^V(z)$ is polynomial in $z$:
$$ g_{\Bp}^V(z) \check{R}_{N,V}(z) =  \begin{pmatrix}
\sum\limits_{j\geq 0} (q^j-zq^{-j-2}) E_{jj} & \frac{z}{q-q^{-1}} \sum\limits_{j\geq 0} (q^j-q^{-j-2}) E_{j+1,j} \\
(q-q^{-1}) \sum\limits_{j>0}   q^{-j} E_{j-1,j} & \sum\limits_{j\geq 0} q^{-j} E_{jj}
\end{pmatrix}. $$
At last, consider $\Theta_{\Bp}^{0,-\varpi_1^{\vee}}(w)= \exp_q((q-q^{-1})x_{1,0}^- \otimes x_{1,-1}^+w ) $ acting on $ V_z \otimes N$. Since $x_{1,0}^- e_1 =  e_2$ in $V_z$, we have a uni-triangular matrix
 $$ \Theta_{\Bp}^{0,-\varpi_1^{\vee}}(w)|_{V_z \otimes N} = \begin{pmatrix}
 1 &   0 \\
 (q-q^{-1}) x_{1,-1}^+ w &  1
\end{pmatrix} = \begin{pmatrix}
1 & 0\\
w(q-q^{-1}) \sum\limits_{j>0} (-q^2) E_{j-1,j} & 1
\end{pmatrix}.  $$
The R-matrix $\check{R}_{N_{\Bp,w}^1, V}(z)$ of Theorem \ref{thm: decomposition R-matrices} multiplied by $\frac{g_{\Bp}^V(z)}{f_{\Bp}^V(zw)}$ is polynomial: 
\begin{gather*}
\frac{g_{\Bp}^V(z)}{f_{\Bp}^V(zw)} \check{R}_{N_{\Bp,w}^1, V}(z) = \Theta_{\Bp}^{0,-\varpi_1^{\vee}}(w)|_{V_z \otimes N} \times \frac{T_{\Bp}(zw)}{f_{\Bp}^V(zw)} \times   g_{\Bp}^V(z) \check{R}_{N,V}(z) \\
= \begin{pmatrix}
\sum\limits_{j\geq 0} (q^j-zq^{-j-2})  E_{jj} & \frac{z}{q-q^{-1}} \sum\limits_{j\geq 0} (q^j-q^{-j-2}) E_{j+1,j} \\
(q-q^{-1}) \sum\limits_{j>0} (q^{-j} - w q^{j+2}) E_{j-1,j} & \sum\limits_{j\geq 0} (q^{-j} - zw q^{j+2}) E_{jj}
\end{pmatrix}.
\end{gather*}   
\end{example}
\section{Polynomiality of R-matrices}  \label{sec: R poly}
In this section we establish polynomiality, upon multiplication by a constant series, for the R-matrix $\check{R}_{V,W}(z)$ of Proposition \ref{prop: R-matrix general} and its inverse, where $V$ is a tensor product of irreducible $\Borel$-modules of rational highest $\ell$-weights and $W$ is a tensor product of finite-dimensional highest $\ell$-weight $\qaf$-modules. 

 By Weyl group symmetry \cite{Chari} a finite-dimensional $\qaf$-module is of highest $\ell$-weight if and only if it is of lowest $\ell$-weight, so the module $W$ is at the same time top graded and bottom graded. The constant series will be described in terms of the  invertible power series $f_{\Bp}^W(z)$ and $g_{\Bp}^W(z)$ for $\Bp \in \mathfrak{d}$ in Definition \ref{defi: f and g}.

\begin{prop}  \label{prop: poly R negative pref}
 Let $\Bp \in \mathfrak{d}$ and $N$ be an irreducible $\Borel$-module of highest $\ell$-weight $\tau_{q^{4\kappa}}(\Bp)^{-1}$. If $W$ is a tensor product of finite-dimensional lowest $\ell$-weight modules over $\qaf$, then the operator $g_{\Bp}^W(z) \check{R}_{N,W}(z)$ sends $N \otimes W$ to $W \otimes N[z]$.
\end{prop}
\begin{proof}
Let $\lambda_0$ be the bottom weight of $W$ and $\zeta$ be the coweight of $\Bp$.

Let $M$ be the irreducible  $\Borel$-module of lowest $\ell$-weight $\Bp$. By Remark \ref{rem: dual}, its graded dual $M^{\vee}$ is precisely the module $N$.  Fix an $\BN$-graded weight basis $\mathcal{B}$ of $M$ as in Subsection \ref{ss: poly mono}. Let $\mathcal{B}^{\vee} = \{b^{\vee}: b \in \mathcal{B}\}$ be the dual basis of the graded dual $N$ defined above Eq.\eqref{R: dual monodromy}.  The monodromy-matrix construction of Definition \ref{defi: monodromy matrix} applied to the graded dual, we obtain from Eq.\eqref{R: dual monodromy} the power series $t^{\vee}_{b_1,b_2}(z) \in \lBorel[[z]]$ for $b_1, b_2 \in \mathcal{B}$. By Proposition \ref{prop: R-matrix general} we have for $b_2 \in \mathcal{B}$ and $\omega \in W$:
 $$ \check{R}_{N,W}(z)(b_2^{\vee} \otimes \omega) = \sum_{b_1 \in \mathcal{B}}  t^{\vee}_{b_1,b_2}(z) \omega \otimes b_1^{\vee} \in (W\otimes N)[[z]]. $$
Since $W$ is finite-dimensional and $t^{\vee}_{b_1,b_2}(z)$ is of weight $\wt(b_1)-\wt(b_2)$, fixing $b_2 \in \mathcal{B}$, we have $t^{\vee}_{b_1,b_2}(z) \omega = 0$ for all but finitely many $b_1$. It is therefore enough to prove polynomiality of $g_{\Bp}^W(z) t^{\vee}_{b_1,b_2}(z)\omega$ for all $b_1, b_2 \in \mathcal{B}$ and $\omega \in W$.

As in Subsection \ref{ss: poly mono}, define $t_{b_1,b_2}^{\vee\bullet}(z) \in \lBorel[[z]]$ for $b_1, b_2 \in \mathcal{B}$ and $\bullet \in \{+,0,-\}$ by:
\begin{align*}   
\CR_{\bullet}(z)^{-1} (b_1 \otimes 1) = \sum_{b_2 \in \mathcal{B}} b_2 \otimes t^{\vee\bullet}_{b_1,b_2}(z).
\end{align*}
Fix $b_1, b_2 \in \mathcal{B}$. Set $\beta := \wt(b_1) - \wt(b_2)$ and $d := p(b_1) - p(b_2)$. We claim:
\begin{itemize}
\item[(i)] If $t^{\vee+}_{b_1,b_2}(z) \neq 0$, then $\beta \in \BQ_-$ and $t^{\vee+}_{b_1,b_2}(z)$ is a $U_q^-(\mathfrak{c})_{\beta}$-valued polynomial in $z$ of degree bounded by $d - \langle \zeta, \beta\rangle$. If $\beta = 0$, then $t^{\vee+}_{b_1,b_2}(z) = \delta_{b_1,b_2}$.
\item[(ii)] If $t^{\vee0}_{b_1,b_2}(z) \neq 0$, then $\beta = 0$ and $t^{\vee0}_{b_1,b_2}(z) T_{\Bp}(z)$ is a $U_q^0(\mathfrak{c})$-valued polynomial of degree $d$. Moreover, $t^{\vee0}_{b_1,b_2}(z) = T_{\Bp}(z)^{-1}$.
\item[(iii)]  If $t^{\vee-}_{b_1,b_2}(z) \neq 0$, then $\beta \in \BQ_+$ and $t^{\vee-}_{b_1,b_2}(z)$ is an $\mathcal{F}_{\beta}$-valued polynomial of degree $d$. If $\beta = 0$, then $t^{\vee-}_{b_1,b_2}(z) = \delta_{b_1,b_2}$.
\item[(iv)]  If $\gamma \in \BQ_+$ and $\omega \in W_{\lambda_0 + \gamma}$, then $g_{\Bp}^W(z) t^{\vee}_{b_1,b_2}(z) \omega$  is a $W_{\lambda_0+\gamma+\beta}$-valued polynomial of degree bounded by $d + \langle \zeta,\gamma\rangle$. 
\end{itemize}
The first three parts are proved in the same way as Proposition \ref{prop: poly mono}(i)--(iii) based on the $\BN$-grading of the lowest $\ell$-weight module $M = L'(\Bp)$. It remains to prove (iv). 
Indeed, by Eq.\eqref{R: inverse monodromy} we have a Gauss decomposition for $t^{\vee}$ similar to Eq.\eqref{Gauss decomposition}:
\begin{equation*}
t^{\vee}_{b_1,b_2}(z) = \sum_{b_3,b_4 \in \mathcal{B}} k_{\wt(b_2)} t^{\vee-}_{b_4,b_2}(z) t^{\vee0}_{b_3,b_4}(z) t^{\vee+}_{b_1,b_3}(z). 
\end{equation*}
Apply it to $\omega$. Given $b_3, b_4 \in \mathcal{B}$, observe first from (i) that $t^{\vee+}_{b_1,b_3}(z) \omega$ is a polynomial of degree bounded by $p(b_1)-p(b_3) + \langle \zeta, \wt(b_3)-\wt(b_1)\rangle$ and with coefficients in the subspace of weight $\lambda_0+\gamma+\wt(b_1)-\wt(b_3)$. Next apply $g_{\Bp}^W(z) T_{\Bp}(z)^{-1}$ to $t^{\vee+}_{b_1,b_3}(z) \omega$. From Proposition \ref{prop: poly T}(ii) we get a polynomial of degree bounded by
$$ p(b_1)-p(b_3) + \langle \zeta, \wt(b_3)-\wt(b_1)\rangle + \langle \zeta, \gamma+\wt(b_1)-\wt(b_3) \rangle = p(b_1)-p(b_3) +  \langle \zeta, \gamma\rangle $$
and of the same weight.
Finally applying $ k_{\wt(b_2)} t^{\vee-}_{b_4,b_2}(z) \times t^{\vee0}_{b_3,b_4}(z) T_{\Bp}(z)$ and using (ii)--(iii), we get a polynomial of degree bounded by
$$p(b_1)-p(b_3) +  \langle \zeta, \gamma\rangle + p(b_4)-p(b_2) + p(b_3)-p(b_4) = p(b_1) - p(b_2) + \langle \zeta, \gamma\rangle.  $$
In summary, $g_{\Bp}^W(z) t^{\vee}_{b_1,b_2}(z) \omega$ is a sum, over $b_3, b_4 \in \mathcal{B}$, of polynomials of degree bounded by  $p(b_1) - p(b_2) + \langle \zeta, \gamma\rangle$. It suffices to show that the summation is finite.
Since $W$ is bottom graded and $t^{\vee+}_{b_1,b_3}(z)$ is of weight $\wt(b_1)-\wt(b_3) \in \BQ_-$, we have $t^{\vee+}_{b_1,b_3}(z) w = 0$ for all but finitely many $b_3$. Together with $\wt(b_3) = \wt(b_4)$, we see that only finitely many $(b_3, b_4)$ contribute.
\end{proof}

\begin{rem}  \label{rem: denominator}
(i) In the above proof let us take $\omega$ to be a lowest $\ell$-weight vector of $W$ and assume $b_1, b_2 \in \mathcal{B}$ are of the same weight. The computation of $t^{\vee}_{b_1,b_2}(z) \omega$ can be simplified: first we must have $b_1 = b_3$ since the subspace $U_q^-(\mathfrak{c})_{\beta}$ kills $\omega$ if $\beta \neq 0$; second in order that $t^{\vee-}_{b_4,b_2}(z) t^{\vee0}_{b_1,b_4}(z)\omega$ is nonzero we must have $b_2 = b_4$. Therefore 
$$ t^{\vee}_{b_1,b_2}(z) \omega = g_{\Bp}^W(z)^{-1} k_{\wt(b_2)} \times  t^{\vee0}_{b_1,b_2}(z)  T_{\Bp}(z) \omega\quad \mathrm{if}\ \wt(b_1) = \wt(b_2). $$ 
Notably $t^{\vee}_{b_0,b_0}(z) \omega = g_{\Bp}^W(z)^{-1} \omega$ where $b_0 \in \mathcal{B}$ is the unique basis vector of zero weight. We view $g_{\Bp}^W(z)$  as the {\it denominator} of $\check{R}_{N,W}(z)$ in the following sense: if $g(z) \in \BC[[z]]$ is such that $g(z) \check{R}_{N,W}(z)$ maps $N\otimes W$ to $W \otimes N[z]$, then $g(z) g_{\Bp}^W(z)^{-1}  \in \BC[z]$.

(ii) By \cite[Definition 5.1]{FH} the transfer matrix $t_N(z, u)$ associated to $N$ is a weighted sum over $b \in \mathcal{B}$, depending on a formal parameter $u$, of the power series $t^{\vee}_{b,b}(z)$. The above proof implies that $g_{\Bp}^W(z) t_N(z,u)$ acting on $W$ is polynomial in $z$, which confirms a particular case of the conjecture on transfer matrices in \cite{FH0}.
\end{rem}
Proposition \ref{prop: poly R negative pref} can be restated as follows: the singularity of $\check{R}_{N,W}(z)$ is controlled by the action of $T_{\Bp}(z)^{-1}$ on $W$. The next result is also in this kind, by replacing $N$ with an irreducible $\Borel$-module of highest $\ell$-weight $\Bp$. 
\begin{prop}  \label{prop: poly R positive pref}
 Let $\Bp \in \mathfrak{d}$ and $K$ be an irreducible $\Borel$-module of highest $\ell$-weight $\Bp$. If $W$ is a tensor product of finite-dimensional highest $\ell$-weight modules over $\qaf$, then the operator $f_{\Bp}^W(z)^{-1} \check{R}_{K,W}(z)$ sends $K \otimes W$ to $W \otimes K[z]$.
\end{prop}
\begin{proof}
Let $\lambda_1$ be the top weight of $W$ and $\zeta$ be the coweight of $\Bp$.

Choose an $\BN$-graded weight basis $\mathcal{B}$ of the highest $\ell$-weight irreducible module $K$ as in Theorem \ref{thm: polynomiality prefund}. Let $(t_{b_1, b_2}(z))_{b_1,b_2 \in \mathcal{B}}$ be the monodromy matrix associated to $(K, \mathcal{B})$. Define the auxiliary power series $t^{\bullet}_{b_1,b_2}(z)$ for $\bullet \in \{+,0,-\}$ as in Subsection \ref{ss: poly mono}. 

Fix $b_1, b_2 \in \mathcal{B}$. Set $\beta := \wt(b_2) - \wt(b_1)$ and $d := p(b_1) - p(b_2)$. We claim:
\begin{itemize}
\item[(i)] If $t_{b_1,b_2}^+(z) \neq 0$, then $\beta \in \BQ_-$ and $t_{b_1,b_2}^+(z)$ is a $U_q^-(\mathfrak{c})_{\beta}$-valued polynomial in $z$ of degree bounded by $d - \langle \zeta, \beta\rangle$. If $\beta = 0$, then $t_{b_1,b_2}^+(z) = \delta_{b_1,b_2}$.
\item[(ii)] If $t_{b_1,b_2}^0(z) \neq 0$, then $\beta = 0$ and $ t_{b_1,b_2}^0(z)T_{\Bp}(z)^{-1}$ is a $U_q^0(\mathfrak{c})$-valued polynomial of degree $d$. Moreover, $t_{b_1,b_1}^0(z) = T_{\Bp}(z)$.
\item[(iii)] If $t_{b_1,b_2}^-(z) \neq 0$, then $\beta \in \BQ_+$ and $t_{b_1,b_2}^-(z)$ is an $\mathcal{F}_{\beta}$-valued polynomial of degree $d$. If $\beta = 0$, then $t_{b_1,b_2}^-(z) = \delta_{b_1,b_2}$.
\item[(iv)] If $\gamma \in \BQ_+$ and $\omega \in W_{\lambda_1-\gamma}$, then $f_{\Bp}^W(z)^{-1} t_{b_1,b_2}(z) \omega$ is a $W_{\lambda_1-\gamma+\beta}$-valued polynomial of degree bounded by $d + \langle \zeta,\gamma-\beta\rangle$. 
\end{itemize}
Again (i)--(iii) are proved as in Proposition \ref{prop: poly mono} and (iv) as in Proposition \ref{prop: poly R negative pref} based on the Gauss decomposition \eqref{Gauss decomposition}.  
\end{proof}

Proposition \ref{prop: poly R positive pref} appeared previously in \cite[Lemma 7.2]{FJMM}. The proof therein seems to use a hidden assumption that $\check{R}_{K,W}(z)$ is rational in $z$ after renormalization. As in Remark \ref{rem: denominator}(i), one can show that $f_{\Bp}^W(z)^{-1}$ is the denominator of $\check{R}_{K,W}(z)$. Our main result of this section generalizes Propositions \ref{prop: poly R negative pref} and \ref{prop: poly R positive pref}.

 \begin{theorem}  \label{thm: poly R Borel}
 Let $s \geq 1$ and $\Bp_t, \Bn_t \in \mathfrak{d}$ be polynomial $\ell$-weights for $1\leq t \leq s$. Set $V$ to be the tensor product of $\Borel$-modules
 $$ V := L(\frac{\Bp_1}{\Bn_1}) \otimes L(\frac{\Bp_2}{\Bn_2}) \otimes \cdots \otimes L(\frac{\Bp_s}{\Bn_s}). $$ 
 Let $W$ be a tensor product of finite-dimensional highest $\ell$-weight $\qaf$-modules. Set 
 \begin{align*}
 \alpha_{V,W}(z) := \prod_{t=1}^s \frac{g_{\Bn_t}^W(zq^{4\kappa})}{f_{\Bp_t}^W(z)}, \quad
 \beta_{W,V}(z) := \prod_{t=1}^s \frac{g_{\Bp_t}^W(z)}{f_{\Bn_t}^W(z)} \in 1 + z\BC[[z]].
 \end{align*}
 Then both operators $\alpha_{V,W}(z) \check{R}_{V,W}(z)$ and $\beta_{W,V}(z) \check{R}_{V,W}(z)^{-1}$ are polynomial in $z$; they send $V \otimes W$ to $W \otimes V[z]$ and $W \otimes V$ to $V \otimes W[z]$ respectively.
 \end{theorem}
 \begin{proof}
 We prove first the case for $\check{R}_{V,W}(z)$. By the quasi-triangularity
 $$ (\Delta \otimes \mathrm{Id})(\CR(z)) = \CR_{13}(z) \CR_{23}(z) $$
 one may assume without loss of generality $s = 1$ and $V = L(\frac{\Bp}{\Bn})$ where $\Bp, \Bn \in \mathfrak{d}$ are polynomial $\ell$-weights. Take $K := L(\Bp)$ and $N := L(\frac{1}{\Bn})$.  Then we have
 $$ \check{R}_{K\otimes N, W}(z) = (\check{R}_{K,W}(z) \otimes \mathrm{Id}_N) (\mathrm{Id}_K \otimes \check{R}_{N,W}(z)).  $$
 It follows from Propositions \ref{prop: poly R negative pref} and \ref{prop: poly R positive pref} that $\alpha_{V,W}(z) \check{R}_{K\otimes N, W}(z)$ is polynomial.
Let $S$ be the sub-$\Borel$-module of $K \otimes N$ generated by  a tensor product of highest $\ell$-weight vectors. Then $S$ is a highest $\ell$-weight module and contains a maximal submodule $S'$. By multiplicativity of highest $\ell$-weights we have a module isomorphism $S/S' \cong V$.

By definition of the universal R-matrix, $\check{R}_{K\otimes N,W}(z)$ maps $S \otimes W$ to $(W \otimes S)[[z]]$ and $S' \otimes W$ to $(W\otimes S')[[z]]$. It induces a linear map $S/S' \otimes W \longrightarrow (W\otimes S/S')[[z]]$ which is the restriction of $\check{R}_{V,W}(z)$ to $V \otimes W$ upon identifying $S/S'$ with $V$. The polynomiality of $\alpha_{V,W}(z) \check{R}_{K\otimes N,W}(z)$ restricts to the polynomiality of $\alpha_{V,W}(z)\check{R}_{V,W}(z)$. 

\medskip

By Proposition \ref{prop: R-matrix general} the inverse R-matrix $\check{R}_{V,W}(z)^{-1}$ is obtained by taking the flip map $W \otimes V \longrightarrow W \otimes V$ and then evaluating $\CR(z)^{-1}$ at $V \otimes W$. It is therefore enough to study the singularity of the action of $\CR(z)^{-1}$ on $V \otimes W$. As in the case of $\check{R}_{V,W}(z)$ one may assume either $V = L(\Bp)$ or $V = L(\frac{1}{\Bp})$, with $\Bp \in \mathfrak{d}$. 

 Assume $V = L(\Bp)$. As in the situation of Proposition \ref{prop: poly R positive pref}, based on the $\BN$-grading of $L(\Bp)$ in Theorem \ref{thm: polynomiality prefund}, we show that the singularity of $\CR(z)^{-1}$ on $V \otimes W$ is governed by the action of $T_{\Bp}(z)^{-1}$ on $W$. It follows from Proposition \ref{prop: poly T}(ii) that $g_{\Bp}^W(z) \check{R}_{V,W}(z)^{-1}$ is polynomial.

Assume $V = L(\frac{1}{\Bp})$. Let $M$ be the irreducible $\Borel$-module of lowest $\ell$-weight $\Bp$ so that by Remark \ref{rem: dual} its graded dual $M^{\wedge}$ is the module $V$. Comparing Eq.\eqref{R: monodromy matrix} with Eq.\eqref{R: inverse dual} we get that the singularity of $\CR(z)^{-1}$ acting on $V\otimes W$ is governed by the action of $\CR(z)$ on $M \otimes W$, which is governed by the action of $T_{\Bp}(z)$ on $W$ as in the situation of Proposition \ref{prop: poly R negative pref}, based on the Gauss decomposition \eqref{Gauss decomposition} and Proposition \ref{prop: poly mono}. Conclude from Proposition \ref{prop: poly T}(i) that $f_{\Bp}^W(z)^{-1} \check{R}_{V,W}(z)^{-1}$ is polynomial.
 \end{proof}

\begin{rem}
(i) We expect that the two polynomiality results of Theorem \ref{thm: poly R Borel} can be related to each other by the functor $\mathcal{F}_q$ of \cite{Pinet}. The series $\alpha_{V,W}(z)$ and $\beta_{W,V}(z)$ are in general not denominators of the corresponding R-matrices.

(ii) Let $v_0 \in V$ and $\omega_0 \in W$ be tensor products of highest $\ell$-weight vectors. Then a careful analysis of the factorization \eqref{def: R bar} together with Remark \ref{rem: from R to T} and Definition \ref{defi: f and g} shows that (see also the proof of \cite[Lemma 2.6]{FM}) 
$$ \CR(z) (v_0 \otimes \omega_0) = \CR_0(z) (v_0 \otimes \omega_0) = \prod_{t=1}^s \frac{f_{\Bp_t}^W(z)}{f_{\Bn_t}^W(z)} (v_0 \otimes \omega_0). $$
In the particular case $\Bp_t = 1$ for all $1\leq t \leq s$, the above scalar coincides with $\beta_{W,V}(z)$ and
we get a module morphism $W_z \otimes V \longrightarrow V \otimes W_z$ sending $\omega_0 \otimes v_0$ to $v_0 \otimes \omega_0$. This is a quantum affine analog of Theorem \ref{thm: poly R Yangian}. 

(iii) Assume that $V$ is the restriction of a module over an antidominantly shifted quantum affine algebra $\CU_{\nu}(\Gaff)$. Let $\zeta$ be a dominant coweight and $\Bp \in \mathfrak{d}_{\zeta}$ be a polynomial $\ell$-weight such that $\nu + \zeta$ is antidominant. By Remark \ref{rem: R tensor trivial} and Theorem \ref{thm: poly R Borel}  both the R-matrices $\check{R}_{V \otimes L_{\zeta}(\Bp),W}(z)$ and $\check{R}_{V \otimes L(\Bp),W}(z)$ multiplied by $ \frac{\alpha_{V,W}(z)}{f_{\Bp}^W(z)}$ are polynomial. We expect an embedding of $\Borel$-modules $V \otimes L_{\zeta}(\Bp) \subset V \otimes L(\Bp)$; notice that $L_{\zeta}(\Bp)$ is one-dimensional while $L(\Bp)$ is infinite-dimensional if $\Bp \neq 1$.
\end{rem}

\section{Theta series as associators in type A}  \label{sec: type A}
For $\Glie$ of type A the Drinfeld--Jimbo coproduct of shifted quantum affine algebras was obtained by Finkelberg--Tsymbaliuk \cite{FT}. In this section, we  adapt the constructions of Yangian associator maps to the case of shifted quantum affine algebras. 

Let $\mu$ and $\nu$ be coweights. For $\alpha \in \hat{\Phi}$ a real affine positive root, by Proposition \ref{prop: root vectors}(iii)--(iv) the root vector $F_{\alpha}$ of the lower Borel subalgebra $\lBorel$ is a non-commutative polynomial of the $x_{i,m}^-$ and $x_{i,m}^+k_i$. Modifying Definition \ref{def: shifted root vectors}, we define the shifted root vector $F_{\alpha}^{\mu,\nu}$ in the shifted quantum affine algebra $\CU_{\mu+\nu}(\Gaff)$ by substituting 
$$   x_{i,m}^- \mapsto x_{i,m+\langle\nu,\alpha_i\rangle}^-,\quad x_{i,m}^+ k_i \mapsto x_{i,m+\langle \mu, \alpha_i\rangle}^+ (\phi_{i,\langle\mu+\nu, \alpha_i\rangle}^-)^{-1}. $$

\begin{prop}\cite[Proposition H.1(b)]{FT} \label{prop: shifted lower Borel}
For $\mu, \nu$ two antidominant coweights, we have an injective algebra homomorphism $\jmath_{\mu,\nu}: \lBorel \longrightarrow \CU_{\mu+\nu}(\Gaff)$ defined by 
\begin{gather*}
F_{\alpha} \mapsto F_{\alpha}^{\mu,\nu}, \quad h_{i,s} \mapsto h_{i,s},\quad k_i^{-1} \mapsto \phi_{i,\langle\mu+\nu, \alpha_i\rangle}^-  \quad \mathrm{for}\ i \in I,\ s \in \BZ_{<0}\ \mathrm{and}\ \alpha \in \hat{\Phi}.
\end{gather*}
\end{prop}
Recall the shifted homomorphisms from Proposition \ref{prop: shifted homomorphism} and the embedding of the upper Borel subalgebra from Proposition \ref{prop: shifted upper Borel}.
\begin{theorem}\cite[\S 10]{FT}  \label{thm: coproduct A}
There exists a unique family of algebra homomorphisms $$\Delta_{\mu,\nu}: \CU_{\mu+\nu}(\Gaff) \longrightarrow \CU_{\mu}(\Gaff) \otimes \CU_{\nu}(\Gaff)$$ 
for  $\mu, \nu \in \BP^{\vee}$ such that
\begin{gather*}
\Delta_{\mu,\nu}(\phi_{i,0}^+) = \phi_{i,0}^+ \otimes \phi_{i,0}^+,\quad \Delta_{\mu,\nu}(\phi_{i,\langle\mu+\nu,\alpha_i\rangle}^-) = \phi_{i,\langle\mu,\alpha_i\rangle}^- \otimes \phi_{i,\langle\nu,\alpha_i\rangle}^-,\\
\Delta_{\mu,\nu}(x_{i,m}^+) = x_{i,m}^+ \otimes 1\ \mathrm{if}\ \langle\mu,\alpha_i\rangle \leq m < 0, \quad \Delta_{\mu,\nu}(x_{i,n}^-) = 1 \otimes x_{i,n}^-\ \mathrm{if}\ \langle\nu,\alpha_i\rangle < n \leq 0,
\end{gather*}
and for $\epsilon, \eta$ antidominant we have the following commutative diagrams:
\begin{gather*}
 \xymatrixcolsep{6pc} \xymatrix{
\CU_{\mu+\nu}(\Gaff) \ar[d]_{\iota_{\epsilon,\eta}^{\mu+\nu}} \ar[r]^{\Delta_{\mu,\nu}} & \CU_{\mu}(\Gaff) \otimes \CU_{\nu}(\Gaff) \ar[d]_{\iota_{\epsilon,0}^{\mu} \otimes \iota_{0,\eta}^{\nu}} \\
\CU_{\mu+\nu+\epsilon+\eta}(\Gaff) \ar[r]^{\Delta_{\mu+\epsilon,\nu+\eta}}          & \CU_{\mu+\epsilon}(\Gaff) \otimes \CU_{\nu+\eta}(\Gaff), }  \\
\xymatrixcolsep{6pc} \xymatrix{
\Borel \ar[d]^{\jmath_{\epsilon+\eta}} \ar[r]^{\Delta} & \Borel \otimes \Borel \ar[d]^{\jmath_{\epsilon} \otimes \jmath_{\eta}} \\
\CU_{\epsilon+\eta}(\Gaff) \ar[r]^{\Delta_{\epsilon,\eta}}          & \CU_{\epsilon}(\Gaff) \otimes \CU_{\eta}(\Gaff), }   \qquad
 \xymatrixcolsep{6pc} \xymatrix{
\lBorel \ar[d]^{\jmath_{\epsilon,\eta}} \ar[r]^{\Delta} & \lBorel \otimes \lBorel \ar[d]^{\jmath_{\epsilon,0} \otimes \jmath_{0,\eta}} \\
\CU_{\epsilon+\eta}(\Gaff) \ar[r]^{\Delta_{\epsilon,\eta}}          & \CU_{\epsilon}(\Gaff) \otimes \CU_{\eta}(\Gaff), } \\
  \xymatrixcolsep{6pc} \xymatrix{
\CU_{\mu+\epsilon+\nu}(\Gaff) \ar[d]^{\Delta_{\mu,\epsilon+\nu}} \ar[r]^{\Delta_{\mu+\epsilon,\nu}} & \CU_{\mu+\epsilon}(\Gaff) \otimes \CU_{\nu}(\Gaff) \ar[d]^{\Delta_{\mu,\epsilon} \otimes \mathrm{Id}} \\
\CU_{\mu}(\Gaff) \otimes \CU_{\epsilon+\nu}(\Gaff) \ar[r]^{\mathrm{Id} \otimes \Delta_{\epsilon,\nu}}          & \CU_{\mu}(\Gaff) \otimes \CU_{\epsilon}(\Gaff) \otimes \CU_{\nu}(\Gaff).} 
\end{gather*} 
\end{theorem}
Our $\Delta_{\mu,\nu}$ is opposite to $\sigma \circ \Delta_{\nu,\mu}$ in \cite{FT} in order to be compatible with the coproduct of the quantum loop algebra $\qaf$ in Eq.\eqref{def: coproduct DJ}. 

Recall from Eq.\eqref{def: Theta series} the Theta series $\Theta_{\Bp}(w)$, which is a power series in $w$ with coefficients in $U_q^-(L\Glie)\otimes U_q^+(L\Glie)$. Given two coweights $\mu$ and $\nu$, via the natural identifications $U_q^{\pm}(L\Glie) = \CU_{\mu}^{\pm}(\Gaff)$, the Theta series $\Theta_{\Bp}(w)$ corresponds to a power series with coefficients in $\CU_{\mu}^-(\Gaff) \otimes \CU_{\nu}^+(\Gaff)$, denoted by $\Theta_{\Bp}^{\mu,\nu}(w)$. 

\begin{lem} 
Given $\Bp \in \mathfrak{d}$ and two coweights $\mu$ and $\nu$, we have
$$ \Delta_{\mu,\nu}(T_{\Bp}^{\mu+\nu}(w)) = (1\otimes T_{\Bp}^{\nu}(w)) \times \Theta_{\Bp}^{\mu,\nu}(w) \times (T_{\Bp}^{\mu}(w) \otimes 1). $$
\end{lem}
\begin{proof}
By zigzag arguments as in the proof of \cite[Corollary 10.12]{FT} we are reduced to the case $\mu = \nu = 0$, which follows by applying $\jmath_{0,0} \otimes \jmath_{0,0}$ to Eq.\eqref{def: Theta series} and using the the third commutative diagram of Theorem \ref{thm: coproduct A}.
\end{proof}
\begin{lem} \label{lem: coproduct estimation}
For $\mu, \nu$ coweights, $i \in I$ and $m \in \BZ$ we have:
\begin{align*}
\Delta_{\mu,\nu}(\phi_{i,m}^{\pm})& \equiv \sum_{n\in \BZ} \phi_{i,n}^{\pm}\otimes \phi_{i,m-n}^{\pm}  \ \mathrm{mod}. \sum_{\beta \in \BQ_>} \CU_{\mu}(\Gaff)_{-\beta} \otimes \CU_{\nu}(\Gaff)_{\beta}, \\
\Delta_{\mu,\nu}(x_{i,m}^+) &\equiv x_{i,m}^+ \otimes 1 + \sum_{s\geq 0} \phi_{i,m-s}^+ \otimes x_{i,s}^+ + \sum_{t<0} \phi_{i,m-t}^- \otimes x_{i,t}^+ \\
&\qquad \qquad  \mathrm{mod}. \sum_{\beta \in \BQ_>} \CU_{\mu}(\Gaff)_{-\beta} \otimes \CU_{\nu}(\Gaff)_{\beta+\alpha_i}, \\ 
\Delta_{\mu,\nu}(x_{i,m}^-) &\equiv 1 \otimes x_{i,m}^- + \sum_{s>0} x_{i,s}^- \otimes \phi_{i,m-s}^+ + \sum_{t\leq 0} x_{i,t}^- \otimes \phi_{i,m-t}^-   \\
&\qquad \qquad \mathrm{mod}. \sum_{\beta \in \BQ_>} \CU_{\mu}(\Gaff)_{-\beta-\alpha_i} \otimes \CU_{\nu}(\Gaff)_{\beta}.
\end{align*}
\end{lem}
\begin{proof}
Adapt the zigzag arguments in the proof of \cite[Theorem 4.12]{coproduct} to reduce to the case $\mu = \nu = 0$; see \cite[Lemma 2.5]{HZ} for a closer situation. In view of the second and third commutative diagrams of Theorem \ref{thm: coproduct A}  it suffices to prove the above coproduct estimations for the quantum loop algebra. The first relation follows from Proposition \ref{prop: twist}. The second and third were claimed without proof in \cite[Proposition 1.2]{Chari}. A proof can be found in \cite[Proposition 3.4]{Wang} which works in all types. 
\end{proof}

Recall the quotient map $\pi_{\mu}: \CU_{\mu}(\Gaff) \longrightarrow \BCU_{\mu}(\Gaff)$ before Lemma \ref{lem: tensor trivial}. The algebra homomorphisms of Lemma \ref{lem: tensor trivial}, initially obtained from Drinfeld formal coproduct, also follow from Drinfeld--Jimbo coproduct by Lemma \ref{lem: coproduct estimation}:
$$ F_{\mu, \nu} = (\pi_{\mu} \otimes \mathrm{Id}) \Delta_{\mu,\nu},\quad G_{\mu, \nu} = (\mathrm{Id} \otimes \pi_{\nu}) \Delta_{\mu,\nu}.  $$
This means that in defining the module structure on the tensor product of a module and a trivial module, the two coproducts of Lemma \ref{lem: tensor trivial} and Theorem \ref{thm: coproduct A} are equivalent.
The next result is proved in the same way as Theorem \ref{thm: trivial associativity Yangian}.

\begin{theorem} 
Given $\mu, \zeta$ and $\nu$ coweights, we have as algebra homomorphisms from $\CU_{\mu+\zeta+\nu}(\Gaff)$ to $\CU_{\mu}(\Gaff)\otimes \CU_{\zeta}(\Gaff) \otimes \BCU_{\nu}(\Gaff)$ and to $\BCU_{\mu}(\Gaff) \otimes \CU_{\zeta}(\Gaff) \otimes \CU_{\nu}(\Gaff)$ respectively, 
\begin{gather*}
(\mathrm{Id} \otimes \mathrm{Id} \otimes \pi_{\nu}) (\mathrm{Id} \otimes \Delta_{\zeta,\nu})  \Delta_{\mu,\zeta+\nu} = (\mathrm{Id} \otimes \mathrm{Id} \otimes \pi_{\nu}) (\Delta_{\mu,\zeta} \otimes \mathrm{Id}) \Delta_{\mu+\zeta,\nu},  \\
(\pi_{\mu} \otimes \mathrm{Id} \otimes \mathrm{Id}) (\mathrm{Id} \otimes \Delta_{\zeta,\nu})  \Delta_{\mu,\zeta+\nu} = (\pi_{\mu} \otimes \mathrm{Id} \otimes \mathrm{Id}) (\Delta_{\mu,\zeta} \otimes \mathrm{Id}) \Delta_{\mu+\zeta,\nu}.  
\end{gather*}
\end{theorem}
Let $M$ be a $\CU_{\mu}(\Gaff)$-module, $N$ be a $\CU_{\nu}(\Gaff)$-module
and $K$ be a $\CU_{\zeta}(\Gaff)$-module. If either $M$ or $N$ is a trivial module then the above theorem implies that the identity map is a $\CU_{\mu+\zeta+\nu}(\Gaff)$-module isomorphism $(M \otimes K) \otimes N \longrightarrow M \otimes (K \otimes N)$. 
We now consider the case when $K = L(\Bp)_{w^{-1}}$ is trivial with $\Bp \in \mathfrak{d}_{\zeta}$ a polynomial $\ell$-weight. 

Consider the $\CU_{\mu+\zeta+\nu}(\Gaff)$-module $M \otimes (L_{\zeta}(\Bp)_{w^{-1}} \otimes N)$. As in Example \ref{example: tensor prefund}, it is the pullback of $M\otimes N[w,w^{-1}]$, viewed as a $\CU_{\mu}(\Glie) \otimes \CU_{\nu}(\Gaff)[w,w^{-1}]$-module by scalar extension, along the algebra homomorphism 
 $$  (\mathrm{Id} \otimes F^{\zeta+\nu}_{\Bp,w}) \Delta_{\mu,\zeta+\nu}: \CU_{\mu+\zeta+\nu}(\Gaff) \longrightarrow \CU_{\mu}(\Gaff) \otimes \CU_{\nu}(\Gaff)[w,w^{-1}]. $$
Replacing the underlying space $M\otimes N[w,w^{-1}]$ with $(M \otimes N)((w))$, we get a completed triple tensor product module over $\CU_{\mu+\zeta+\nu}(\Gaff)$, denoted by $\overline{M \otimes (L_{\zeta}(\Bp)_{w^{-1}} \otimes N)}$. In the same way, the other triple tensor product modules are completed: 
\begin{gather*}
\overline{(M\otimes L_{\zeta}(\Bp)_{w^{-1}}) \otimes N}, \quad \overline{(M\otimes N) \otimes L_{\zeta}(\Bp)_{w^{-1}}}, \quad \overline{M \otimes (N \otimes L_{\zeta}(\Bp)_{w^{-1}})}, \\
\overline{(L_{\zeta}(\Bp)_{w^{-1}} \otimes M)\otimes N}, \quad \overline{L_{\zeta}(\Bp)_{w^{-1}} \otimes (M \otimes N)}.
\end{gather*}

The next result is proved in the same way as Theorem \ref{thm:Yangian associator} and Corollary \ref{cor: nontrivial associator Yangian} by diagram chasing, based on the previous results of this section and Proposition \ref{prop: one-dim R}.
\begin{theorem}  
Let $M$ and $N$ be modules over $\CU_{\mu}(\Gaff)$ and $\CU_{\nu}(\Gaff)$ respectively. For $\zeta$ a coweight and $\Bp \in \mathfrak{d}_{\zeta}$, we have an isomorphism of $\CU_{\mu+\zeta+\nu}(\Gaff)$-modules induced by the action of $\Theta_{\Bp}^{\mu,\nu}(w) \in (\CU_{\mu}(\Gaff) \otimes \CU_{\nu}(\Gaff))[[w]]$ on $M \otimes N$,
$$ \Theta_{\Bp}^{\mu,\nu}(w)|_{M\otimes N}:  \overline{(M \otimes L_{\zeta}(\Bp)_{w^{-1}}) \otimes N} \longrightarrow \overline{M \otimes (L_{\zeta}(\Bp)_{w^{-1}} \otimes N)}.   $$
If either $M$ is bottom graded or $N$ is top graded, then $\Theta_{\Bp}^{\mu,\nu}(w)|_{M\otimes N}$ restricts to a module isomorphism from $(M \otimes L_{\zeta}(\Bp)_{w^{-1}}) \otimes N$ to $M \otimes (L_{\zeta}(\Bp)_{w^{-1}} \otimes N)$.
\end{theorem}

\appendix

\section{Factorizations of asymptotic representations}  \label{app: asym}
In this appendix we identify the asymptotic representations for the Yangian \cite{Z2} and for the quantum loop algebra \cite{Z1} with tensor product modules over shifted Yangians and shifted quantum affine algebras. Fix a Dynkin node $i \in I$.

\subsection{The Yangian case} In \cite{Z2} we constructed two families of representations via a limit procedure of Hernandez--Jimbo \cite{HJ}: the {\it asymptotic representations} $\mathscr{L}(\frac{\Psi_{i,y}}{\Psi_{i,0}})$ for $y \in \BC$ defined over the ordinary Yangian $Y_0(\Glie)$; the {\it negative prefundamental representation} $L_{i,0}^-$   defined over the shifted Yangian $Y_{-\varpi_i^{\vee}}(\Glie)$. The proof of \cite[Proposition 3.15]{HZ} identified $L_{i,0}^-$ with the irreducible module $L(\Psi_{i,0}^{-1})$. In this subsection, we prove that $\mathscr{L}(\frac{\Psi_{i,y}}{\Psi_{i,0}})$ is the tensor product module $L(\Psi_{i,0}^{-1}) \otimes L(\Psi_{i,y})$.

Define $\mathcal{Y}_i$ to be the subset of $Y_0(\Glie)$  whose elements are 
$$ x_{i,n}^+,\ x_{j,n}^{\pm},\ \xi_{j,n} \quad \mathrm{for}\ n \in \BN,\   j \in I \setminus \{i\}.  $$
By abuse of language view $\mathcal{Y}_i$ also as a subset of $Y_{-\varpi_i^{\vee}}(\Glie)$. Fix $y \in \BC$.

For $n \in \BN$, let $W_n$ denote the finite-dimensional irreducible $Y_0(\Glie)$-module of highest $\ell$-weight $\Psi_{i,0}^{-1} \Psi_{i,-nd_i}$ with $\rho_n: Y_0(\Glie) \longrightarrow \mathrm{End}(W_n)$ the structural map.  By \cite[\S 3.2]{Z2} we have an inductive system of vector spaces $F_{m,n}: W_n \longrightarrow W_m$ for $m > n$ with two properties: the $F_{m,n}$ commute with the actions of $\mathcal{Y}_i$; for $n \in \BN$ there exists a unique linear map $f_n: W_n \longrightarrow W_{n+1}$ such that for $m > n+1$ we have: 
\begin{align*}
\rho_m(\xi_i(z)) F_{m,n} &= F_{m,n} \frac{z+md_i}{z+nd_i} \rho_n(\xi_i(z)), \\
\rho_m(x_i^-(z)) F_{m,n} &= F_{m,n+1} (\frac{z+md_i}{z+nd_i} F_{n+1,n} \rho_n(x_i^-(z)) + \frac{md_i-nd_i}{z+nd_i} f_n).
\end{align*}
Our $f_n$ corresponds to $\frac{1}{d_i} \mathscr{F}_{n+1,n}(-\otimes x_{i,0}^- \omega_{n,n+1})$ in the notations of \cite[Lemma 15]{Z2}.
Let $W_{\infty}$ be the inductive limit of the $W_n$. It is the underlying space of the asymptotic representation $\mathscr{L}( \frac{\Psi_{i,y}}{\Psi_{i,0}})$ and the negative prefundamental representation $L(\Psi_{i,0}^{-1})$. Their structural maps $\rho^y: Y_0(\Glie) \longrightarrow \mathrm{End}(W_{\infty})$ and $\rho_{\infty}: Y_{-\varpi_i^{\vee}}(\Glie) \longrightarrow \mathrm{End}(W_{\infty})$ are defined as follows (\cite[Definition 16, Proposition 23]{Z2}):
\begin{gather*}
\rho^y(X) = \rho_{\infty}(X) = \lim_{n \rightarrow \infty} \rho_n(X) \quad \mathrm{for}\ X \in \mathcal{Y}_i, \\
\rho^y(\xi_i(z)) = \lim_{n \rightarrow \infty}   \frac{z-y}{z+nd_i} \rho_n(\xi_i(z)), \quad \rho_{\infty}(\xi_i(z)) = \lim_{n \rightarrow \infty}   \frac{1}{z+nd_i} \rho_n(\xi_i(z)), \\
\rho^y(x_i^-(z)) = \lim_{n\rightarrow \infty} (\frac{z-y}{z+nd_i} F_{n+1,n} \rho_n(x_i^-(z)) + \frac{-y-nd_i}{z+nd_i} f_n), \\
\rho_{\infty}(x_i^-(z)) = \lim_{n\rightarrow \infty} (\frac{1}{z+nd_i} F_{n+1,n} \rho_n(x_i^-(z)) + \frac{1}{z+nd_i} f_n).
\end{gather*}
Here $\lim\limits_{n\rightarrow \infty}$ means the inductive limit of a morphism of inductive systems; see \cite[\S 3.1]{Z2}. Naively, $\rho^y$ is obtained from $\rho_m$ for $m$ big enough by replacing $md_i$ with $-y$, and $\rho_{\infty}$ is obtained from $\rho_m$ by first dividing the actions of $x_i^-(z)$ and $\xi_i(z)$ by $md_i$ and then replacing $md_i$ with zero. It is clear from the  above definition that
\begin{gather*}
    \rho^y(\xi_i(z)) = (z-y) \rho_{\infty}(\xi_i(z)), \\
    \rho^y(x_i^-(z)) - (z-y) \rho_{\infty}(x_i^-(z)) =  -\lim_{n\rightarrow \infty}  f_n = -\rho_{\infty}(x_{i,0}^-).
\end{gather*}
Comparing with Example \ref{example: tensor prefund Yangian}, we have as $Y_0(\Glie)$-modules
\begin{equation}  \label{asym: Yangian}
\mathscr{L}(\frac{\Psi_{i,y}}{\Psi_{i,0}}) \cong L(\Psi_{i,0}^{-1}) \otimes L(\Psi_{i,y}). 
\end{equation}

\subsection{The quantum loop case} For $a \in \BC^{\times}$ let $\Phi_{i,a} \in \mathfrak{t}_{\ell}^*$ be the polynomial $\ell$-weight whose $j$th component is $(a - z a^{-1})^{\delta_{ij}}$  for $j\in I$. In particular, $\Phi_{i,1} = \Psi_{i,1}$.  

 In \cite{HJ} Hernandez--Jimbo constructed the negative prefundamental representation $L(\Phi_{i,1}^{-1})$ of the upper Borel subalgebra via a limit procedure. Their limit procedure was modified in \cite{Z1} to get asymptotic representations $\mathscr{L}(\frac{\Phi_{i,c}}{\Phi_{i,1}})$, for $c \in \BC^{\times}$, of the quantum loop algebra. In this subsection we determine the $\CU_{-\varpi_i^{\vee}}(\Gaff)$-module structure on $L(\Phi_{i,1}^{-1}) = L_{-\varpi_i^{\vee}}(\Phi_{i,1}^{-1})$ and use it to reconstruct the asymptotic representations.

Let $\mathcal{U}_i$ denote the subset of $\qaf$ with elements
$$ x_{i,s}^+,\quad x_{j,s}^{\pm},\quad \phi_{j,s}^{\pm}\quad \textrm{for $s \in \BZ$ and $j \in I \setminus \{i\}$.} $$
View $\mathcal{U}_i$ also as a subset of $\CU_{-\varpi_i^{\vee}}(\Gaff)$. Fix $c \in \BC^{\times}$. 

For $n \in \BN$, let $V_n$ denote the finite-dimensional irreducible $\qaf$-module of highest $\ell$-weight $\Phi_{i,1}^{-1} \Phi_{i,q_i^n}$ with $\rho_n: \qaf \longrightarrow \mathrm{End} (V_n)$ the structural map. By \cite[\S 4.2]{HJ} we have an inductive system of vector spaces $F_{m,n}: V_n \longrightarrow V_m$ for $m > n$ with two properties: the $F_{m,n}$ commute with the actions of $\mathcal{U}_i$; for $n \in \BN$ there exists a unique linear map $f_n: V_n \longrightarrow V_{n+1}$ such that for $m > n+1$ we have: 
\begin{align*}
\rho_m(\phi_i^{\pm}(z)) F_{m,n} &= F_{m,n} \frac{q_i^m - z q_i^{-m}}{q_i^n - z q_i^{-n}} \rho_n(\phi_i^{\pm}(z)), \\
\rho_m(x_{i,0}^-) F_{m,n} &= F_{m,n+1} (q_i^{n-m} F_{n+1,n} \rho_n(x_{i,0}^-) + \frac{q_i^{m-n} - q_i^{n-m}}{q_i-q_i^{-1}} f_n), \\
\rho_m(x_{i,1}^-) F_{m,n} &= F_{m,n+1} (q_i^{m-n} F_{n+1,n} \rho_n(x_{i,1}^-) + \frac{q_i^{m-3n} - q_i^{-n-m}}{q_i-q_i^{-1}} f_n).
\end{align*}
The first equation is \cite[Proposition 4.2]{HJ}. Let us sketch a proof to the last two equations. Recall from \cite[Appendix A]{Z1} the module morphism $\mathscr{F}_{m,n}: V_n \otimes Z_{nm} \longrightarrow V_m$ where $Z_{nm}$ is a $\qaf$-module with a highest $\ell$-weight vector $v_{nm}$. The last two equations are proved in the same way as \cite[Lemma 15]{Z2} based on the definitions
$$\mathscr{F}_{m,n}(v \otimes v_{nm}) =: F_{m,n}(v), \quad  \mathscr{F}_{n+1,n}(v\otimes x_{i,0}^- v_{n,n+1}) =: f_n(v) $$ 
for $v \in V_n$ and the following key formulas:
\begin{gather*}
     x_{i,0}^-(v \otimes v_{nm}) = v \otimes  x_{i,0}^- v_{nm} + x_{i,0}^- v \otimes q_i^{n-m} v_{nm} \in V_n \otimes Z_{nm}, \\ 
     x_{i,1}^-(v \otimes v_{nm}) = v \otimes q_i^{-2n} x_{i,0}^- v_{nm} + x_{i,1}^- v \otimes q_i^{m-n} v_{nm} \in V_n \otimes Z_{nm}, \\
     \mathscr{F}_{m,n}(v \otimes x_{i,0}^- v_{nm}) = [m-n]_{q_i} F_{m,n+1} f_n(v) \in V_m.
\end{gather*}
Let $V_{\infty}$ be the inductive limit of the $V_n$. It is the underlying space of the asymptotic representation $\mathscr{L}( \frac{\Phi_{i,c}}{\Phi_{i,1}})$ and the negative prefundamental representation $L(\Phi_{i,1}^{-1})$. Their structural maps $\rho^c: \qaf \longrightarrow \mathrm{End}(V_{\infty})$ and $\rho_{\infty}: \CU_{-\varpi_i^{\vee}}(\Gaff) \longrightarrow \mathrm{End}(V_{\infty})$ are defined as follows (\cite[comments after Remark 4.6]{HJ} and \cite[Appendix A]{Z1}):
\begin{gather*}
\rho^c(X) = \rho_{\infty}(X) = \lim_{n \rightarrow \infty} \rho_n(X) \quad \mathrm{for}\ X \in \mathcal{U}_i, \\
\rho^c(\phi_i^{\pm}(z)) = \lim_{n \rightarrow \infty}   \frac{c-zc^{-1}}{q_i^n - z q_i^{-n}} \rho_n(\phi_i^{\pm}(z)), \quad \rho_{\infty}(\phi_i^{\pm}(z)) = \lim_{n \rightarrow \infty}   \frac{1}{q_i^n - z q_i^{-n}} \rho_n(\phi_i^{\pm}(z)), \\
\rho^c(x_{i,1}^-) = \lim_{n\rightarrow \infty} ( c q_i^{-n} F_{n+1,n} \rho_n(x_{i,1}^-) + \frac{c q_i^{-3n} - c^{-1} q_i^{-n}}{q_i-q_i^{-1}} f_n), \\
\rho_{\infty}(x_{i,0}^-) = \lim_{n\rightarrow \infty}  \frac{q_i^{-n}}{q_i-q_i^{-1}} f_n,\quad \rho_{\infty}(x_{i,1}^-) = \lim_{n\rightarrow \infty} (q_i^{-n}F_{n+1,n} \rho_n(x_{i,1}^-) + \frac{q_i^{-3n}}{q_i-q_i^{-1}} f_n).
\end{gather*}
Naively,  $\rho^c$ is obtained from $\rho_m$ for $m$ big enough by replacing $q_i^{\pm m}$ with $c^{\pm 1}$, and $\rho_{\infty}$ is obtained from $\rho_m$ by first dividing the actions of $x_{i,s}^-$ and $\phi_{i,s}^{\pm}$ by $q_i^m$ and then replacing $q_i^{-2m}$ with zero. It is clear from the above definition that
$$ \rho^c(x_{i,1}^-) = \rho_{\infty}(c x_{i,1}^- - c^{-1} x_{i,0}^-),\quad \rho^c(\phi_i^{\pm}(z)) = \rho_{\infty}((c - z c^{-1}) \phi_i^{\pm}(z)). $$
By Lemma \ref{lem: tensor trivial} we have as $\CU_0(\Gaff)$-modules 
\begin{equation}  \label{asym quantum}
\mathscr{L}(\frac{\Phi_{i,c}}{\Phi_{i,1}}) \cong L(\Phi_{i,1}^{-1})  \otimes L_{\varpi_i^{\vee}}(\Phi_{i,c}).
\end{equation}

\section{Proof of Proposition \ref{prop: purity Yangian} }  \label{app: coefficients}
In this appendix we proof Proposition \ref{prop: purity Yangian} which asserts that the Yangian Theta series $\Theta_{\Bp}^{\mu,\nu}(w)$ lie in the subalgebra $Y_{\mu}^+(\Glie) \otimes Y_{\nu}^-(\Glie)$.

We need the power series $\CR_{\beta}^-(w) \in Y_0^-(\Glie)_{-\beta} \otimes Y_0^+(\Glie)_{\beta})[[w^{-1}]]$ for $\beta \in \BQ_+$  constructed in \cite[\S 4]{GTLW}; initially $\CR_0^-(w) = 1$. The sum of all these power series, denoted by $\CR^-(w)$, is an invertible element in the completed tensor product $Y_0^-(\Glie) \otimes_w Y_0^+(\Glie)$; here we extend Definition \ref{defi: second completion tensor product} to $\BQ$-graded vector spaces in the obvious way.
\begin{theorem}\cite[Theorem 4.1]{GTLW} 
For $\Bp$ a polynomial $\ell$-weight, we have the following identity in the algebra $(Y_0(\Glie) \otimes_w Y_0(\Glie))[[z^{-1}]]$:
\begin{equation}
 (\tau_w \otimes \mathrm{Id})\Delta(S_{\Bp}(z)) = \CR^-(w)^{-1} (\tau_w(S_{\Bp}(z)) \otimes S_{\Bp}(z)) \CR^-(w). \label{equ: twistor Yangian}
\end{equation}
\end{theorem}
\begin{proof}
Recall from the proof of Lemma \ref{lem: coproduct S} the elements $s_{\Bp,m}$ and $t_{i,m}$, which arise as logarithms of $S_{\Bp}(z)$ and $\xi_i(z) = \overline{\xi}_i(z)$. By \cite[Theorem 4.1]{GTLW},
$$ (\tau_w \otimes \mathrm{Id})\Delta(t_{i,m}) = \CR^-(w)^{-1} (\tau_w(t_{i,m}) \otimes 1 + 1 \otimes t_{i,m}) \CR^-(w).  $$
Since the $s_{\Bp,m}$ are linear combinations of the $t_{i,m}$, the same equation holds with $t_{i,m}$ replaced by $s_{\Bp,m}$. Taking exponentials gives the desired identity.
\end{proof}
The algebra automorphism $\tau_w^{-1} \otimes \mathrm{Id}$ of $Y_0(\Glie) \otimes Y_0(\Glie)[w]$ does not extend to the completion $Y_0(\Glie) \otimes_w Y_0(\Glie)$. It can not be applied to Eq.\eqref{equ: twistor Yangian} to simplify the equation.
\begin{lem}  \label{lem: purity}
Let $\Bp$ be a polynomial $\ell$-weight of coweight $\zeta$ and $\beta \in \BQ_+$.
\begin{itemize}
    \item[(i)] For $x \in Y_0^+(\Glie)_{\beta}$, we have $S_{\Bp}(z)^{-1} x S_{\Bp}(z) \in Y_0^+(\Glie)_{\beta}[[z^{-1}]]$.
    \item[(ii)] For $y \in Y_0^-(\Glie)_{-\beta}$, we have $\tau_w(S_{\Bp}(z)) y \tau_w(S_{\Bp}(z))^{-1} \in Y_0^-(\Glie)_{-\beta}[w][[z^{-1}]]$. All the $z$-coefficients are polynomials in $w$ of degree bounded by $\langle \zeta, \beta\rangle$.
\end{itemize}
\end{lem}
\begin{proof}
(i) follows from Eq.\eqref{comm: S x general}. For (ii), by multiplicativity  one may assume $\Bp = \Psi_{i,0}$ and  $y = x_{j,n}^-$ for $i, j \in I$ and $n \in \BN$. If $i \neq j$, then $\tau_w(S_i(z)) x_{j,n}^- \tau_w(S_i(z))^{-1} = x_{j,n}^-$ is constant. If $i = j$, again by Eq.\eqref{comm: S x general},  
\begin{align*}
&\quad \tau_w (S_i(z)) x_i^-(u) \tau_w(S_i(z))^{-1} = \tau_w (S_i(z) \tau_w^{-1}(x_i^-(u)) S_i(z)^{-1}) \\
&= \tau_w (S_i(z)  x_i^-(u+w) S_i(z)^{-1}) = \tau_w((1-(u+w)z^{-1}) x_i^-(u+w) + z^{-1} x_{i,0}^-) \\
&= (1-(u+w) z^{-1}) x_i^-(u) + z^{-1} x_{i,0}^-, \\
&\tau_w(S_i(z)) x_{i,n}^- \tau_w(S_i(z))^{-1} = (1-wz^{-1}) x_{i,n}^- - z^{-1} x_{i,n+1}^-.
\end{align*}
In particular all the $z$-coefficients are polynomials in $w$ of degree $\leq 1$.
\end{proof}

\noindent {\bf Proof of Proposition \ref{prop: purity Yangian}.}
As in Step 4 of the proof of Theorem \ref{thm: Yangian poly Theta} it suffices to prove the statement for the Omega series in the case $\mu = \nu = 0$.
 By Lemma \ref{lem: purity}, the following two elements of the algebra $(Y_0(\Glie) \otimes_w Y_0(\Glie))[[z^{-1}]]$, denoted by $f(z,w)$ and $g(z,w)$ respectively, belong to the subalgebra $(Y_0^-(\Glie) \otimes_w Y_0^+(\Glie))[[z^{-1}]]$:
$$ (1\otimes S_{\Bp}(z)^{-1}) \CR^-(w)^{-1} (1\otimes S_{\Bp}(z)),\quad (\tau_w(S_{\Bp}(z)) \otimes 1) \CR^-(w) (\tau_w(S_{\Bp}(z))^{-1} \otimes 1). $$
By Eq.\eqref{equ: twistor Yangian} we have a factorization in the algebra $(Y_0(\Glie) \otimes_w Y_0(\Glie))[[z^{-1}]]$:
$$ (\tau_w \otimes \mathrm{Id})\Delta(S_{\Bp}(z)) = (1\otimes S_{\Bp}(z)) f(z,w)g(z,w) (\tau_w(S_{\Bp}(z)) \otimes 1). $$
On the other hand Definition \ref{def: Theta series Yangian}  gives another factorization of $(\tau_w \otimes \mathrm{Id})\Delta(S_{\Bp}(z))$. After cancelling the common factors $\tau_w(S_{\Bp}(z)) \otimes 1$ on the right and $1\otimes S_{\Bp}(z)$ on the left, which are all invertible in the algebra $(Y_0(\Glie) \otimes_w Y_0(\Glie))[[z^{-1}]]$, we obtain 
$$ (Y_0(\Glie)^{\otimes 2}[w])[[z^{-1}]] \ni (\tau_w \otimes \mathrm{Id}) (\Omega_{\Bp}(z)) = f(z,w) g(z,w) \in (Y_0^-(\Glie) \otimes_w Y_0^+(\Glie))[[z^{-1}]].  $$
It follows that $(\tau_w \otimes \mathrm{Id}) (\Omega_{\Bp}(z)) \in (Y_0^-(\Glie) \otimes Y_0^+(\Glie)[w])[[z^{-1}]]$. From the injectivity of $\tau_w: Y_0(\Glie) \longrightarrow Y_0(\Glie)[w]$ we conclude that $\Omega_{\Bp}(z) \in (Y_0^-(\Glie) \otimes Y_0^+(\Glie))[[z^{-1}]]$.
\hfill $\Box$

\end{document}